\documentclass[10pt]{article}
\usepackage{schola-article}

\title{Proof Theory for Bimodal Provability Logics}
\author{Borja Sierra Miranda\and Thomas Studer}

\date{\small Logic and Theory Group\\
University of Bern\\
Bern, Switzerland\\
\texttt{\{borja.sierra,thomas.studer\}@unibe.ch}}

\usepackage{modalops}
\usepackage{quiver}
\usepackage{multicol}

\newcommand\union{\cup}
\newcommand\Union{\bigcup}
\newcommand\inter{\cap}

\newcommand\set[1]{\left\{#1\right\}}

\DeclareMathOperator\domain{Dom}

\newcommand\function[1][]{\xlongrightarrow{#1}}

\newcommand\blan{\mathcal{L}_{\nec_{\leq 1}}}
\newcommand\var{\mathrm{Var}}
\newcommand\wff{\mathrm{Fm}}

\newcommand\K[1][]{\mathsf{K}_{#1}}
\newcommand\GL{\mathsf{GL}}
\newcommand\GLS{\mathsf{GLS}}
\newcommand\PA{\mathsf{PA}}
\newcommand\Grz{\mathsf{Grz}}
\newcommand\wGrz{\mathsf{wGrz}}
\newcommand\KT{\mathsf{K4}}

\newcommand\g[1]{\mathcal{G}#1}
\newcommand\n[1]{\mathcal{G}^\infty#1}

\newcommand\CS{\mathsf{CS}}
\newcommand\CSM{\mathsf{CSM}}
\newcommand\ER{\mathsf{ER}}
\newcommand\KTCS{\mathsf{K4CS}}
\newcommand\KTCSM{\mathsf{K4CSM}}
\newcommand\KTER{\mathsf{K4ER}}

\newcommand\Kax[1][]{\mathrm{K}_{#1}}
\newcommand\Lax[1][]{\mathrm{L}_{#1}}
\newcommand\Cax[1][]{\mathrm{C}_{#1}}
\newcommand\Max[1][]{\mathrm{M}_{#1}}
\newcommand\ERax[1][]{\mathrm{ER}_{#1}}

\newcommand\MP{\mathrm{MP}}
\newcommand\NEC[1][]{\mathrm{Nec}_{#1}}

\newcommand\ax{\mathrm{ax}}
\newcommand\Ax{\mathrm{Ax}}
\newcommand\botL{{\bot}\mathrm{L}}
\newcommand\botR{{\bot}\mathrm{R}}
\newcommand\toL{{\to}\mathrm{L}}
\newcommand\toR{{\to}\mathrm{R}}
\newcommand\negL{{\neg}\mathrm{L}}
\newcommand\negR{{\neg}\mathrm{R}}
\newcommand\veeL{{\vee}\mathrm{L}}
\newcommand\veeR{{\vee}\mathrm{R}}
\newcommand\wedgeL{{\wedge}\mathrm{L}}
\newcommand\wedgeR{{\wedge}\mathrm{R}}

\newcommand\modal[2][]{\nec^{#2}_{#1}}

\newcommand\lob[2][]{\mathrm{L}\ddot{\mathrm{o}}\mathrm{b}^{#2}_{#1}}

\newcommand\wk{\mathrm{Wk}}
\newcommand\inv[1]{\mathrm{Inv}_{#1}}
\newcommand\ctr{\mathrm{Ctr}}
\newcommand\cut{\mathrm{Cut}}

\newcommand\rep{\mathrm{Rep}}
\newcommand\emp{\mathrm{Emp}}

\newcommand\voc{\mathrm{Voc}}

\newcommand\lhg{\mathrm{lhg}}
\newcommand\lrul{\mathrm{lRul}}
\newcommand\hg{\mathrm{hg}}
\newcommand\sub{\mathrm{Sub}}

\newcommand\satc[1]{|#1|_{\mathrm{Sat}}}
\newcommand\lgth{\ell}

\makeatletter
\newcommand{\todo}[1]{\marginpar{\textbf{TODO\footnotemark}}\@latex@warning{TODO: #1}\footnotetext{ #1}}
\makeatother
\begin{document}

\maketitle

\begin{abstract}
  \noindent
  We provide the first (non-labelled) sequent calculi for bimodal provability logics with 'usual' provability predicates. In particular, we introduce calculi for the logics \(\CS\), \(\CSM\) and \(\ER\).
  Additionally, we present non-wellfounded versions of our calculi, and use them to establish a cut-elimination procedure.
  Finally, we prove the first interpolation results for these logics showing that they  all  enjoy the uniform Lyndon interpolation property.
\end{abstract}

\section{Introduction}

Provability logic is the area of modal logic that studies the notion of proof through the lense of modal logic, by interpreting the modal operators as \emph{provable in some mathematical theory} (usually arithmetical theories).
One of the most important results in the area is Solovay's characterization  \cite{solovay} of the provability logic of \(\PA\) (Peano Arithmetic) with respect to \(\PA\), being the logic \(\GL\), and with respect to the standard model~\(\mathbb{N}\), being the logic \(\GLS\).

A natural generalization is to consider a modal language with two modalities, \(\nec_0\) and \(\nec_1\), which will lead to bimodal provability logics.
The modality \(\nec_0\) is interpreted as provability in some theory \(T\) while \(\nec_1\) is interpreted as provability in another theory \(S\).
The bimodal language is thus capable of expressing  relationships between provability in \(T\) and provability in \(S\). 
For instance,  \(S\) is an extension of \(T\)  is expressed by the schema \(\nec_0 \phi \to \nec_1 \phi\) called monotonicity, the statement that the theory \(S\) proves every instance of soundness in \(T\) is expressed by the schema \(\nec_1(\nec_0 \phi \to \phi)\), and that \(S\) decides provability in \(T\) is expressed by \(\nec_1 \nec_0 \phi \vee \nec_1 \neg \nec_0 \phi\).
%, and hence able to formulate properties such as \(S\) is an extension of \(T\) by the schema \(\nec_0 \phi \to \nec_1 \phi\), \(S\) proves every instance of soundness in \(T\) as the schema \(\nec_1(\nec_0 \phi \to \phi)\) or \(S\) decides provability in \(T\) as \(\nec_1 \nec_0 \phi \vee \nec_1 \neg \nec_0 \phi\).

In this paper, we will develop the proof theory of three major bimodal provability logics: \(\CS\), \(\CSM\), and \(\ER\).
The logic \(\CS\) is the minimal bimodal provability logic~\cite{smorynski} (see also the discussion in~\cite{Beklemishev1994-BEKOBL}).  % \cite{Carlson1986ModalLW},
The logic \(\CSM\) is the extension of \(\CS\) with the monotoncity schema \(\nec_0 \phi \to \nec_1 \phi\) mentioned before.
%stating that everything that is \(T\)-provable is \(S\)-provable and 
The logic \(\ER\), discoverd by Carlson~\cite{Carlson1986ModalLW} and indepently by Montagna~\cite{Montagna},  is the provability logic of
sound theories  \(T\) and~\(S\)  where  \(S\) contains the local reflection principle for   \(T\). Thus  \(\ER\) is the extension of  \(\CS\) with the principle of essential reflexivity $\Box_1(\Box_0 \phi \to \phi)$.
Note that in \(\CSM\) and \(\ER\), we have that the theory \(S\) is an extension of the theory~\(T\).
The provability logic of extensions is one of the main topics in bimodal provability logic \cite{Beklemishev1994-BEKOBL, Beklemishev1996-BEKBLF}.

The essential technical tool  for our proof-theoretic study of  \(\CS\), \(\CSM\), and \(\ER\) are non-wellfounded proofs.
Proof systems of this kind have succesfully been developed to a plethora of logics with explicit fixpoints,  see, e.g.,  \cite{Brotherston, ill-founded-intui-linear-time, guillermo-CTL, thomasCK, KokkinisStuder+2016+171+192, saurin, das, masterModality} among many others.
Shamkanov~\cite{shamkanovGl} has observed that there is also a close connection between non-wellfounded proofs and the provability logic~\(\GL\),  although the fixpoints of this logic are not explicit.  Later,  he and Savateev have extended this approach to capture the logics  \(\Grz\)~\cite{shamkanovGrz} and \(\wGrz\)~\cite{shamkanovwGrz}.  
Despite these results,  almost no proof-theoretic work has been carried out for bimodal provability logics with two `usual'~\cite{JAPARIDZE1998475} provability predicates.\footnote{There are also bi- and polymodal logics, e.g., $\mathsf{GLB}$ and $\mathsf{GLP}$,  modeling stronger notions of provability like $\omega$-provability.  For some of these logics,  there are sequent systems  available,  e.g.,~\cite{ShamkanovNested}.  Note,  however,  that these polymodal logics are not extensions of $\CS$,  but in some sense orthogonal to it.} The only exception is~\cite{justus}, which introduces a non-wellfounded  labelled sequent calculus for \(\CS\).

One important application of non-wellfounded proofs is to establish interpolation properties, see, e.g., \cite{interpolation-guillermo, pdl-interpolation, converse-pdl-interpolation, uip-interpretability, gls-ulip}.  We, too,  will use our systems to show interpolation results for   \(\CS\), \(\CSM\), and \(\ER\).
Our main contributions are as follows:
\begin{enumerate}
\item We provide the first sequent calculi for 'usual' bi-modal provability logics.  Note than in the case of \(\CS\), there is a labelled calculus available. Our approach, however, does not require the use of labels and is thus syntactically pure.  Moreover, it does not only work for $\CS$,  but can be generalized to $\CSM$ and~$\ER$.
\item We provide syntactic cut-elimination results for our calculi.  Hence we obtain completeness of the cut-free systems.
\item We provide the first interpolation results for  \(\CS\), \(\CSM\), and \(\ER\).  Actually, we directly prove uniform Lyndon interpolation (see \cite{Kurahashi}) for these three logics, which is the strongest form of interpolation. 
\item Last but not least, we make the technical observation that the proof theory of \(\ER\) is rather simple compared to its semantics.
\end{enumerate}
%
%
%The connection of non-wellfounded proofs and provability logics was established by Savateev and Shamkanov that it is also an useful tool for the study of provability logics, particularly for the logics \(\GL\) \cite{shamkanovGl}, \(\Grz\) \cite{shamkanovGrz} and \(\wGrz\) \cite{shamkanovwGrz}; even if the fixpoints of these logics are not explicit.
%A
%
%
Let us elaborate on the last point.
\(\ER\) is not Kripke complete, and we have to enrich  Kripke models with a topology (thus working in the general Kripke semantics setting) to obtain a completeness result for~\(\ER\)~\cite{Visser1995-VISACO-3}.
On the proof-theoretic side, the picture is simpler: we only have to work with two types of sequents.
Sequents of the shape \(\Gamma \Rightarrow_0 \Delta\) will be interpreted as usual, i.e., \(\bigwedge \Gamma \to \bigvee \Delta\), while sequents of the shape \(\Gamma \Rightarrow_1 \Delta\) will be interpreted as \(\nec_1(\bigwedge \Gamma \to \bigvee \Delta)\).
This kind of methodology has been previously seen in sequent calculi for the provability logic \(\GLS\) \cite{boolos-GLS, lev-GLS, kushida-gls} (the provability logic of Peano arithmetic according to true arithmetic).
However, in the systems for \(\GLS\), it is impossible to go from \( \Rightarrow_1\)-sequents to \( \Rightarrow_0\)-sequents, while in our system,  both direction (from \( \Rightarrow_1\) to \( \Rightarrow_0\) and vice versa) are possible.

\paragraph{Structure of the paper.}
In Section~\ref{sec:preliminaries}, we will introduce the logics and concepts needed for the rest of the paper,  in  particular  the bimodal logics   \(\CS\), \(\CSM\), and \(\ER\),  (modal) equational systems, uniform Lyndon interpolation, and the basics of non-wellfounded (local progress) proof theory.
In Section~\ref{sec:wellfounded}, we will introduce  wellfounded (finite) sequent calculi for \(\CS\), \(\CSM\) and \(\ER\).  We will establish the correspondence of these sequent calculi with cut and the Hilbert-style axiomatizations of the logics.
In Section~\ref{sec:cut-elim}, we will introduce non-wellfounded sequent calculi for \(\CS\), \(\CSM\) and \(\ER\),  establish translations with the wellfounded calculi, and prove cut elimination for all of them.
Finally, in Section~\ref{sec:interpolation}, we will prove uniform Lyndon interpolation for \(\CS\) and \(\ER\) via interpolation templates and for \(\CSM\) via an adequate interpretation of \(\CSM\) in \(\CS\).

\section{Preliminaries}\label{sec:preliminaries}

In this section we will introduce the concepts from modal logic and proof theory needed for the paper.

\subsection{Bimodal provabiliy logic}\label{subsec:bimodal-prov-logic}

We fix an infinite countable set \(\var\) whose elements will be called \emph{propositional variables}.

\begin{definition}
  We define the \emph{language of bimodal logic} \(\blan\) by the the following grammar
  \[
    \phi ::= p \mid \bot \mid \phi \to \phi \mid \nec_0 \phi \mid \nec_1 \phi,
  \]
  where \(p \in \var\).
  The expressions of \(\blan\) are called (bimodal) formulas.
\end{definition}
As usual, we define the following abreviations:
\begin{align*}
  &\neg \phi := \phi \to \bot,
  &&\phi \vee \psi := \neg \phi \to \psi,
  &\phi \wedge \psi := \neg(\neg \phi \vee \neg \psi),
  &&\pos_i \phi := \neg \nec_i \neg \phi.
\end{align*}
Given a formula \(\phi\) we define its complexity, denoted \(|\phi|\) as
\begin{align*}
  &|p| = |\bot| = 1,
  &&|\phi_0 \to \phi_1| = |\phi_0| + |\phi_1| + 1,
  &|\nec_i \phi_0| = |\phi_0|+ 1.
\end{align*}
If \(\phi, \psi_0, \ldots, \psi_{n-1}\) are formulas and \(p_0,\ldots,p_{n-1} \in \var\), then \(\phi[\psi_0/p_0,\ldots, \psi_{n-1}/p_{n-1}]\) denotes the result of simultenously substituting every occurrence of \(p_i\) in \(\phi\) by \(\psi_i\) for \(i < n\).
The formulas of shape \(\nec_i \phi\) will be called \(\nec\)-formulas.

A \emph{(bimodal) normal logic} \(L\) is a set of formulas with the following properties:
\begin{enumerate}
  \item Every classical propositional tautology in the language \(\blan\) is in \(L\).
  \item (Axiom \(\Kax[i]\)) For any \(\phi, \psi\), \(\nec_i(\phi \to \psi) \to \nec_i \phi \to \nec_i \psi \) is in \(L\) for \(i \in \set{0,1}\).
  \item \(L\) is closed under the rules
    \[
      \AxiomC{\(\phi\)}
      \AxiomC{\(\phi \to \psi\)}
      \RightLabel{\(\MP\)}
      \BinaryInfC{\(\psi\)}
      \DisplayProof
      \qquad
      \AxiomC{\(\phi\)}
      \RightLabel{\(\NEC[0]\)}
      \UnaryInfC{\(\nec_0 \phi\)}
      \DisplayProof
      \qquad
      \AxiomC{\(\phi\)}
      \RightLabel{\(\NEC[1]\)}
      \UnaryInfC{\(\nec_1 \phi\)}
      \DisplayProof
    \]
\end{enumerate}
The elements of \(L\) are called the \emph{theorems of \(L\)}, and we will usually denote \(\phi \in L\) as \(L \vdash \phi\).
For a set of formulas \(\Gamma\), we will denote \(L \vdash \Gamma\) to mean \(L \vdash \phi\) for any \(\phi \in \Gamma\).

Given a bimodal logic \(L\) we define its \emph{normal extension by a set of formulas \(\Gamma\)}, denoted \(L \oplus \Gamma\), as the smallest bimodal logic containing \(\Gamma\).
We also define the smallest bimodal normal logic as the intersection of all bimodal normal logics and we denote it as \(\K[2]\), i.e., is the smallest set of formulas containing the propositonal tautologies, axiom \((\Kax[i])\) for \(i \leq 1\) and is closed under \((\MP)\) and \((\NEC[i])\) for \(i \leq 1\).

We proceed to define the bimodal logics \(\CS\), \(\CSM\) and \(\ER\).
\begin{definition}
  Define the following sets of formulas:
  \begin{multicols}{2}
    \begin{itemize}
      \item \(\Lax[i] = \set{\nec_i(\nec_i \phi \to \phi) \to \nec_i \phi \mid \phi \in \blan}\),
      \item \(\Max[i,j] = \set{\nec_i \phi \to \nec_j \phi \mid \phi \in \blan}\),
      \item \(\Cax[i,j] = \set{\nec_j \phi \to \nec_i \nec_j \phi \mid \phi \in \blan}\),
      \item \(\ERax[i,j] = \set{\nec_i (\nec_j \phi \to \phi) \mid \phi \in \blan}\).
    \end{itemize}
  \end{multicols}
  Then we define the following bimodal logics:
  \begin{multicols}{2}
    \begin{itemize}
      \item \(\CS = \K[2] \oplus (\Lax[0] \union \Lax[1] \union \Cax[0,1] \union \Cax[1,0])\),
      \item \(\ER = \CS \oplus \ERax[1,0]  = \CSM \oplus \ERax[1,0]\).
      \item \(\CSM = \CS \oplus \Max[0,1]\),
      \item[] \qedhere
    \end{itemize}
  \end{multicols}
\end{definition}
Note that $ \CS \oplus \ERax[1,0] \vdash  \Max[0,1]$.
We define the set of formulas \(4_i = \set{\nec_i \phi \to \nec_i \nec_i \phi \mid \phi \in \blan}\), i.e., \(\Cax[i,i]\).
Then, by the same reasoning that can be carried in \(\GL\), we obtain the following result.

\begin{lemma}
  \(\CS \vdash 4_i\), \(\CSM \vdash 4_i\) and \(\ER \vdash 4_i\) for \(i \leq 1\).
\end{lemma}

\subsection{Interpolation and modal equation systems}
\label{subsec:interpolation}

The main application of our non-wellfounded calculi will be to establish uniform Lyndon interpolation.
In this subsection we remember the related concepts that will be necessary for the proof.
A \emph{vocabulary} is just a set of propositonal variables.
First we define the positive and negative vocabulary of a formula.

\begin{definition}
  Given a formula \(\phi\) we define its \emph{positive and negative vocabulary}, denoted by \(\voc_+(\phi)\) and \(\voc_-(\phi)\), respectively, by
  \begin{align*}
    &\voc_+(p) = \set{p}, 
    &&\voc_-(p) = \varnothing, \\
    &\voc_+(\bot) = \varnothing, 
    &&\voc_-(\bot) = \varnothing, \\
    &\voc_+(\phi_0 \to \phi_1) = \voc_-(\phi_0) \union \voc_+(\phi_1), 
    &&\voc_-(\phi_0 \to \phi_1) = \voc_+(\phi_0) \union \voc_-(\phi_1), \\
    &\voc_+(\nec_i \phi_0) = \voc_+(\phi),
    &&\voc_-(\nec_i \phi_0) = \voc_-(\phi).
  \end{align*}
  We will write \(\overline{+}\) to mean \(-\) and \(\overline{-}\) to mean \(+\).
\end{definition}

The following lemma can be proved by induction on the complexity of \(\phi\).

\begin{lemma}
  \label{substitution-and-polarity}
  Let \(p\) be a propositional variable and \(\phi(p_0,\ldots,p_{n-1}, q_0,\ldots,q_{m-1})\), \(\psi_0,\ldots,\psi_{n-1}\),  and \(\chi_0,\ldots,\chi_{m-1}\) be formulas. 
%  We have that 
If\/ \(\voc_-(\phi) \cap \set{p_0,\ldots,p_{n-1}} =\varnothing\) and\/  \(\voc_+(\phi) \cap \set{q_0,\ldots,q_{m-1}} = \varnothing\), then for \(b \in \set{+,-}\)
  \begin{multline*}
    \voc_b(\phi(\psi_0,\ldots,\psi_{n-1}, \chi_0,\ldots,\chi_{m-1})) \subseteq\\
    \voc_b(\phi) \setminus\set{p_0,\ldots,p_{n-1},q_0,\ldots,q_{m-1}}
    \union \Union_{i < n} \voc_b(\psi_i) \union \Union_{j < m} \voc_{\overline{b}}(\chi_j).
  \end{multline*}
\end{lemma}

Uniform Lyndon interpolation says that given a formula \(\phi\) and vocabularies \(V_+\) and \(V_-\) there is a formula which best approximates \(\phi\) to the right in the set of formulas whose positive vocabulary is contained in \(V_+\) and whose negative vocabulary is contained in \(V_-\).
The exact definition is below.

\begin{definition}
  A logic \(L\) has \emph{uniform Lyndon interpolation} if for any formula \(\phi\) and vocabularies \(V_+, V_-\) there is a formula \(\iota\), called \emph{interpolant} of \(\phi\) for \((V_+,V_-)\) in \(L\), such that
  \begin{enumerate}
    \item \(\voc_+(\iota) \subseteq V_+\) and \(\voc_-(\iota) \subseteq V_-\),
    \item \(L \vdash \phi \to \iota\), and
    \item for any \(\psi\) such that \(\voc_+(\psi) \subseteq V_+\) and \(\voc_-(\psi) \subseteq V_-\), \(L \vdash \phi \to \psi\) implies \(L \vdash \iota \to \psi\). \qedhere
  \end{enumerate}
\end{definition}

We notice that given two interpolants \(\iota_0, \iota_1\) of \(\phi\) for \((V_+,V_-)\) in \(L\) we have that \(L \vdash \iota_0 \leftrightarrow \iota_1\).
In words, interpolants are unique up to logical equivalence.
For this reason it is common to talk about \emph{the} interpolant of a formula, even if the uniqueness only holds up to logical equivalence.

For the calculation of interpolants, it will be necessary to use a generalization of the notion of fixpoint called equational systems.
These equational systems are commonly used for \(\mu\)-calculus (e.g. see \cite{muCalcEq}), while in provability logic the concept is usually called simultaneous fixed-point theorem (e.g see \cite{on-the-proof-of-solovay}, called \(n\)-ary fixed point theorem).\footnote{\label{bekic-footnote} In fact, the idea that solving individual fixpoints suffices to solve simultaneous fixpoints is proven in full generality in Beki\'c Theorem, see \cite{Bekić1984}.}
We will adapt the concept to our needs, as the Lyndon property forces us to keep track of the polarity of variables.
We remember that a formula \(\phi\) is modalized in a variable \(p\) if all the occurences of \(p\) in \(\phi\) are under the scope of a \(\nec_i\).

\begin{definition}[Lyndon equational systems]
  Let \(V_+,V_-\) be vocabularies and \(\bar{p} = (p_0,\ldots,p_{n-1})\) be a finite sequence of pairwise different variables not occuring in \(V_+ \union V_-\).
  A \emph{Lyndon equational system over \((\bar{p},V_+,V_-)\)} is a set of triples \(\mathcal{E} = \set{(p_i,b_i,\phi_i) \mid i < n}\) such that \(b_i \in \set{+,-}\), \(\phi_i\) is a formula, \(\voc_{b_i}(\phi_i) \subseteq V_+ \union B_+\) and \(\voc_{\overline{b_i}}(\phi_i) \subseteq V_- \union B_-\), where \(B_+ = \set{p_j \mid j < n, b_j = +}\) and \(B_- = \set{p_j \mid j < n, b_j = -}\).
  The elements of \(\bar{p}\) are called \emph{unknowns} and those of \(\mathcal{E}\) are called \emph{equations}.

  A \emph{solution in \(L\)} to \(\mathcal{E}\) is a sequence  \((\psi_0,\ldots,\psi_{n-1})\) such that for each \(i \in n\),  we have
  \(\voc_{b_i}(\psi_i) \subseteq V_+\), \(\voc_{\overline{b_i}}(\psi_i) \subseteq V_-\) and \(L \vdash \psi_i \leftrightarrow \phi_i[\psi_0/p_0,\ldots,\psi_{n-1}/p_{n-1}]\).
  \(\mathcal{E}\) is said to be
  \begin{multicols}{2}
    \begin{enumerate}
      \item \emph{Solvable in \(L\)} if it has a solution in \(L\).
      \item \emph{Simple} if \(\phi_i\) is a \(\nec\)-formula for \(i < n\).
      \item \emph{Modalized} if \(\phi_{i}\) is modalized in \(p_0,\ldots,p_i\).
      \item \emph{Positive} if \(b_i = {+}\) for  \(i < n\).
    \end{enumerate}
  \end{multicols}
  Given an equational system \(\mathcal{E}\) over \((\bar{p}, V_+, V_-)\) with solution \((\psi_0,\ldots,\psi_{n-1})\) we will also call the substitution \((\cdot)^*\) where \(p_i^* = \psi_i\) and \(q^* = q\) for \(q\) not in \(\bar{p}\) a solution of \(\mathcal{E}\).
\end{definition}

The fundamental tool to solve Lyndon equational system is the concept of Lyndon fixpoint.

\begin{definition}[Lyndon Fixpoints]
  Let \(\phi(p)\) and \(\psi\) be formulas.
  We say that \(\psi\) is a \emph{Lyndon fixpoint of~\(\phi\) (with respect to \(p\)) in \(L\)} if \(\voc_{b}(\psi) \subseteq \voc_b(\phi) \setminus \set{p}\) for \(b \in \set{+,-}\) and \(L \vdash \psi \leftrightarrow \phi(\psi)\).
  A logic \(L\) is said to have \emph{basic Lyndon fixpoints} if every \(\nec\)-formula has a Lyndon fixpoint with respect to any variable, and is said to have \emph{positive modalized Lyndon fixpoints} if every formula \(\phi\) has a Lyndon fixpoint with respect to the variables in the set \(\set{p \in \var \setminus \voc_-(\phi) \mid \phi \text{ modalized in }p}\).
\end{definition}

If \(L\) is a bimodal normal logic with basic Lyndon fixpoints then it is possible to show that \(L\) has positive modalized Lyndon fixpoints and every positive modalized equational systems have solution in \(L\) (see Appendix~\ref{sec:solving-equational-systems}; the unimodal case, which is completely analogous, was shown at \cite{gls-ulip}).
We proceed to show that \(\CS\), \(\CSM\) and \(\ER\) have basic Lyndon fixpoints.
The proof follows \cite[pg.~78]{smorynski}, where it is proven that \(\GL\) has basic fixpoints, although it needs some tiny adaptation for the bimodal case (in particular the use of \(\Cax[1,0]\) and \(\Cax[0,1]\)).

\begin{lemma}[Substitution theorem]
  For any formulas \(\phi(p), \psi, \chi\) we have that
  \[
    \CS \vdash (\psi \leftrightarrow \chi) \wedge \bigwedge_{i \leq 1}\nec_{i}(\psi \leftrightarrow \chi) \to (\phi(\psi) \leftrightarrow \phi(\chi)).
  \]
\end{lemma}
\begin{proof}
  By induction on the complexity of \(\phi\).
\end{proof}

\begin{lemma}
  \(\CS\) has basic Lyndon fixpoints, in particular \(\nec_i \phi(p)\) has \(\nec_i \phi(\top)\) as a fixpoint.
  As a corollary \(\CSM\) and \(\ER\) also have basic Lyndon fixpoints.
\end{lemma}
\begin{proof}
  We have the following reasoning
  \begin{align*}
    &\CS \vdash \nec_i \phi(\top) \to (\top \leftrightarrow \nec_i\phi(\top))
    &&\text{propositional reasoning},
    \\
    &\CS \vdash \nec_i \phi(\top) \to (\top \leftrightarrow \nec_i\phi(\top)) \wedge \bigwedge_{i \leq 1}\nec_{i}(\top \leftrightarrow \nec_i\phi(\top))
    &&\text{by \((\Cax[\overline{i},i])\) and \((4_i)\)}, \\
    &\CS \vdash \nec_i \phi(\top) \to (\nec_i \phi(\top) \leftrightarrow \nec_i \phi( \nec_i \phi(\top)))
    &&\text{by Substitution Lemma}.
  \end{align*}
  Then it is clear that \(\CS \vdash \nec_i \phi(\top) \to \nec_i \phi(\nec_i \phi(\top))\), for the inverse implication we have the following reasoning.
  \begin{align*}
    &\CS \vdash \nec_i \phi(\top) \to (\top \leftrightarrow \nec_i\phi(\top)) \wedge \bigwedge_{i \leq 1}\nec_{i}(\top \leftrightarrow \nec_i\phi(\top))
    &&\text{shown above}, \\
    &\CS \vdash \nec_i \phi(\top) \to (\phi(\nec_i \phi(\top)) \leftrightarrow \phi(\top))
    &&\text{by Substitution Lemma}, \\
    &\CS \vdash \nec_i \phi(\top) \to (\phi(\nec_i \phi(\top)) \to \phi(\top))
    &&\text{by propositional reasoning}, \\
    &\CS \vdash \phi(\nec_i \phi(\top)) \to (\nec_i \phi(\top) \to \phi(\top))
    &&\text{by propositional reasoning}, \\
    &\CS \vdash \nec_i\phi(\nec_i \phi(\top)) \to \nec_i(\nec_i \phi(\top) \to \phi(\top))
    &&\text{by \(\NEC[i]\) and \((\Kax[i])\)}, \\
    &\CS \vdash \nec_i\phi(\nec_i \phi(\top)) \to \nec_i \phi(\top)
    &&\text{by \((\Lax[i])\)}.
  \end{align*}
\end{proof}
We obtain the desired properties of \(\CS\), \(\CSM\) and \(\ER\).

\begin{theorem}
  We have the following.
  \begin{enumerate}
    \item \(\CS\), \(\CSM\) and \(\ER\) have positive modalized Lyndon fixpoints.
    \item Any positive modalized Lyndon equational system has a solution in \(\CS\), \(\CSM\) and \(\ER\).
  \end{enumerate}
\end{theorem}

\subsection{Local progress proof theory}\label{subsec:local-progress-proof-theory}

Let us start by fixing the notions concerning trees.
Given a set \(X\) we will write \(X^*\) (\(X^+\)) to mean the set of (non-empty) finite sequences of elements in \(X\), \(\epsilon\) will denote the empty sequence.
As usual, we will write \(w \leq v\) to mean that \(w\) is an initial prefix of \(v\) and \((w,v]\) to mean the set \(\set{u \in X^* \mid w < u \leq v}\).
A (finitely branching) tree on \(A\) is a function \(T\) whose image is contained in \(A\) and whose domain is a prefix-closed non-empty subset of \(\mathbb{N}^*\) such that for every \(w \in \domain(T)\) there is an unique natural number \(k\) (called the \emph{arity of \(w\) in \(T\)}) such that \(wi \in \domain(T)\) iff \(i < k\).
The \(wi \in \domain(T)\) are also called the \emph{immediate successors of \(w\)}.
Note that the domain of any tree always contains the word \(\epsilon\), and its called the \emph{root of \(T\)}.

The elements of \(\domain(T)\) are also called the \emph{nodes of \(T\)}, the \(0\)-ary nodes are called \emph{leaves} and the rest of the nodes are called \emph{interior nodes}.
Finally, an \emph{infinite branch in a tree \(T\)} is sequence of nodes \((w_i)_{i \in \mathbb{N}}\) such that \(w_0 = \epsilon\) and for each \(i\) there is a \(j\) such that \(w_{i+1} = w_i j\).

We fix a set \(\text{Seq}\) whose elements we will call \emph{sequents}, a sequent rule is a subset of \(\text{Seq}^+\).
Given a rule \(R\), we say that \((S_0,\ldots,S_{n-1},S)\) is an \emph{instance of \(R\) with premises \(S_0,\ldots,S_{n-1}\) and conclusion \(S\)} if \((S_0,\ldots,S_{n-1},S) \in R\).
A rule \(R\) is said to be \emph{\(n\)-ary} if \(R \subseteq \text{Seq}^{n+1}\).
We introduce the kind of sequent calculi we are going to use, called \emph{local progress sequent calculi}~\cite{coalgebraic}.

\begin{definition}
  A \emph{local progress sequent calculus} is a pair \(\mathcal{G} = (\mathcal{R}, (L_R)_{R \in \mathcal{R}})\) where \(\mathcal{R}\) is a set of sequent rules and each \(L_R\) is a function that takes an instance \(r = (S_0,\ldots,S_{n-1},S)\) of \(R\) and returns a subset of \(\set{0,\ldots,n-1}\).
  \(\mathcal{G}\) is said to be \emph{wellfounded} if each \(L_R\) is the constant function returning \(\varnothing\).
\end{definition}

We are prepared to define the notion of proof in a local progress calculus.
From the defintion of proof we can infer that a wellfounded sequent calculus is just a sequent calculus in the usual (wellfounded) proof theory.

\begin{definition}
  Given a local progress sequent calculus \(\mathcal{G} = (\mathcal{R}, (L_R)_{R \in \mathcal{R}})\), a \emph{preproof in \(\mathcal{G}\)} is a tree with labels in \(\text{Seq} \times \mathcal{R}\) such that for each node \(w\) with immediate successors \(w0,\ldots,w(n-1)\) we have that \((S_0,\ldots,S_{n-1},S) \in R\) where \(S\) is the sequent at \(w\), \(R\) is the rule at \(w\) and \(S_i\) is the sequent at \(wi\).

  Let \(w\) be a node in \(\pi\) with immediate successors \(w0,\ldots,w(n-1)\).
  We say that \(wi\) is a \emph{progressing node} if \(i \in L_R(r)\) where \(R\) is the rule at \(w\) and \(r = (S_0,\ldots,S_{n-1},S)\) where \(S\) is the sequent at \(w\) and \(S_i\) the sequent at \(wi\).
  A proof is a preproof in which any infinite branch has infinitely many progressing nodes.
\end{definition}

We will write \(\mathcal{G} \vdash S\) to mean that \(S\) is provable in \(\mathcal{G}\) and \(\pi \vdash_{\mathcal{G}} S\) to mean that \(\pi\) is a proof of \(S\) in \(\mathcal{G}\), omitting the subscript \(_\mathcal{G}\) when it is clear from context.
It will be common to write an instance \((S_0,\ldots,S_{n-1},S)\) of a rule \(R\) as
\[
  \AxiomC{\(S_0\)}
  \AxiomC{\(\cdots\)}
  \AxiomC{\(S_{n-1}\)}
  \RightLabel{\(R\)}
  \TrinaryInfC{\(S\)}
  \DisplayProof
\]
Given a proof \(\pi\) in \(\mathcal{G}\) we will define its main local fragment as the finite tree obtained from cutting the tree at the first progressing nodes from the root (in particular removing the progressing nodes).
Figure~\ref{fig:fragments} shows a picture of how a proof in a local progress calculus looks like.
The gray triangle at the bottom is the main local fragment, while the circular nodes are non-progressing nodes and the square nodes progressing nodes.
The \emph{local height of \(\pi\)}, denoted \(\lhg(\pi)\), is the height of its main local fragment.
The \emph{local rules of \(\pi\)}, denoted \(\lrul(\pi)\), is the set of rules occuring in the main local fragment.
We will say that \(\pi\) is \emph{locally \(R\)-free} if \(R \not \in\lrul(\pi)\).
We note that if \(\mathcal{G}\) is a wellfounded calculus the notion of local height agrees with the usual notion of height and the local rules is the set of rules occuring in the proof.
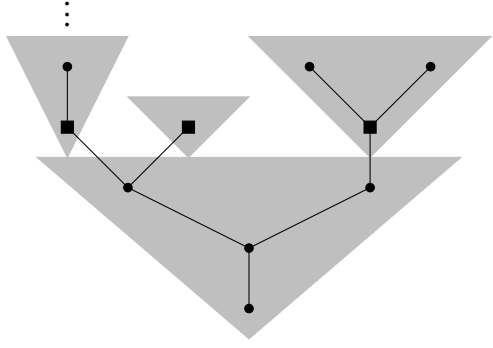
\begin{figure}
	\centering
	\begin{tikzpicture}[scale=0.8]
		\filldraw[gray!50] (0,-0.5) -- (3.5,2.5) -- (-3.5,2.5) -- cycle;
		\filldraw (0,0) circle (2pt);

		\filldraw (0,1) circle (2pt); \draw (0,0) -- (0,1);

		\filldraw[gray!50] (-3,2.5) -- (-4,4.5) -- (-2,4.5) -- cycle;
		\filldraw[gray!50] (-1,2.5) -- (-2,3.5) -- (0,3.5) -- cycle;
		\filldraw[gray!50] (2,2.5) -- (0,4.5) -- (4,4.5) -- cycle;
    \filldraw (-2,2) circle (2pt); \draw (0,1) -- (-2,2);
    \filldraw (2,2) circle (2pt); \draw (0,1) -- (2,2);

		\filldraw (-3.1,2.9) rectangle (-2.9,3.1); \draw (-2,2) -- (-3,3);
		\filldraw (-1.1,2.9) rectangle (-0.9,3.1); \draw (-2,2) -- (-1,3);
		\filldraw (1.9,2.9) rectangle (2.1,3.1); \draw (2,2) -- (2,3);

		\filldraw (-3,4) circle (2pt); \draw (-3,3) -- (-3,4); 
		\node at (-3,5) {\large\(\vdots\)};
		\filldraw (1,4) circle (2pt); \draw (2,3) -- (1,4); 
		\filldraw (3,4) circle (2pt); \draw (2,3) -- (3,4); 
	\end{tikzpicture}
	\caption{Structure of proofs in local progress calculi}
  	\label{fig:fragments}
\end{figure}

Given a local progress calculus \(\mathcal{G} = (\mathcal{R},(L_R)_{R \in \mathcal{R}})\) and a rule \(R'\) we define \(\mathcal{G} + R'\) as the local progress calculus with rules \(\mathcal{R} \union \set{R'}\) and \(L_{R'}\) the constant function returning \(\varnothing\).
Rules can interact with local progress calculi in different ways, the following definition introduce some of these ways and the theorem below shows that some of them are equivalent.

\begin{definition}\label{def:prop-of-rules}
  Let \(\mathcal{G}\) be a sequent calculus and \(R\) be a rule.
  We say that \(R\) is
  \begin{enumerate}
    \item \(R\) admissible if for any instance \((S_0,\ldots,S_{n-1},S) \in R\) we have that \(\pi_i \vdash_{\mathcal{G}} S_i\) for \(i < n\) implies the existence of \(\pi \vdash_{\mathcal{G}} S\).
      In addition we say that \(R\) is \emph{admissible preserving local height} if \(\lhg(\pi) \leq \max_{i < n}\lhg(\pi_i)\) and
        \emph{admissible preserving local rules} if \(\lrul(\pi) \subseteq \Union_{i < n}\lrul(\pi_i)\).
    \item \(R\) eliminable if \(\mathcal{G} + R\vdash S\) implies \(\mathcal{G} \vdash S\). 
    \item \(R\) locally admissible if for any instance \((S_0,\ldots,S_{n-1},S) \in R\) we have that if there are locally \(R\)-free \(\pi_i \vdash_{\mathcal{G}} S_i\)  for \(i < n\) then there is a locally \(R\)-free \(\pi \vdash_{\mathcal{G}} S\).
    \item An \(n\)-ary rule \(R\) is \(i\)-invertible  in \(\mathcal{G}\) (where \(i < n\)) if the rule
      \(
        \set{(S,S_i) \mid (S_0,\ldots,S_{n-1},S) \in R}
      \)
      is admissible.
      \(R\) is \emph{invertible} if it is \(i\)-invertible for each \(i < n\).
  \end{enumerate}
We will talk about invertibility preserving local height and/or local rules with the obvious meaning.
\end{definition}
Given an admissible rule \(R\) in \(\mathcal{G}\),an instance \((S_0,\ldots,S_{n-1},S) \in R\) and proofs \(\pi_i \vdash_{\mathcal{G}} S_i\) for \(i < n\), we will write \(R(\pi_0,\ldots,\pi_{n-1})\) to mean a proof of \(S\) in \(\mathcal{G}\) that exists by admissibility.
For the invertibility of a rule \(R\) we will write instead \(\inv{R}(\pi)\).

\begin{remark}
Let us briefly discuss our choice of terminology.
Lorenzen~\cite{Lorenzen} introduced admissible rules as those rules that can be added to a calculus without changing the set of provable formulas.  
Shortly thereafter,  Schütte~\cite{Schuette} gave the definition of admissiblity that also we use in Definition~\ref{def:prop-of-rules}.
Nowadays, admissibility is often used with respect to a logic and not a calculus, see, e.g.,~\cite{Rybakov1997}.  Then a rule $R$ is said to be admissible in a logic $L$ if it holds that if all premises of $R$ are elements of $L$,  then so is its conclusion. 

Over traditional well-foundend calculi,  Schütte's definition of a rule $R$ being admissible in a calculus $\mathcal{G}$ is equivalent to saying that $R$ is eliminable in $\mathcal{G}$ (in the sense of Definition~\ref{def:prop-of-rules}).  This is no longer the case for non-wellfounded calculi.  In that context,  the appropriate notion is local admissibility as shown in the next theorem.

Gentzen~\cite{Gentzen1935} originally stated his Hauptsatz as \emph{each derivation can be transformed into a derivation of the same endsequent that does not use cut.}
Although he mentions a transformation,  there are also many non-constructive cut-elimination results available in the literature,  see,  e.g.,~\cite{Tait1966-TAIANP,Takahashi1967} for early examples.
Also our definition of eliminability does not require an effective procedure. 
In  case we want to talk about the effectiveness of our methodology, we will say effective cut elimination explicitely.
\end{remark}

%
%\begin{remark}
%In the proof theoretical tradition, it is common to study the concept of effective cut elimination, i.e., study how to get rid of the cut rule (see below) in proofs using computable (effective) means. It is not weird in the community to understand that cut elimination is always effective and call non-effective cut elimination as cut admissibility.
%        This step is also justified since the two notions that we consider here, what we call being admissible and being eliminable at Definition~\ref{def:prop-of-rules}, are equivalent in the usual (wellfounded) proof theory.
%        This equivalence does not hold any longer in the non-wellfounded proof theory that we will use and the best approximation can be found at Theorem~\ref{th:local-adm-and-eliminability}.
%      For this reason, we will not follow the tradition and for us a rule being eliminable means exactly what it says: that the rule can be eliminated from the calculus without affecting the provability of sequents.
%    In case we want to talk about the effectiveness of our methodology we will say effective cut elimination explicitely.
%
%\end{remark}

The following is the key result of local progress proof theory (for details see \cite{proofth-interpretability}).

\begin{theorem}\label{local-adm-and-eliminability}
  Let \(\mathcal{G}\) be a local progress sequent calculus, then \(R\) is eliminable in \(\mathcal{G}\) iff \(R\) is locally admissible in \(\mathcal{G}\).
  If \(\mathcal{G}\) is wellfounded, then both are equivalent to \(R\) is admissible in \(\mathcal{G}\).
\end{theorem}

Finally, sometimes we will need to work with non-wellfounded proofs that have a particular finite representation.
To define this notion we introduce the notion of trees with backedges.
A \emph{tree with backedges} is an ordered pair \(\tau = (T,(\cdot)^\circ)\) such that \(T\) is a finite tree and \((\cdot)^\circ\) is a function from a subset of the leafs of~\(\tau\) to the nodes of \(\tau\) such that if \(w\) is in the domain then \(w^\circ < w\), the nodes in the domain of \((\cdot)^\circ\) are called \emph{repeat nodes}.

\begin{definition}
  Given a local progress calculus \(\mathcal{G} = (\mathcal{R}, (L_R)_{R \in \mathcal{R}})\), a \emph{cyclic preproof in \(\mathcal{G}\)} is a tree with backedges \(\pi = (T, (\cdot)^\circ)\) on \(\text{Seq} \times \mathcal{R}\) such that for any non-repeat node \(w\) with immediate successors \(w0\), \ldots, \(w(n-1)\), we have that \((S_0,\ldots,S_{n-1},S) \in R\), where \(S\) is the sequent at \(w\),  \(R\) is the rule at \(w\) and \(S_i\) is the sequent at \(wi\),
    and for any repeat node \(w\), \(w\) and \(w^\circ\) have the same sequent and rule.
    A \emph{cyclic proof in \(\mathcal{G}\)} is a cyclic preproof such that for any repeat node \(w\) there is a progressing node in \((w^\circ,w]\).
\end{definition}
In cyclic preproofs we will annotate repeat nodes with the word \(\rep\) at the right, instead of the rule at the node.

\section{Wellfounded proof theory of bimodal provability}
\label{sec:wellfounded}

In this section, we define the wellfounded sequent calculi \(\g{\CS}\), \(\g{\CSM}\) and \(\g{\ER}\).
In addition, we show that \(\g{\CS}\) corresponds to \(\CS\), \(\g{\CSM}\) to \(\CSM\) and \(\g{\ER}\) to \(\ER\).

\subsection{CS and CSM}

A \emph{sequent} is a pair \((\Gamma, \Delta)\) of finite multisets of formulas, usually denoted by \(\Gamma \Rightarrow \Delta\). \(\Gamma\) is called the \emph{left side} of the sequent while \(\Delta\) is called the \emph{right side}.
Given a multisets of formulas \(\Gamma\), we write \(\Gamma^s\) to mean the set obtained by deleting repetitions from \(\Gamma\).
If $\Sigma$ is a multiset $\{\phi_1,\ldots,\phi_n\}$,  then $\nec_i\Sigma$ stands for $\{\nec_i\phi_1,\ldots,\nec_i\phi_n\}$ where $i \in \{0,1\}$. Moreover, we use $\overline{i}$ for $1-i$, i.e.~$\overline{0}=1$ and vice versa.
Further,  $\necd_i \phi$ stands for $\nec_i \phi, \phi$ and we use this similarly for multisets of formulas.

\begin{figure}
  \[
    \AxiomC{}
    \RightLabel{\(\ax\)}
    \UnaryInfC{\(p, \Gamma \Rightarrow p, \Delta\)}
    \DisplayProof
    \qquad
    \AxiomC{}
    \RightLabel{\(\botL\)}
    \UnaryInfC{\(\bot, \Gamma \Rightarrow \Delta\)}
    \DisplayProof
    \qquad
    \AxiomC{\(\Gamma \Rightarrow \Delta\)}
    \RightLabel{\(\botR\)}
    \UnaryInfC{\(\Gamma \Rightarrow \bot, \Delta\)}
    \DisplayProof
  \]

  \[
    \AxiomC{\(\Gamma \Rightarrow \phi, \Delta\)}
    \AxiomC{\(\psi, \Gamma \Rightarrow \Delta\)}
    \RightLabel{\(\toL\)}
    \BinaryInfC{\(\phi \to \psi, \Gamma \Rightarrow \Delta\)}
    \DisplayProof
    \qquad
    \AxiomC{\(\phi, \Gamma \Rightarrow \psi, \Delta\)}
    \RightLabel{\(\toR\)}
    \UnaryInfC{\(\Gamma \Rightarrow \phi \to \psi, \Delta\)}
    \DisplayProof
  \]
  \caption{Propositional rules}
  \label{fig:prop-rules}
\end{figure}

\begin{figure}
  \[
    \AxiomC{\(\necd_i \Sigma_i, \nec_{\overline{i}} \Sigma_{\overline{i}}, \nec_i \phi \Rightarrow \phi\)}
    \RightLabel{\(\modal[i]{\CS}\)}
    \UnaryInfC{\(\nec_i \Sigma_i, \nec_{\overline{i}} \Sigma_{\overline{i}}, \Gamma \Rightarrow \nec_i \phi, \Delta\)}
    \DisplayProof
  \]
  \[
    \AxiomC{\(\necd_0 \Sigma_0, \nec_{1} \Sigma_{1}, \nec_0 \phi \Rightarrow \phi\)}
    \RightLabel{\(\modal[0]{\CSM}\)}
    \UnaryInfC{\(\nec_0 \Sigma_0, \nec_{1} \Sigma_{1}, \Gamma \Rightarrow \nec_i \phi, \Delta\)}
    \DisplayProof
    \quad
    \AxiomC{\(\necd_0 \Sigma_0, \necd_{1} \Sigma_{1}, \nec_1 \phi \Rightarrow \phi\)}
    \RightLabel{\(\modal[1]{\CSM}\)}
    \UnaryInfC{\(\nec_0 \Sigma_0, \nec_{1} \Sigma_{1}, \Gamma \Rightarrow \nec_i \phi, \Delta\)}
    \DisplayProof
  \]
  \caption{Modal rules for \(\g{\CS}\) and \(\g{\CSM}\)}
  \label{fig:modal-rules-CS-CSM}
\end{figure}

The rules needed to define calculi for \(\CS\) and \(\CSM\) are given in  Figures~\ref{fig:prop-rules} and \ref{fig:modal-rules-CS-CSM}.
In the rules \((\toL)\), \((\toR)\), \((\modal[i]{\CS})\) and \((\modal[i]{\CSM})\) the formula displayed at the conclusion is called the \emph{principal formula} (of the rule instance) and the displayed \(\nec\)-formula in the left side of the premise of the rules \((\modal[i]{\CS})\) and \((\modal[i]{\CSM})\) is called \emph{diagonal formula}.
In \((\modal[i]{\CS})\) and \((\modal[i]{\CSM})\) the formulas belonging to \(\nec_0 \Sigma_0, \nec_1 \Sigma_1\) are called \emph{auxiliary formulas}.
In \((\ax)\), \((\botL)\), \((\modal[i]{\CS})\) and \((\modal[i]{\CSM})\) the multisets \(\Gamma\) and \(\Delta\) are called the \emph{weakening formulas} of the rule instance.
Notice that the weakening formulas of a rule instance can be changed arbitrarily and will still be a rule instance of the same rule.

In the propositional rules of Figure~\ref{fig:prop-rules},  there are obvious omissions: namely, the usual rules \((\negL)\), \((\negR)\), \((\wedgeL)\), \((\wedgeR)\), \((\veeL)\) and \((\veeR)\).
This omission is not problematic, as these rules can be simulated with the propositional rules of the system.
Let us show as an example how to simulate \((\negR)\) and \((\veeL)\), remembering that \(\phi \vee \psi = \neg \phi \to \psi\).
\begin{align*}
  \AxiomC{\(\phi, \Gamma \Rightarrow \Delta\)}
  \RightLabel{\(\botR\)}
  \UnaryInfC{\(\phi, \Gamma \Rightarrow \bot, \Delta\)}
  \RightLabel{\(\toR\)}
  \UnaryInfC{\(\Gamma \Rightarrow \neg \phi, \Delta\)}
  \DisplayProof
  &&
  \AxiomC{\(\phi, \Gamma \Rightarrow \Delta\)}
  \RightLabel{\(\negR\)}
  \UnaryInfC{\(\Gamma \Rightarrow \neg \phi, \Delta\)}
  \AxiomC{\(\psi, \Gamma \Rightarrow \Delta\)}
  \RightLabel{\(\toL\)}
  \BinaryInfC{\(\phi \vee \psi, \Gamma \Rightarrow \Delta\)}
  \DisplayProof
\end{align*}
From now own we will use the rules \((\negL)\), \((\negR)\), \((\wedgeL)\), \((\wedgeR)\), \((\veeL)\), \((\veeR)\) inside proofs as abbreviations for the corresponding application of rules.

\begin{definition}
  We define the following wellfounded sequent calculi:
  \begin{enumerate}
    \item The calculus \(\g{\CS}\) is defined as having the rules of Figure~\ref{fig:prop-rules} together with the rules \((\modal[0]{\CS})\) and \((\modal[1]{\CS})\) from Figure~\ref{fig:modal-rules-CS-CSM}.
    \item The calculus \(\g{\CSM}\) is defined as having the rules of Figure~\ref{fig:prop-rules} together with the rules \((\modal[0]{\CSM})\) and \((\modal[1]{\CSM})\) from Figure~\ref{fig:modal-rules-CS-CSM}.
      \qedhere
  \end{enumerate}
\end{definition}

Then, it is easy to show the following lemma by induction on the complexity of \(\phi\).

\begin{lemma}
  We have that \(\g{\CS} \vdash \phi, \Gamma \Rightarrow \phi, \Delta\)\ and\ \(\g{\CSM} \vdash \phi, \Gamma \Rightarrow \phi, \Delta\).
\end{lemma}

When we use this lemma inside  proofs in \(\g{\CS}\) or \(\g{\CSM}\),  we will annotate the sequent with the rule~\((\Ax)\).

\begin{figure}
  \[
    \AxiomC{\(\Gamma \Rightarrow \Delta\)}
    \RightLabel{\(\wk\)}
    \UnaryInfC{\(\Gamma, \Gamma' \Rightarrow \Delta, \Delta'\)}
    \DisplayProof
    \qquad
    \AxiomC{\(\Gamma, \Gamma' \Rightarrow \Delta, \Delta'\)}
    \RightLabel{\(\ctr\)}
    \UnaryInfC{\(\Gamma, (\Gamma')^s \Rightarrow \Delta, (\Delta')^s\)}
    \DisplayProof
  \]
  where we remember that for a multiset \(\Theta\), \(\Theta^s\) is \(\Theta\) without repeptitions.
  \[
    \AxiomC{\(\Gamma \Rightarrow \Delta, \chi\)}
    \AxiomC{\(\chi, \Gamma \Rightarrow \Delta\)}
    \RightLabel{\(\cut\)}
    \BinaryInfC{\(\Gamma \Rightarrow \Delta\)}
    \DisplayProof
  \]
  \caption{Structural rules}
  \label{fig:structural-rules}
\end{figure}

As usual, even if they do not form part of the sequent calculus explicitely, there are some structural rules (see Figure~\ref{fig:structural-rules}) which we expect to behave well with the defined calculi.
In the following lemma we put together the structural properties of \(\g{\CS}\) and \(\g{\CSM}\) which are easy to show, all of them by induction on the height of the proof (using Theorem~\ref{local-adm-and-eliminability} when necessary).

\begin{lemma}\label{lm:structural-rules-CS-CSM}
  In \(\g{\CS}\), \(\g{\CSM}\), \(\g{\CS} + \cut\) and \(\g{\CSM} + \cut\) we have the following.
  \begin{enumerate}
    \item \((\wk)\) is eliminable and admissible preserving height and rules.
    \item \((\botR)\), \((\toL)\) and \((\toR)\) are invertible preserving height and rules.
    \item \((\ctr)\) is eliminable and admissible preserving height and rules.
  \end{enumerate}
\end{lemma}

Showing cut elimination will require a bit more work and will be done in Section~\ref{sec:cut-elim}.
However, using the previous lemma, we can establish the connection between \(\CS\) and \(\g{\CS} + \cut\) and between \(\CSM\) and \(\g{\CS} + \cut\).
The first correspondence is in the following lemma.

\begin{lemma}\label{lm:correspondence-to-hilbert-CS-CSM}
  For any formula \(\phi\) we have the following.
  \begin{enumerate}
    \item \(\CS \vdash \phi\) implies \(\g{\CS} + \cut \vdash { \Rightarrow \phi}\).
    \item \(\CSM \vdash \phi\) implies \(\g{\CSM} + \cut \vdash { \Rightarrow \phi}\).
  \end{enumerate}
\end{lemma}
\begin{proof}
  Proof of 1.
  It suffices to show that \(\set{\phi \mid \g{\CS} \vdash { \Rightarrow \phi}}\) contains the axioms of \(\CS\) and is closed under modus ponens, \((\NEC[0])\) and \((\NEC[1])\).
  That it contains all propositional tautologies is clear thanks to having rules \((\ax)\), \((\botL)\), \((\toL)\) and \((\toR)\).
  The proofs of \((\Kax[i])\), \((\Lax[i])\) and \((\Cax[i,\overline{i}])\) are as follows.
  \[
    \AxiomC{}
    \RightLabel{\(\Ax\)}
    \UnaryInfC{\(\nec_i(\phi \to \psi), \necd_i \phi, \nec_i \psi \Rightarrow  \phi, \psi\)}
    \AxiomC{}
    \RightLabel{\(\Ax\)}
    \UnaryInfC{\(\psi, \nec_i(\phi \to \psi), \necd_i \phi, \nec_i \psi \Rightarrow  \psi\)}
    \RightLabel{\(\toL\)}
    \BinaryInfC{\(\necd_i(\phi \to \psi), \necd_i \phi, \nec_i \psi \Rightarrow  \psi\)}
    \RightLabel{\(\modal[i]{\CS}\)}
    \UnaryInfC{\(\nec_i(\phi \to \psi), \nec_i \phi \Rightarrow  \nec_i \psi\)}
    \doubleLine
    \RightLabel{\(\toR\)}
    \UnaryInfC{\(\Rightarrow \nec_i(\phi \to \psi) \to \nec_i \phi \to \nec_i \psi\)}
    \DisplayProof
  \]
  \vspace{0.3cm}
  \[
    \AxiomC{}
    \RightLabel{\(\Ax\)}
    \UnaryInfC{\(\nec_i(\nec_i \phi \to \phi), \nec_i \phi \Rightarrow \nec_i \phi, \phi \)}
    \AxiomC{}
    \RightLabel{\(\Ax\)}
    \UnaryInfC{\(\phi, \nec_i(\nec_i \phi \to \phi), \nec_i \phi \Rightarrow \phi \)}
    \RightLabel{\(\toL\)}
    \BinaryInfC{\( \necd_i(\nec_i \phi \to \phi), \nec_i \phi \Rightarrow \phi \)}
    \RightLabel{\(\modal[i]{\CS}\)}
    \UnaryInfC{\( \nec_i(\nec_i \phi \to \phi) \Rightarrow \nec_i \phi \)}
    \RightLabel{\(\toR\)}
    \UnaryInfC{\(\Rightarrow \nec_i(\nec_i \phi \to \phi) \to \nec_i \phi \)}
    \DisplayProof
  \]
  \vspace{0.3cm}
  \[
    \AxiomC{}
    \RightLabel{\(\Ax\)}
    \UnaryInfC{\( \nec_i \phi, \nec_{\overline{i}}\nec_i \phi \Rightarrow \nec_i \phi\)}
    \RightLabel{\(\modal[\overline{i}]{\CS}\)}
    \UnaryInfC{\( \nec_i \phi \Rightarrow \nec_{\overline{i}}\nec_i \phi\)}
    \RightLabel{\(\toR\)}
    \UnaryInfC{\( \Rightarrow \nec_i \phi \to \nec_{\overline{i}}\nec_i \phi\)}
    \DisplayProof
  \]

  All left to show are the closure properties. We start with Modus Ponens.
  Assume \(\pi \vdash { \Rightarrow \phi \to \psi}\) and \(\tau \vdash { \Rightarrow\phi}\) in \(\g{\CS} + \cut\), then we have the proof
  \[
    \AxiomC{\(\wk(\tau)\)}
    \noLine
    \UnaryInfC{\(\Rightarrow \psi, \phi\)}
    \AxiomC{\(\inv{\toR}(\pi)\)}
    \noLine
    \UnaryInfC{\(\phi \Rightarrow \psi\)}
    \RightLabel{\(\cut\)}
    \BinaryInfC{\( \Rightarrow \psi\)}
    \DisplayProof
  \]
  Finally,  for Necessitation,  assume that \(\pi \vdash { \Rightarrow \phi}\). Then the desired proof in \(\g{\CS} + \cut\) is
  \[
    \AxiomC{\(\wk(\pi)\)}
    \noLine
    \UnaryInfC{\( \nec_i \phi \Rightarrow \phi\)}
    \RightLabel{\(\modal[i]{\CS}\)}
    \UnaryInfC{\( \Rightarrow \nec_i \phi\)}
    \DisplayProof
  \]

  Proof of 2.
  The proof is largely the same.  We just have to replace the rule \((\modal[i]{\CS})\) in the previous proofs with \((\modal[i]{\CSM})\).
  We only show that \((\Max[0,1])\) is contained in \(\set{\phi \mid \g{\CSM} + \cut \vdash { \Rightarrow \phi}}\).
  \[
    \AxiomC{}
    \RightLabel{\(\Ax\)}
    \UnaryInfC{\(  \necd_0 \phi, \nec_1 \phi \Rightarrow \phi\)}
    \RightLabel{\(\modal[i]{\CSM}\)}
    \UnaryInfC{\(  \nec_0 \phi \Rightarrow \nec_1 \phi\)}
    \RightLabel{\(\toR\)}
    \UnaryInfC{\( \Rightarrow \nec_0 \phi \to \nec_1 \phi\)}
    \DisplayProof
  \]
\end{proof}

For the full correspondence, we need to interpret sequents as formulas.
Given a sequent \(S = \Gamma \Rightarrow \Delta\) we define its \emph{formula interpretation} as
\[
  S^\flat = \bigwedge \Gamma \to \bigvee \Delta.
\]

\begin{definition}
  Let \(R\) be a sequent rule and \(L\) be a logic.
  We say that \(R\) is \emph{sound in \(L\)} if for any instance \((S_0,\ldots,S_{n-1},S) \in R\) we have that \(L \vdash S^\flat_i\) for all \(i < n\) implies \(L \vdash S^\flat\).
\end{definition}

\begin{lemma}\label{lm:soundness-of-rules-CS-CSM}
  The rules of \(\g{\CS} + \cut\) are sound in \(\CS\) and the rules of \(\g{\CSM} + \cut\) are sound in \(\CSM\).
\end{lemma}
\begin{proof}
  Proof for \(\g{\CS}\).
  First, let \(R\) be one of \((\ax)\), \((\botL)\), \((\botR)\), \((\toL)\), \((\toR)\) or \((\cut)\).
  Then for any instance \((S_0,\ldots,S_{n-1},S) \in R\) we have that \(\bigwedge_{i < n} S_i^\flat \to S^\flat\) is a propositional tautology,  and hence \(\CS \vdash \bigwedge_{i < n} S_i^\flat \to S^\flat\).  Soundness of the rule \(R\) follows trivially from this.

  Finally, let \(R = (\modal[i]{\CS})\) and assume we have an instance
  \[
    \AxiomC{\(\necd_i \Sigma_i, \nec_{\overline{i}} \Sigma_{\overline{i}}, \nec_i \phi \Rightarrow \phi\)}
    \RightLabel{\(\modal[i]{\CS}\)}
    \UnaryInfC{\(\nec_i \Sigma_i, \nec_{\overline{i}} \Sigma_{\overline{i}}, \Gamma \Rightarrow \nec_i \phi, \Delta\)}
    \DisplayProof
  \]
  Then we have the following reasoning
  \begin{align*}
    &\CS \vdash \bigwedge \necd_i \Sigma_i \wedge \bigwedge \nec_{\overline{i}} \Sigma_{\overline{i}} \wedge \nec_i \phi \to \phi,
    &&\text{by assumption}, \\
    &\CS \vdash \bigwedge \necd_i \Sigma_i \wedge \bigwedge \nec_{\overline{i}} \Sigma_{\overline{i}} \to (\nec_i \phi \to \phi),
    &&\text{by propositional reasoning}, \\
    &\CS \vdash \bigwedge \nec_i\necd_i \Sigma_i \wedge \bigwedge \nec_i \nec_{\overline{i}} \Sigma_{\overline{i}} \to \nec_i (\nec_i \phi \to \phi),
    &&\text{by (\(\NEC[i]\)) and \((\Kax[i])\)}, \\
    &\CS \vdash \bigwedge \nec_i \Sigma_i \wedge \bigwedge \nec_i \nec_{\overline{i}} \Sigma_{\overline{i}} \to \nec_i (\nec_i \phi \to \phi),
    &&\text{by transitivity}, \\
    &\CS \vdash \bigwedge \nec_i \Sigma_i \wedge \bigwedge \nec_{\overline{i}} \Sigma_{\overline{i}} \to \nec_i (\nec_i \phi \to \phi),
    &&\text{by \((\Cax[i,\overline{i}])\)}, \\
    &\CS \vdash \bigwedge \nec_i \Sigma_i \wedge \bigwedge \nec_{\overline{i}} \Sigma_{\overline{i}} \to \nec_i \phi,
    &&\text{by \((\Lax[i])\)}, \\
    &\CS \vdash \bigwedge \nec_i \Sigma_i \wedge \bigwedge \nec_{\overline{i}} \Sigma_{\overline{i}} \wedge \bigwedge \Gamma \to \nec_i \phi \vee \bigvee \Delta,
    &&\text{by propositional reasoning},
  \end{align*}
  as desired.

  Proof for \(\g{\CSM}\).
  In case \(R\) is one of \((\ax)\), \((\botL)\), \((\botR)\), \((\toL)\), \((\toR)\), \((\cut)\) or \((\modal[0]{\CSM})\) the reasoning is the same as for \(\g{\CS}\).
  So assume \(R = (\modal[1]{\CSM})\) and that we have an instance
  \[
    \AxiomC{\(\necd_0 \Sigma_0, \necd_1 \Sigma, \nec_1 \phi \Rightarrow \phi\)}
    \RightLabel{\(\modal[1]{\CSM}\)}
    \UnaryInfC{\(\nec_0 \Sigma_0, \nec_1 \Sigma, \Gamma \Rightarrow \nec_1 \phi, \Delta\)}
    \DisplayProof
  \]
  Then we have the following reasoning
  \begin{align*}
    &\CSM \vdash \bigwedge \necd_0 \Sigma_0 \wedge \bigwedge \necd_{1} \Sigma_{1} \wedge \nec_1 \phi \to \phi,
    &&\text{by assumption}, \\
    &\CSM \vdash \bigwedge \necd_0 \Sigma_0 \wedge \bigwedge \necd_{1} \Sigma_{1} \to (\nec_1 \phi \to \phi),
    &&\text{by propositional reasoning}, \\
    &\CSM \vdash \bigwedge \nec_1\necd_0 \Sigma_0 \wedge \bigwedge \nec_1\necd_{1} \Sigma_{1} \to \nec_1 (\nec_1 \phi \to \phi),
    &&\text{by (\(\NEC[1]\)) and \((\Kax[1])\)}, \\
    &\CSM \vdash \bigwedge \nec_1\necd_0 \Sigma_0 \wedge \bigwedge \nec_1 \Sigma_{1} \to \nec_1 (\nec_1 \phi \to \phi),
    &&\text{by transitivity}, \\
    &\CSM \vdash \bigwedge \nec_0 \Sigma_0 \wedge \bigwedge \nec_1 \Sigma_0 \wedge \bigwedge \nec_1 \Sigma_{1} \to \nec_1 (\nec_1 \phi \to \phi),
    &&\text{by \((\Cax[0,1])\)}, \\
    &\CSM \vdash \bigwedge \nec_0 \Sigma_0 \wedge \bigwedge \nec_1 \Sigma_{1} \to \nec_1 (\nec_1 \phi \to \phi),
    &&\text{by \((\Max[0,1])\)}, \\
    &\CSM \vdash \bigwedge \nec_0 \Sigma_0 \wedge \bigwedge \nec_1 \Sigma_{1} \to \nec_1 \phi,
    &&\text{by \((\Lax[i])\)}, \\
    &\CSM \vdash \bigwedge \nec_0 \Sigma_0 \wedge \bigwedge \nec_1 \Sigma_{1} \wedge \bigwedge \Gamma \to \nec_1 \phi \vee \bigvee \Delta,
    &&\text{by propositional reasoning},
  \end{align*}
  as desired.
\end{proof}

The full correspondence is stated in the following theorem.
\begin{theorem}\label{th:correspondence-to-hilbert-CS-CSM}
  Let \(S\) be a sequent.
  We have that:
  \begin{enumerate}
    \item \(\CS \vdash S^\flat\) iff \(\g{\CS} + \cut \vdash S\).
    \item \(\CSM \vdash S^\flat\) iff \(\g{\CSM} + \cut \vdash S\).
  \end{enumerate}
\end{theorem}
\begin{proof}

  Proof of 1.
  Let \(S = \Gamma \Rightarrow \Delta\).
  First assume that \(\CS \vdash \bigwedge \Gamma \to \bigvee \Delta \), then by Lemma~\ref{lm:correspondence-to-hilbert-CS-CSM} using invertibility of propositional rules we get that \(\g{\CS} + \cut \vdash { \Rightarrow \bigwedge \Gamma \to \bigvee \Delta}\).
  By invertibility of the propositional rules we obtain the desired \(\g{\CS} + \cut \vdash \Gamma \Rightarrow \Delta\).
  For the opposite direction it suffices to use Lemma~\ref{lm:soundness-of-rules-CS-CSM} with an induction on the height of the sequent proof.

  Proof of 2. Analogous to the proof of 1.
\end{proof}

\subsection{ER}

In order to acommodate the axiom \((\ERax[1,0])\) we need to make \(\nec_0\) reflexive, but only when it is directly under a \(\nec_1\) modality. 
Hence we need a mechanism to control when we are directly under a \(\nec_1\) modality in a proof and  allow the use of the reflexivity rule only in those situations.
For that purpose we will define a calculus with two different kinds of sequents, as  shown below.

An \(\ER\)-\emph{sequent} is a triple \((\Gamma, \Delta, i)\) where \(\Gamma, \Delta\) are multisets of formulas and \(i \in \set{0,1}\).
We will denote \((\Gamma, \Delta, 0)\) by \(\Gamma \Rightarrow_0 \Delta\) and \((\Gamma, \Delta, 1)\) by \(\Gamma \Rightarrow_1 \Delta\). \(\Gamma\) is called the \emph{left side} of the sequent while \(\Delta\) is called the \emph{right side}.
We will use \(\gg\) to mean either \( \Rightarrow_0 \) or \(\Rightarrow_1\) (or also \( \Rightarrow \) if we want to talk about non-\(\ER\)-sequents and \(\ER\)-sequents at the same time).

\begin{figure}
  \[
    \AxiomC{}
    \RightLabel{\(\ax\)}
    \UnaryInfC{\(p, \Gamma \Rightarrow_i p, \Delta\)}
    \DisplayProof
    \qquad
    \AxiomC{}
    \RightLabel{\(\botL\)}
    \UnaryInfC{\(\bot, \Gamma \Rightarrow_i \Delta\)}
    \DisplayProof
    \qquad
    \AxiomC{\(\Gamma \Rightarrow_i \Delta\)}
    \RightLabel{\(\botR\)}
    \UnaryInfC{\(\Gamma \Rightarrow_i \bot, \Delta\)}
    \DisplayProof
  \]

  \[
    \AxiomC{\(\Gamma \Rightarrow_i \phi, \Delta\)}
    \AxiomC{\(\psi, \Gamma \Rightarrow_i \Delta\)}
    \RightLabel{\(\toL\)}
    \BinaryInfC{\(\phi \to \psi, \Gamma \Rightarrow_i \Delta\)}
    \DisplayProof
    \qquad
    \AxiomC{\(\phi, \Gamma \Rightarrow_i \psi, \Delta\)}
    \RightLabel{\(\toR\)}
    \UnaryInfC{\(\Gamma \Rightarrow_i \phi \to \psi, \Delta\)}
    \DisplayProof
  \]
  \caption{Propositional rules for \(\ER\)}
  \label{fig:prop-rules-ER}
\end{figure}

\begin{figure}
  \[
    \AxiomC{\(\necd_i \Sigma_i, \nec_{\overline{i}} \Sigma_{\overline{i}}, \nec_i \phi \Rightarrow_i \phi\)}
    \RightLabel{\(\modal[i]{\ER}\)}
    \UnaryInfC{\(\nec_i \Sigma_i, \nec_{\overline{i}} \Sigma_{\overline{i}}, \Gamma \gg \nec_i \phi, \Delta\)}
    \DisplayProof
    \qquad
    \AxiomC{\(\necd_0 \phi, \Gamma \Rightarrow_1 \Delta\)}
    \RightLabel{\(\ERax[1,0]\)}
    \UnaryInfC{\(\nec_0 \phi, \Gamma \Rightarrow_1 \Delta\)}
    \DisplayProof
  \]
  \caption{Modal rules for \(\g{\ER}\)}
  \label{fig:modal-rules-ER}
\end{figure}

The rules that we are going to need to define calculi for \(\ER\) can be found at Figures~\ref{fig:prop-rules-ER} and \ref{fig:modal-rules-ER}.
In the rules \((\toL)\), \((\toR)\), \((\modal[i]{\ER})\), \((\ERax[1,0])\) the formula displayed at the conclusion is called the \emph{principal formula} (of the rule instance) and the displayed \(\nec\)-formula in the left side of the premise of the rules \((\modal[i]{\ER})\) is called \emph{diagonal formula}.
In \((\modal[i]{\ER})\) the formulas belonging to \(\nec_0 \Sigma_0, \nec_1 \Sigma_1\) are called \emph{auxiliary formulas}.
In \((\ax)\), \((\botL)\), \((\modal[i]{\ER})\) the multisets \(\Gamma\) and \(\Delta\) are called the \emph{weakening formulas} of the rule instance.
Notice that the weakening formulas of a rule instance can be changed arbitrarily and will still be a rule instance of the same rule.

\begin{definition}
  We define the wellfounded sequent calculi \(\g{\ER}\) as having the rules of Figure~\ref{fig:prop-rules-ER} together with the rules of Figure~\ref{fig:modal-rules-ER}.
\end{definition}

Again, we have  the following lemma.

\begin{lemma}
  We have that \(\g{\ER} \vdash \phi, \Gamma \gg \phi, \Delta\).
\end{lemma}

\begin{figure}[htb]
  \[
    \AxiomC{\(\Gamma \gg \Delta\)}
    \RightLabel{\(\wk\)}
    \UnaryInfC{\(\Gamma, \Gamma' \gg \Delta, \Delta'\)}
    \DisplayProof
    \qquad
    \AxiomC{\(\Gamma, \Gamma' \gg \Delta, \Delta'\)}
    \RightLabel{\(\ctr\)}
    \UnaryInfC{\(\Gamma, (\Gamma')^s \gg \Delta, (\Delta')^s\)}
    \DisplayProof
  \]
  where we remember that for a multiset \(\Theta\), \(\Theta^s\) is \(\Theta\) without repeptitions.
  \[
    \AxiomC{\(\Gamma \gg \Delta, \chi\)}
    \AxiomC{\(\chi, \Gamma \gg \Delta\)}
    \RightLabel{\(\cut\)}
    \BinaryInfC{\(\Gamma \gg \Delta\)}
    \DisplayProof
    \qquad
    \AxiomC{\(\Gamma \Rightarrow_0 \Delta\)}
    \RightLabel{\(\Rightarrow_1\)}
    \UnaryInfC{\(\Gamma \Rightarrow_1 \Delta\)}
    \DisplayProof
  \]
  \caption{Structural rules}
  \label{fig:structural-rules-ER}
\end{figure}

In the following lemma, we put together some structural properties of \(\g{\ER}\) (see Figure~\ref{fig:structural-rules-ER}),  which can be shown by an induction on the height and Theorem~\ref{local-adm-and-eliminability}.
Again, cut elimination will be proven in Section~\ref{sec:cut-elim}

\begin{lemma}\label{lm:structural-rules-ER}
  In \(\g{\ER}\) and \(\g{\ER} + \cut\) we have the following.
  \begin{enumerate}
    \item \((\wk)\) is eliminable and admissible preserving height and rules.
    \item \((\Rightarrow_1)\) is eliminable and admissible preserving height and rules.
    \item \((\botR)\), \((\toL)\) and \((\toR)\) are invertible preserving height and rules.
    \item \((\ctr)\) is eliminable and admissible preserving height and rules.
  \end{enumerate}
\end{lemma}

Again, using the good structural behaviour of the sequent calculus we can start relating \(\ER\) and \(\g{\ER} + \cut\).
Let us show the first correspondence.

\begin{lemma}\label{lm:correspondence-to-hilbert-ER}
  For any formula \(\phi\), \(\ER \vdash \phi\) implies \(\g{\ER} + \cut \vdash { \Rightarrow_0 \phi}\).
\end{lemma}
\begin{proof}
  We have to show that \(\set{\phi \mid \g{\ER} + \cut \vdash { \Rightarrow_0 \phi}}\) contains the axioms of \(\ER\) and is closed under \((\MP)\), \((\NEC[0])\) and \((\NEC[1])\).
  In most of them the necessary proofs in \(\g{\ER} + \cut\) look as the proofs in Lemma~\ref{lm:correspondence-to-hilbert-CS-CSM}, so we will skip them.
  The only new axiom is \((\ERax[1,0])\) and the desired proof is
  \[
    \AxiomC{}
    \RightLabel{\(\Ax\)}
    \UnaryInfC{\( \necd_0 \phi, \nec_1(\nec_0 \phi \to \phi)\Rightarrow_1  \phi\)}
    \RightLabel{\(\ERax[1,0]\)}
    \UnaryInfC{\( \nec_0 \phi, \nec_1(\nec_0 \phi \to \phi)\Rightarrow_1  \phi\)}
    \RightLabel{\(\toR\)}
    \UnaryInfC{\( \nec_1(\nec_0 \phi \to \phi)\Rightarrow_1 \nec_0 \phi \to \phi\)}
    \RightLabel{\(\modal[1]{\ER}\)}
    \UnaryInfC{\( \Rightarrow_0 \nec_1(\nec_0 \phi \to \phi)\)}
    \DisplayProof
  \]

  The closure properties are proven as in Lemma~\ref{lm:correspondence-to-hilbert-CS-CSM}.
\end{proof}

As before, for the full correspondence we need to interpret \(\ER\)-sequents as formulas.
Given a sequent \(S = \Gamma \Rightarrow_i \Delta\) we define its \emph{formula interpretation} as
\[
  S^\flat = 
  \begin{cases}
    \bigwedge \Gamma \to \bigvee \Delta &\text{if \(i = 0\)}, \\
    \nec_1(\bigwedge \Gamma \to \bigvee \Delta) &\text{if \(i = 1\)}.
  \end{cases}
\]

\begin{lemma}\label{lm:soundness-of-rules-ER}
  The rules of \(\g{\ER} + \cut\) are sound in \(\ER\).
\end{lemma}
\begin{proof}
  If \(R\) is one of \((\ax)\), \((\botL)\), \((\botR)\), \((\toL)\), \((\toR)\) or \((\cut)\) with \( \Rightarrow_0\)-sequents, then the proof is exactly as in Lemma~\ref{lm:soundness-of-rules-CS-CSM}.
  Assume \(R\) is one of those rules but with \(\Rightarrow_1\)-sequents and consider an instance
  \[
    \AxiomC{\(\Gamma_0 \Rightarrow_1 \Delta_0\)}
    \AxiomC{\(\cdots\)}
    \AxiomC{\(\Gamma_{n-1} \Rightarrow_1 \Delta_{n-1}\)}
    \RightLabel{\(R\)}
    \TrinaryInfC{\(\Gamma \Rightarrow_1 \Delta\)}
    \DisplayProof
  \]
  Note that \(\ER \vdash (\bigwedge_{i < n} \bigwedge \Gamma_i \to \bigvee \Delta_i) \to (\bigwedge \Gamma \to \bigvee \Delta)\) as it is a propositional tautology.
  Then using \((\NEC[1])\) and \((\Kax[1])\) we obtain that \(\ER \vdash (\bigwedge_{i < n} \nec_1(\bigwedge \Gamma_i \to \bigvee \Delta_i)) \to \nec_1(\bigwedge \Gamma \to \bigvee \Delta)\), from where the desired soundness follows.

  Assume \(R\) is \((\modal[i]{\ER})\) and that we have an instance
  \[
    \AxiomC{\(\necd_i \Sigma_i, \nec_{\overline{i}} \Sigma_{\overline{i}}, \nec_i \phi \Rightarrow_0 \phi\)}
    \RightLabel{\(\modal[i]{\ER}\)}
    \UnaryInfC{\(\nec_i \Sigma_i, \nec_{\overline{i}} \Sigma_{\overline{i}}, \Gamma \gg \nec_i \phi, \Delta\)}
    \DisplayProof
  \]
  Then we have the following reasoning
  \begin{align*}
    &\ER \vdash \bigwedge \necd_i \Sigma_i \wedge \bigwedge \nec_{\overline{i}} \Sigma_{\overline{i}} \wedge \nec_i \phi \to \phi,
    &&\text{by assumption}, \\
    &\ER \vdash \bigwedge \necd_i \Sigma_i \wedge \bigwedge \nec_{\overline{i}} \Sigma_{\overline{i}} \to (\nec_i \phi \to \phi),
    &&\text{by propositional reasoning}, \\
    &\ER \vdash \bigwedge \nec_i\necd_i \Sigma_i \wedge \bigwedge \nec_i\nec_{\overline{i}} \Sigma_{\overline{i}} \to \nec_i(\nec_i \phi \to \phi),
    &&\text{by \((\NEC[i])\) and \((\Kax[i])\)}, \\
    &\ER \vdash \bigwedge \nec_i \Sigma_i \wedge \bigwedge \nec_i\nec_{\overline{i}} \Sigma_{\overline{i}} \to \nec_i(\nec_i \phi \to \phi),
    &&\text{by \((4_i)\)}, \\
    &\ER \vdash \bigwedge \nec_i \Sigma_i \wedge \bigwedge \nec_{\overline{i}} \Sigma_{\overline{i}} \to \nec_i(\nec_i \phi \to \phi),
    &&\text{by \((\Cax[i,\overline{i}])\)}, \\
    &\ER \vdash \bigwedge \nec_i \Sigma_i \wedge \bigwedge \nec_{\overline{i}} \Sigma_{\overline{i}} \to \nec_i\phi,
    &&\text{by \((\Lax[i])\)}, \\
    &\ER \vdash \bigwedge \nec_i \Sigma_i \wedge \bigwedge \nec_{\overline{i}} \Sigma_{\overline{i}} \wedge \bigwedge \Gamma \to \nec_i\phi \vee \bigvee \Delta,
    &&\text{by propositional reasoning}, \\
    &\ER \vdash \nec_1\left(\bigwedge \nec_i \Sigma_i \wedge \bigwedge \nec_{\overline{i}} \Sigma_{\overline{i}} \wedge \bigwedge \Gamma \to \nec_i\phi \vee \bigvee \Delta\right),
    &&\text{by \((\NEC[1])\) (only if \({\gg} \text{ is } {\Rightarrow_1}\))}.
  \end{align*}

  Assume \(R\) is \((\ERax[1,0])\), that we have an instance
  \[
    \AxiomC{\(\necd_0 \phi, \Gamma \Rightarrow_1 \Delta\)}
    \RightLabel{\(\ERax[1,0]\)}
    \UnaryInfC{\(\nec_0 \phi, \Gamma \Rightarrow_1 \Delta\)}
    \DisplayProof
  \]
  and  that \(\ER \vdash \nec_1(\phi \wedge \nec_0 \phi \wedge \bigwedge \Gamma \to \bigvee \Delta)\).
  By propositional reasoning
  \[
    \ER \vdash (\nec_0 \phi \to \phi) \wedge (\phi \wedge \nec_0 \phi \wedge \bigwedge \Gamma \to \bigvee \Delta) \to (\nec_0 \phi \wedge \bigwedge \Gamma \to \bigvee \Delta),
  \]
  so using \((\NEC[1])\) and \((\Kax[1])\) we obtain
  \[
    \ER \vdash \nec_1(\nec_0 \phi \to \phi) \wedge \nec_1(\phi \wedge \nec_0 \phi \wedge \bigwedge \Gamma \to \bigvee \Delta) \to \nec_1(\nec_0 \phi \wedge \bigwedge \Gamma \to \bigvee \Delta).
  \]
  Since \(\ER \vdash \nec_1(\nec_0 \phi \to \phi)\) by axiom \((\ERax[1,0])\) and \(\ER \vdash \nec_1(\phi \wedge \nec_0 \phi \wedge \bigwedge \Gamma \to \bigvee \Delta)\) we can conclude the desired \(\ER \vdash \nec_1(\nec_0 \phi \wedge \bigwedge \Gamma \to \bigvee \Delta)\).
\end{proof}

Notice that the correspondence shown in the next theorem is only partial.
In particular, it only works for sequents with the \( \Rightarrow_0\) arrow.
For the full correspondence,  i.e.~for all sequents, we will first need to prove cut elimination.
\begin{theorem}\label{th:correspondence-to-hilbert-ER-weak}
  Let \(S\) be a \(\Rightarrow_0\)-sequent.
  We have that \(\ER \vdash S^\flat\) iff \(\g{\ER} + \cut \vdash S\).
\end{theorem}
\begin{proof}
  The proof is completely analogous to Theorem~\ref{th:correspondence-to-hilbert-CS-CSM} but using Lemmas~\ref{lm:correspondence-to-hilbert-ER} and \ref{lm:soundness-of-rules-ER} instead of Lemmas~\ref{lm:correspondence-to-hilbert-CS-CSM} and \ref{lm:soundness-of-rules-CS-CSM}.
\end{proof}

\section{Cut elimination}\label{sec:cut-elim}

The methodology that we will use to show cut elimination for \(\g{\CS}\), \(\g{\CSM}\) and \(\g{\ER}\) will be local progress proof theory.
The advantage of this approach is that the calculi do not have diagonal formulas, which greatly simplifies the inductive measure needed in the cut elimination procedure.
In particular, contrary to the original proof \cite{Valentini}, this cut elimination does not need to make a deep analysis (i.e., follow the trace of a formula upwards) of the proofs above the cut rule; instead it suffices to make a case distinction on their last rules.\footnote{Another simplified cut elimination for \(\GL\) was obtained in \cite{Brighton, ian-cut-elim}, using a measure based on proof search. We prefer to show cut elimination in a non-wellfounded system as this can also be used to establish uniform Lyndon interpolation.}

If wanted,  this cut-elimination procedure could be made effective as it is only necessary to approximate a big enough portion of the non-wellfounded proof.
However, since we do not need the effectiveness of cut elimination for our other results, we will keep the exposition as simple as possible and disregard the effectiveness.

\subsection{Non-wellfounded calculi}

The connection of provability logics with non-wellfounded calculi is wellknown since the work of Savateev and Shamkanov \cite{shamkanovGl}.
As usual, non-wellfounded calculi for the logics \(\CS\), \(\CSM\) and \(\ER\) can be obtained from the wellfounded calculi by removing the diagonal formula in the modal rules.

The particular rules needed are shown in Figure~\ref{fig:non-wellfounded-rules}.
Again, the formula displayed in the conclusion is called the \emph{principal formula} (of the rule instance) and  the formulas in \(\nec_0 \Sigma_0\), \(\nec_1 \Sigma_1\) will be called  \emph{auxiliary formulas}.
 \(\Gamma\) and \(\Delta\) will be called the \emph{weakening formulas} of the rule instance, and  it still holds that arbitarily changing the weakening formulas of a rule instance still makes it a rule instance of the same rule.

\begin{figure}
  \[
    \AxiomC{\(\necd_i \Sigma_i, \nec_{\overline{i}} \Sigma_{\overline{i}} \Rightarrow \phi\)}
    \RightLabel{\(\modal[i]{\KTCS}\)}
    \UnaryInfC{\(\nec_i \Sigma_i, \nec_{\overline{i}} \Sigma_{\overline{i}}, \Gamma \Rightarrow \nec_i \phi, \Delta\)}
    \DisplayProof
  \]

  \[
    \AxiomC{\(\necd_0 \Sigma_0, \nec_{1} \Sigma_{1} \Rightarrow \phi\)}
    \RightLabel{\(\modal[0]{\KTCSM}\)}
    \UnaryInfC{\(\nec_0 \Sigma_0, \nec_{1} \Sigma_{1}, \Gamma \Rightarrow \nec_i \phi, \Delta\)}
    \DisplayProof
    \quad
    \AxiomC{\(\necd_0 \Sigma_0, \necd_{1} \Sigma_{1} \Rightarrow \phi\)}
    \RightLabel{\(\modal[1]{\KTCSM}\)}
    \UnaryInfC{\(\nec_0 \Sigma_0, \nec_{1} \Sigma_{1}, \Gamma \Rightarrow \nec_i \phi, \Delta\)}
    \DisplayProof
  \]

  \[
    \AxiomC{\(\necd_i \Sigma_i, \nec_{\overline{i}} \Sigma_{\overline{i}} \Rightarrow_i \phi\)}
    \RightLabel{\(\modal[i]{\KTER}\)}
    \UnaryInfC{\(\nec_i \Sigma_i, \nec_{\overline{i}} \Sigma_{\overline{i}}, \Gamma \gg \nec_i \phi, \Delta\)}
    \DisplayProof
  \]
  \caption{Non-wellfounded rules}
  \label{fig:non-wellfounded-rules}
\end{figure}

\begin{definition}
  We define the following local progress sequent calculi.
  \begin{enumerate}
    \item \(\n{\CS}\) is given by the rules of Figure~\ref{fig:prop-rules} together with the rules \((\modal[0]{\KTCS})\) and \((\modal[1]{\KTCS})\) from Figure~\ref{fig:non-wellfounded-rules}.
      Progress is only made at the premises of \((\modal[0]{\KTCS})\) and \((\modal[1]{\KTCS})\).
    \item \(\n{\CSM}\) is given by the rules of Figure~\ref{fig:prop-rules} together with the rules \((\modal[0]{\KTCSM})\) and \((\modal[1]{\KTCSM})\) from Figure~\ref{fig:non-wellfounded-rules}.
      Progress is only made at the premises of \((\modal[0]{\KTCSM})\) and \((\modal[1]{\KTCSM})\).
    \item \(\n{\ER}\) is given by the rules of Figure~\ref{fig:prop-rules-ER} together with the rule \((\ERax[1,0])\) of Figure~\ref{fig:modal-rules-ER} and the rules \((\modal[0]{\KTER})\) and \((\modal[1]{\KTER})\) from Figure~\ref{fig:non-wellfounded-rules}.
      Progress is only made at the premises of \((\modal[0]{\KTER})\) and \((\modal[1]{\KTER})\).
      \qedhere
  \end{enumerate}
\end{definition}

As usual, these calculi behave well structurally, which can be shown by a simple induction on the local height.

\begin{lemma}\label{lm:structural-rules-non-wellfounded}
  In \(\n{\CS}\), \(\n{\CSM}\), \(\n{\ER}\), \(\n{\CS} + \cut\), \(\n{\CSM} + \cut\), \(\n{\ER} + \cut\) we have the following:
  \begin{enumerate}
    \item \((\wk)\) is eliminable and admissible preserving local height and local rules;
    \item \((\botR)\), \((\toL)\) and \((\toR)\) are invertible preserving local height and local rules;
    \item \((\ctr)\) is eliminable and admissible preserving local height and local rules.
  \end{enumerate}
  In addition, \((\Rightarrow_1)\) is eliminable and admissible preserving local height and local rules in \(\n{\ER}\) and in \( \n{\ER} + \cut\).
\end{lemma}

\subsection{Translations}\label{subsec:translations}

The proof of cut elimination for the \(\g{L}\) calculi (where \(L \in \set{\CS, \CSM, \ER}\)) can be pictorically represented as follows:
% https://q.uiver.app/#q=WzAsNCxbMCwwLCJcXGd7TH0gKyBcXGN1dCJdLFsyLDAsIlxcbntMfSArIFxcY3V0Il0sWzIsMSwiXFxue0x9Il0sWzAsMSwiXFxne0x9Il0sWzAsMV0sWzEsMl0sWzIsM10sWzMsMF1d
\[\begin{tikzcd}
	{\g{L} + \cut} && {\n{L} + \cut} \\
	{\g{L}} && {\n{L}}
	\arrow[from=1-1, to=1-3]
	\arrow[from=1-3, to=2-3]
	\arrow[from=2-1, to=1-1]
	\arrow[from=2-3, to=2-1]
\end{tikzcd}\]
In this subsection we will define the arrows on the top and bottom of the picture, i.e., how to translate from \(\g{L} + \cut\) to \(\n{L} + \cut\) and from \(\n{L}\) to \(\g{L}\).
For the translation at the top we will use the admissibility of L\"ob's rule in each of the wellfounded calculi, as shown in the next lemma.

\begin{lemma}
  We have that
  \begin{enumerate}
    \item The rules
      \[
        \AxiomC{\(\necd_i \Sigma_i, \nec_{\overline{i}} \Sigma_{\overline{i}}, \nec_i \phi \Rightarrow \phi\)}
        \RightLabel{\(\lob[i]{\CS}\)}
        \UnaryInfC{\(\necd_i \Sigma_i, \nec_{\overline{i}} \Sigma_{\overline{i}}\Rightarrow \phi\)}
        \DisplayProof
      \]
      for \(i \leq 1\) are admissible in \(\g{\CS} + \cut\).
    \item The rules
      \[
        \AxiomC{\(\necd_0 \Sigma_0, \nec_{1} \Sigma_{1}, \nec_0 \phi \Rightarrow \phi\)}
        \RightLabel{\(\lob[0]{\CSM}\)}
        \UnaryInfC{\(\necd_0 \Sigma_0, \nec_{1} \Sigma_{1} \Rightarrow \phi\)}
        \DisplayProof
        \quad
        \AxiomC{\(\necd_0 \Sigma_0, \necd_{1} \Sigma_{1}, \nec_1 \phi \Rightarrow \phi\)}
        \RightLabel{\(\lob[1]{\CSM}\)}
        \UnaryInfC{\(\necd_0 \Sigma_0, \necd_{1} \Sigma_{1} \Rightarrow \phi\)}
        \DisplayProof
      \]
      are admissible in \(\g{\CSM} + \cut\).
    \item The rules
      \[
        \AxiomC{\(\necd_i \Sigma_i, \nec_{\overline{i}} \Sigma_{\overline{i}}, \nec_i \phi \Rightarrow_i \phi\)}
        \RightLabel{\(\lob[i]{\ER}\)}
        \UnaryInfC{\(\necd_i \Sigma_i, \nec_{\overline{i}} \Sigma_{\overline{i}}\Rightarrow_i \phi\)}
        \DisplayProof
      \]
      for \(i \leq 1\) are admissible in \(\g{\ER} +\cut\).
  \end{enumerate}
\end{lemma}
\begin{proof}
  Proof of 1.
  Let \(\pi \vdash \necd_i \Sigma_i, \nec_{\overline{i}} \Sigma_{\overline{i}}, \nec_i \phi \Rightarrow \phi\) in \(\g{\CS} + \cut\).
  The desired proof is
  \[
    \AxiomC{\(\pi\)}
    \noLine
    \UnaryInfC{\(\necd_i \Sigma_i, \nec_{\overline{i}} \Sigma_{\overline{i}}, \nec_i \phi \Rightarrow \phi\)}
    \RightLabel{\(\modal[i]{\CS}\)}
    \UnaryInfC{\(\necd_i \Sigma_i, \nec_{\overline{i}} \Sigma_{\overline{i}} \Rightarrow \phi, \nec_i\phi\)}
    \AxiomC{\(\pi\)}
    \noLine
    \UnaryInfC{\(\necd_i \Sigma_i, \nec_{\overline{i}} \Sigma_{\overline{i}}, \nec_i \phi \Rightarrow \phi\)}
    \RightLabel{\(\cut\)}
    \BinaryInfC{\(\necd_i \Sigma_i, \nec_{\overline{i}} \Sigma_{\overline{i}} \Rightarrow \phi\)}
    \DisplayProof
  \]

  Proof of 2.
  For \(\lob[0]{\CSM}\) we construct the same proof as in \(\CS\), let us show the admissibility of \(\lob[1]{\CSM}\).
  Assume we have the proof \(\pi \vdash \necd_0 \Sigma_0, \necd_{1} \Sigma_{1}, \nec_1 \phi \Rightarrow \phi\) in \(\g{\CSM} + \cut\).
  The desired proof is
  \[
    \AxiomC{\(\pi\)}
    \noLine
    \UnaryInfC{\(\necd_0 \Sigma_0, \necd_{1} \Sigma_{1}, \nec_1 \phi \Rightarrow \phi\)}
    \RightLabel{\(\modal[1]{\CSM}\)}
    \UnaryInfC{\(\necd_0 \Sigma_0, \necd_{1} \Sigma_{1} \Rightarrow \phi, \nec_1 \phi\)}
    \AxiomC{\(\pi\)}
    \noLine
    \UnaryInfC{\(\necd_0 \Sigma_0, \necd_{1} \Sigma_{1}, \nec_1 \phi \Rightarrow \phi\)}
    \RightLabel{\(\cut\)}
    \BinaryInfC{\(\necd_0 \Sigma_0, \necd_{1} \Sigma_{1} \Rightarrow \phi\)}
    \DisplayProof
  \]

  Proof of 3.
  The proofs are completely analogous to the case of \(\CS\).
\end{proof}

The translation is then defined via corecursion.

\begin{theorem}
  \label{th:from-wellfounded-to-nonwellfounded-bimodal}
  For any sequent \(S\) we have the following.
  \begin{enumerate}
    \item If \(\g{\CS} + \cut \vdash S\), then \(\n{\CS} + \cut \vdash S\).
    \item If \(\g{\CSM} + \cut \vdash S\), then \(\n{\CSM} + \cut \vdash S\).
    \item If \(\g{\ER} + \cut \vdash S\), then \(\n{\ER} + \cut \vdash S\).
  \end{enumerate}
\end{theorem}
\begin{proof}
  Let us write the proof for \(\CS\), for \(\CSM\) and \(\ER\) the proof is completely analogous.
  We define a function \(\alpha\) from proofs in \(\g{\CS} + \cut\) to preproofs in \(\n{\CS} + \cut\) by corecursion such that \(\alpha(\pi)\) has the same conclusion as \(\pi\).
  When the definition is finished we will justify that \(\alpha(\pi)\) is always a proof in \(\n{\CS} + \cut\).
  So let \(\pi\) be a proof in \(\g{\CS} + \cut\), we proceed by cases on the last rule \(R\) applied to \(\pi\).

  If \(R\) is one of \((\ax)\), \((\botL)\), \((\botR)\), \((\toL)\), \((\toR)\), \((\cut)\) then we define \(\alpha\) as follows:
  \[
    \AxiomC{\(\pi_0\)}
    \noLine
    \UnaryInfC{\(S_0\)}
    \AxiomC{\(\cdots\)}
    \AxiomC{\(\pi_{n-1}\)}
    \noLine
    \UnaryInfC{\(S_{n-1}\)}
    \RightLabel{\(R\)}
    \TrinaryInfC{\(S\)}
    \DisplayProof
    \quad
    \overset{\alpha}{\longmapsto}
    \quad
    \AxiomC{\(\alpha(\pi_0)\)}
    \noLine
    \UnaryInfC{\(S_0\)}
    \AxiomC{\(\cdots\)}
    \AxiomC{\(\alpha(\pi_{n-1})\)}
    \noLine
    \UnaryInfC{\(S_{n-1}\)}
    \RightLabel{\(R\)}
    \TrinaryInfC{\(S\)}
    \DisplayProof
  \]
  since \(R\) is also a rule of \(\n{\CS} + \cut\).
  If \(R\) is \((\modal[i]{\CS})\) then we define \(\alpha\) as follows
  \[
    \AxiomC{\(\pi_0\)}
    \noLine
    \UnaryInfC{\(\necd_i \Sigma_i, \nec_{\overline{i}} \Sigma_{\overline{i}}, \nec_i \phi \Rightarrow \phi\)}
    \RightLabel{\(\modal[i]{\CS}\)}
    \UnaryInfC{\(\nec_i \Sigma_i, \nec_{\overline{i}} \Sigma_{\overline{i}}, \Gamma \Rightarrow \nec_i \phi, \Delta\)}
    \DisplayProof
    \quad
    \overset{\alpha}{\longmapsto}
    \quad
    \AxiomC{\(\alpha(\lob[i]{\CS}(\pi_0))\)}
    \noLine
    \UnaryInfC{\(\necd_i \Sigma_i, \nec_{\overline{i}} \Sigma_{\overline{i}} \Rightarrow \phi\)}
    \RightLabel{\(\modal[i]{\KTCS}\)}
    \UnaryInfC{\(\nec_i \Sigma_i, \nec_{\overline{i}} \Sigma_{\overline{i}}, \Gamma \Rightarrow \nec_i \phi, \Delta\)}
    \DisplayProof
  \]
  If \(R\) is one of \((\ax)\), \((\botL)\), \((\botR)\), \((\toL)\), \((\toR)\), \((\cut)\) then the height of the corecursive calls is strictly lower than the height of the input.
  This is not the case if \(R\) is \((\modal[i]{\CS})\), but then there is an application of \(\modal[i]{\KTCS}\) from the root to the corecursive call.
  This imply that any infinite branch in \(\alpha(\pi)\) will have infinitely many applications of \(\modal[i]{\KTCS}\), so it will be a proof as desired.
\end{proof}

For the translation at the bottom of the picture, we will exploit the absence of \((\cut)\) and we will use that the subformula property holds in the cut-free calculi.

\begin{definition}
  Let \(\phi\) be a formula, we define the set of its subformulas as
  \begin{align*}
    &\sub(p) = \set{p},
    &&\sub(\bot) = \set{\bot}, \\
    &\sub(\phi_0 \to \phi_1) = \set{\phi_0 \to \phi_1} \union \sub(\phi_0) \union \sub(\phi_1),
    &&\sub(\nec_i \phi_0) = \set{\nec_i \phi_0} \union \sub(\phi_0).
  \end{align*}
  For any multiset \(\Gamma\) we will write \(\sub(\Gamma)\) to mean the set \(\Union_{\phi \in \Gamma} \sub(\phi)\) and for a (possibly \(\ER\)-)sequent \(\Gamma \gg \Delta\) we will write \(\sub(\Gamma \gg \Delta)\) to mean \(\sub(\Gamma) \union \sub(\Delta)\).
\end{definition}

The following lemma follows by inspection of the shape of the rules.

\begin{lemma}[Local subformula property]
  Let \(R\) be a rule of \(\g{\CS}\), \(\n{\CS}\), \(\g{\CSM}\), \(\n{\CSM}\), \(\g{\ER}\), or \(\n{\ER}\) and \((S_0,\ldots,S_{n-1},S) \in R\).
  Then \(\sub(S_i) \subseteq \sub(S)\) for all \(i < n\).
\end{lemma}

\begin{theorem}
  For any finite sets \(\Lambda_0, \Lambda_1\), we have that
  \begin{enumerate}
    \item \(\n{\CS} \vdash \Gamma \Rightarrow \Delta\) implies \(\g{\CS} \vdash \nec_0 \Lambda_0, \nec_1 \Lambda_1, \Gamma \Rightarrow \Delta\);
    \item \(\n{\CSM} \vdash \Gamma \Rightarrow \Delta\) implies \(\g{\CSM} \vdash \nec_0 \Lambda_0, \nec_1 \Lambda_1, \Gamma \Rightarrow \Delta\);
    \item \(\n{\ER} \vdash \Gamma \gg \Delta\) implies \(\g{\ER} \vdash \nec_0 \Lambda_0, \nec_1 \Lambda_1, \Gamma \gg \Delta\).
  \end{enumerate}
\end{theorem}
\begin{proof}
  We prove the statement for \(\ER\), the reasoning for \(\CS\) and \(\CSM\) is analogous.
  Let \(\pi \vdash \Gamma \gg \Delta\) in \(\n{\ER}\), we proceed by induction on the ordinal measure \(
  \omega(\sum_{i \leq 1} |\sub(\Gamma \gg \Delta) \setminus \Lambda_i|) + \lhg(\pi)\) and a case distinction on the last rule \(R\) applied in \(\pi\).
  If \(R\) is one of \((\ax)\), \((\botL)\), \((\botR)\), \((\toL)\), \((\toR)\) or \((\ERax[1,0])\), then it suffices to apply the induction hypothesis to the immediate subproofs and then apply \(R\).
  In more detail,  \(\pi\) has the form
  \[
    \AxiomC{\(\pi_0\)}
    \noLine
    \UnaryInfC{\(\Gamma_0 \gg \Delta_0\)}
    \AxiomC{\(\cdots\)}
    \AxiomC{\(\pi_{n-1}\)}
    \noLine
    \UnaryInfC{\(\Gamma_{n-1} \gg \Delta_{n-1}\)}
    \RightLabel{\(R\)}
    \TrinaryInfC{\(\Gamma \gg \Delta\)}
    \DisplayProof
  \]
  where \(\sub(\Gamma_i \gg \Delta_i) \subseteq \sub(\Gamma \gg \Delta)\) and \(\lhg(\pi_i) < \lhg(\pi)\).
  Then by the induction hypothesis there are proofs \(\tau_i \vdash \nec_0 \Lambda_0, \nec_1 \Lambda_1, \Gamma_i \gg \Delta_i\) for \(i < n\).
  The desired proof is
  \[
    \AxiomC{\(\tau_0\)}
    \noLine
    \UnaryInfC{\(\nec_0 \Lambda_0, \nec_1 \Lambda_1, \Gamma_0 \gg \Delta_0\)}
    \AxiomC{\(\cdots\)}
    \AxiomC{\(\tau_{n-1}\)}
    \noLine
    \UnaryInfC{\(\nec_0 \Lambda_0, \nec_1 \Lambda_1,\Gamma_{n-1} \gg \Delta_{n-1}\)}
    \RightLabel{\(R\)}
    \TrinaryInfC{\(\nec_0 \Lambda_0, \nec_1 \Lambda_1,\Gamma \gg \Delta\)}
    \DisplayProof
  \]

  If \(R\) is  \((\modal[0]{\KTER})\), then \(\pi\) is of the shape
  \[
    \AxiomC{\(\pi_0\)}
    \noLine
    \UnaryInfC{\(\necd_0 \Sigma_0, \nec_1 \Sigma_1 \Rightarrow_0 \phi\)}
    \RightLabel{\(\modal[0]{\KTER}\)}
    \UnaryInfC{\(\nec_0 \Sigma_0, \nec_1 \Sigma_1, \Gamma \gg \nec_0 \phi, \Delta\)}
    \DisplayProof
  \]
  where we denote the conclusion of \(\pi_0\) by \(S_0\) and the conclusion of \(\pi\) by \(S\).
  There are two cases, if \(\phi \in \Lambda_0\) then the desired proof is
  \[
    \AxiomC{}
    \RightLabel{\(\Ax\)}
    \UnaryInfC{\(\nec_0 \Lambda_0, \nec_1 \Lambda_1, \nec_0 \Sigma_0, \nec_1 \Sigma_1, \Gamma \gg \nec_0 \phi, \Delta\)}
    \DisplayProof
  \]
  If \(\phi \not \in \Lambda_0\), then we have  that \(|\sub(S_0) \setminus (\Lambda_0 \union \set{\phi})| < |\sub(S_0) \setminus \Lambda_0| \leq |\sub(S) \setminus \Lambda_0|\), where the second inequality holds because of  \(\sub(S_0) \subseteq \sub(S)\).
  For the same reason,  we  have  $|\sub(S_1) \setminus \Lambda_1| \leq |\sub(S) \setminus \Lambda_1|$. Hence,  we can apply the induction hypothesis with sets \(\Lambda_0 \union \set{\phi}, \Lambda_1\) to \(\pi_0\).
Thus, we obtain a proof 
$\tau_0 \vdash \nec_0 \Lambda_0, \nec_1 \Lambda_1, \necd_0 \Sigma_0, \nec_1 \Sigma_1, \nec_0 \phi \Rightarrow_0 \phi\) in \(\g{\ER}$.
  The desired proof is
  \[
    \AxiomC{\(\wk(\tau_0)\)}
    \noLine
    \UnaryInfC{\(\necd_0 \Lambda_0, \nec_1 \Lambda_1, \necd_0 \Sigma_0, \nec_1 \Sigma_1, \nec_0 \phi \Rightarrow_0 \phi\)}
    \RightLabel{\(\modal[0]{\ER}\)}
    \UnaryInfC{\(\nec_0 \Lambda_0, \nec_1 \Lambda_1, \nec_0 \Sigma_0, \nec_1 \Sigma_1, \Gamma \gg \nec_0 \phi, \Delta\)}
    \DisplayProof
  \]

  The case when \(R \text{ is } (\modal[1]{\KTER})\) is analogous to the previous case.
\end{proof}

\subsection{Local admissibility of cut}

In this subsection, we prove cut elimination for the systems \(\n{\CS}\), \(\n{\CSM}\) and \(\n{\ER}\).
Using the translations defined in Subsection~\ref{subsec:translations}, we obtain cut elimination for \(\g{\CS}\), \(\g{\CSM}\) and \(\g{\ER}\) as a corollary.

\begin{theorem}\label{th:cut-elim-CS}
  \(\cut\) is eliminable in \(\n{\CS}\).
  As a corollary \(\cut\) is eliminable in \(\g{\CS}\).
\end{theorem}
\begin{proof}
  We will show that \(\cut\) is locally admissible in \(\n{\CS}\).
  Let \(\pi \vdash \Gamma \Rightarrow \Delta, \chi\) and \(\tau \vdash \chi, \Gamma \Rightarrow \Delta\) in \(\n{\CS}\).
  We proceed by induction on the ordinal measure \(\omega|\chi| + (\lhg(\pi) + \lhg(\tau))\) and case analysis.
  Most cases can be found in Appendix~\ref{sec:cut-reductions}, here we only show the cases that are unique for this calculus.

  Assume \(\pi\) and \(\tau\) are of the following shape
  \[
    \AxiomC{\(\pi_0\)}
    \noLine
    \UnaryInfC{\(\necd_i \Sigma_i, \nec_{\overline{i}} \Sigma_{\overline{i}} \Rightarrow \chi_0\)}
    \RightLabel{\(\modal[i]{\KTCS}\)}
    \UnaryInfC{\(\nec_i \Sigma_i, \nec_{\overline{i}} \Sigma_{\overline{i}}, \Gamma' \Rightarrow \nec_i \chi_0, \nec_i \phi, \Delta'\)}
    \DisplayProof
    \quad
    \AxiomC{\(\tau_0\)}
    \noLine
    \UnaryInfC{\(\necd_i \chi_0, \necd_i \Sigma_i, \nec_{\overline{i}} \Sigma_{\overline{i}} \Rightarrow \phi\)}
    \RightLabel{\(\modal[i]{\KTCS}\)}
    \UnaryInfC{\(\nec_i \chi_0,\nec_i \Sigma_i, \nec_{\overline{i}} \Sigma_{\overline{i}}, \Gamma' \Rightarrow  \nec_{i} \phi, \Delta'\)}
    \DisplayProof
  \]
  where \(\chi = \nec_i \chi_0\), \(\Gamma = \nec_i \Sigma_i, \nec_{\overline{i}} \Sigma_{\overline{i}}, \Gamma'\), \(\Delta = \nec_i \phi, \Delta'\) and we assume without loss of generality that the auxiliary formulas of the last rule of \(\pi\) and \(\tau\) are equal, as otherwise we can use weakening on \(\pi_0\) and \(\tau_0\) to make them equal by making the rule instances maximal\footnote{A rule instance of a modal rule is called \emph{maximal} if none of the weakening fomrulas is a \(\nec\)-formula.}.
  Note that this step is permited since it does not alter the local height and mantains the local cut-freeness of \(\pi\) and \(\tau\).
  The desired locally cut-free proof is
  \[
    \AxiomC{\(\wk(\pi_0)\)}
    \noLine
    \UnaryInfC{\(\necd_i \Sigma_i, \nec_{\overline{i}} \Sigma_{\overline{i}} \Rightarrow \phi, \chi_0 \)}
    \AxiomC{\(\pi_0\)}
    \noLine
    \UnaryInfC{\(\necd_i \Sigma_i, \nec_{\overline{i}} \Sigma_{\overline{i}} \Rightarrow \chi_0\)}
    \RightLabel{\(\modal[i]{\KTCS}\)}
    \UnaryInfC{\(\chi_0, \necd_i \Sigma_i, \nec_{\overline{i}} \Sigma_{\overline{i}} \Rightarrow \phi, \nec_i \chi_0\)}
    \AxiomC{\(\tau_0\)}
    \noLine
    \UnaryInfC{\(\necd_i \chi_0, \necd_i \Sigma_i, \nec_{\overline{i}} \Sigma_{\overline{i}} \Rightarrow \phi\)}
    \RightLabel{\(\cut\)}
    \BinaryInfC{\(\chi_0, \necd_i \Sigma_i, \nec_{\overline{i}} \Sigma_{\overline{i}} \Rightarrow \phi\)}
    \RightLabel{\(\cut\)}
    \BinaryInfC{\(\necd_i \Sigma_i, \nec_{\overline{i}} \Sigma_{\overline{i}} \Rightarrow \phi\)}
    \RightLabel{\(\modal[i]{\KTCS}\)}
    \UnaryInfC{\(\nec_i \Sigma_i, \nec_{\overline{i}} \Sigma_{\overline{i}}, \Gamma' \Rightarrow  \nec_{i} \phi, \Delta'\)}
    \DisplayProof
  \]

  Assume \(\pi\) and \(\tau\) are of the following shape
  \[
    \AxiomC{\(\pi_0\)}
    \noLine
    \UnaryInfC{\(\necd_i \Sigma_i, \nec_{\overline{i}} \Sigma_{\overline{i}} \Rightarrow \chi_0\)}
    \RightLabel{\(\modal[i]{\KTCS}\)}
    \UnaryInfC{\(\nec_i \Sigma_i, \nec_{\overline{i}} \Sigma_{\overline{i}}, \Gamma' \Rightarrow \nec_i \chi_0, \nec_{\overline{i}} \phi, \Delta'\)}
    \DisplayProof
    \quad
    \AxiomC{\(\tau_0\)}
    \noLine
    \UnaryInfC{\(\nec_i \chi_0, \nec_i \Sigma_i, \necd_{\overline{i}} \Sigma_{\overline{i}} \Rightarrow \phi\)}
    \RightLabel{\(\modal[\overline{i}]{\KTCS}\)}
    \UnaryInfC{\(\nec_i \chi_0,\nec_i \Sigma_i, \nec_{\overline{i}} \Sigma_{\overline{i}}, \Gamma' \Rightarrow  \nec_{\overline{i}} \phi, \Delta'\)}
    \DisplayProof
  \]
  where \(\chi = \nec_i \chi_0\), \(\Gamma = \nec_i \Sigma_i, \nec_{\overline{i}} \Sigma_{\overline{i}}, \Gamma'\), \(\Delta = \nec_{\overline{i}} \phi, \Delta'\) and we assume again without loss of generality that the auxiliary formulas of the last rule of \(\pi\) and \(\tau\) are equal.
  The desired locally cut-free proof is
  \[
    \AxiomC{\(\pi_0\)}
    \noLine
    \UnaryInfC{\(\necd_i \Sigma_i, \nec_{\overline{i}} \Sigma_{\overline{i}} \Rightarrow \chi_0\)}
    \RightLabel{\(\modal[i]{\KTCS}\)}
    \UnaryInfC{\(\nec_i \Sigma_i, \necd_{\overline{i}} \Sigma_{\overline{i}} \Rightarrow \phi, \nec_i \chi_0 \)}
    \AxiomC{\(\tau_0\)}
    \noLine
    \UnaryInfC{\(\nec_i \chi_0, \nec_i \Sigma_i, \necd_{\overline{i}} \Sigma_{\overline{i}} \Rightarrow \phi\)}
    \RightLabel{\(\cut\)}
    \BinaryInfC{\(\nec_i \Sigma_i, \necd_{\overline{i}} \Sigma_{\overline{i}} \Rightarrow \phi\)}
    \RightLabel{\(\modal[\overline{i}]{\KTCS}\)}
    \UnaryInfC{\(\nec_i \Sigma_i, \nec_{\overline{i}} \Sigma_{\overline{i}}, \Gamma' \Rightarrow  \nec_{\overline{i}} \phi, \Delta'\)}
    \DisplayProof
  \]
\end{proof}

\begin{theorem}\label{th:cut-elim-CSM}
  \(\cut\) is eliminable in \(\n{\CSM}\).
  As a corollary \(\cut\) is eliminable in \(\g{\CSM}\).
\end{theorem}
\begin{proof}
  We will show that \(\cut\) is locally admissible in \(\n{\CS}\).
  Let \(\pi \vdash \Gamma \Rightarrow \Delta, \chi\) and \(\tau \vdash \chi, \Gamma \Rightarrow \Delta\) in \(\n{\CS}\).
  We proceed by induction on the ordinal measure \(\omega|\chi| + (\lhg(\pi) + \lhg(\tau))\) and by a case analysis.
  Most cases can be found in Appendix~\ref{sec:cut-reductions}, here we only show the cases that are unique for this calculus.

  Assume \(\pi\) and \(\tau\) are of the following shape
  \[
    \AxiomC{\(\pi_0\)}
    \noLine
    \UnaryInfC{\(\necd_0 \Sigma_0, \nec_1 \Sigma_1 \Rightarrow  \chi_0\)}
    \RightLabel{\(\modal[0]{\KTCSM}\)}
    \UnaryInfC{\(\nec_0 \Sigma_0, \nec_1 \Sigma_1, \Gamma' \Rightarrow \nec_0 \chi_0, \nec_0 \phi, \Delta'\)}
    \DisplayProof
    \quad
    \AxiomC{\(\tau_0\)}
    \noLine
    \UnaryInfC{\(\necd_0 \chi_0, \necd_0 \Sigma_0, \nec_1 \Sigma_1 \Rightarrow  \phi, \)}
    \RightLabel{\(\modal[0]{\KTCSM}\)}
    \UnaryInfC{\(\nec_0 \chi_0, \nec_0 \Sigma_0, \nec_1 \Sigma_1, \Gamma' \Rightarrow  \nec_0 \phi, \Delta'\)}
    \DisplayProof
  \]
  where \(\chi = \nec_0 \chi_0\), \(\Gamma = \nec_0 \Sigma_0, \nec_1 \Sigma_1, \Gamma'\), \(\Delta = \nec_0 \phi, \Delta'\) and we assume without loss of generality that the auxiliary formulas of the last rule of \(\pi\) and \(\tau\) are equal, as otherwise we can use weakening on \(\pi_0\) and \(\tau_0\) to make them equal by making the rule instances maximal.
  Note that this step is permited since it does not alter the local height and mantains the local cut-freeness of \(\pi\) and \(\tau\).
  The desired locally cut-free proof is
  \[
    \AxiomC{\(\wk(\pi_0)\)}
    \noLine
    \UnaryInfC{\(\necd_0 \Sigma_0, \nec_1 \Sigma_1 \Rightarrow  \phi, \chi_0\)}
    \AxiomC{\(\pi_0\)}
    \noLine
    \UnaryInfC{\(\necd_0 \Sigma_0, \nec_1 \Sigma_1 \Rightarrow  \chi_0\)}
    \RightLabel{\(\modal[0]{\KTCSM}\)}
    \UnaryInfC{\(\chi_0, \necd_0 \Sigma_0, \nec_1 \Sigma_1 \Rightarrow  \phi, \nec_0 \chi_0\)}
    \AxiomC{\(\tau_0\)}
    \noLine
    \UnaryInfC{\(\necd_0 \chi_0, \necd_0 \Sigma_0, \nec_1 \Sigma_1 \Rightarrow  \phi \)}
    \RightLabel{\(\cut\)}
    \BinaryInfC{\(\chi_0, \necd_0 \Sigma_0, \nec_1 \Sigma_1 \Rightarrow  \phi \)}
    \RightLabel{\(\cut\)}
    \BinaryInfC{\( \necd_0 \Sigma_0, \nec_1 \Sigma_1 \Rightarrow  \phi \)}
    \RightLabel{\(\modal[0]{\KTCSM}\)}
    \UnaryInfC{\(\nec_0 \Sigma_0, \nec_1 \Sigma_1, \Gamma' \Rightarrow  \nec_0 \phi, \Delta'\)}
    \DisplayProof
  \]

  Assume \(\pi\) and \(\tau\) are of the following shape
  \[
    \AxiomC{\(\pi_0\)}
    \noLine
    \UnaryInfC{\(\necd_0 \Sigma_0, \nec_1 \Sigma_1 \Rightarrow \chi_0\)}
    \RightLabel{\(\modal[0]{\KTCSM}\)}
    \UnaryInfC{\(\nec_0 \Sigma_0, \nec_1 \Sigma_1, \Gamma' \Rightarrow \nec_0 \chi_0, \nec_1 \phi, \Delta'\)}
    \DisplayProof
    \quad
    \AxiomC{\(\tau_0\)}
    \noLine
    \UnaryInfC{\(\necd_0 \chi_0, \necd_0 \Sigma_0, \necd_1 \Sigma_1 \Rightarrow   \phi\)}
    \RightLabel{\(\modal[1]{\KTCSM}\)}
    \UnaryInfC{\(\nec_0 \chi_0, \nec_0 \Sigma_0, \nec_1 \Sigma_1, \Gamma' \Rightarrow  \nec_1 \phi, \Delta'\)}
    \DisplayProof
  \]
  where \(\chi = \nec_0 \chi_0\), \(\Gamma = \nec_0 \Sigma_0, \nec_1 \Sigma_1, \Gamma'\), \(\Delta = \nec_1 \phi, \Delta\) and we assume  without loss of generality that the auxiliary formulas of the last rule of \(\pi\) and \(\tau\) are equal, as otherwise we can use weakening on \(\pi_0\) and \(\tau_0\) to make them equal by making the rule instances maximal.
  The desired locally cut-free proof is
  \[
    \AxiomC{\(\wk(\pi_0)\)}
    \noLine
    \UnaryInfC{\(\necd_0 \Sigma_0, \necd_1 \Sigma_1 \Rightarrow   \phi, \chi_0\)}
    \AxiomC{\(\pi_0\)}
    \noLine
    \UnaryInfC{\(\necd_0 \Sigma_0, \nec_1 \Sigma_1 \Rightarrow \chi_0\)}
    \RightLabel{\(\modal[0]{\KTCSM}\)}
    \UnaryInfC{\(\chi_0, \necd_0 \Sigma_0, \necd_1 \Sigma_1 \Rightarrow   \phi, \nec_0 \chi_0\)}
    \AxiomC{\(\tau_0\)}
    \noLine
    \UnaryInfC{\(\necd_0 \chi_0, \necd_0 \Sigma_0, \necd_1 \Sigma_1 \Rightarrow   \phi\)}
    \RightLabel{\(\cut\)}
    \BinaryInfC{\(\chi_0, \necd_0 \Sigma_0, \necd_1 \Sigma_1 \Rightarrow   \phi\)}
    \RightLabel{\(\cut\)}
    \BinaryInfC{\(\necd_0 \Sigma_0, \necd_1 \Sigma_1 \Rightarrow   \phi\)}
    \RightLabel{\(\modal[1]{\KTCSM}\)}
    \UnaryInfC{\(\nec_0 \Sigma_0, \nec_1 \Sigma_1, \Gamma' \Rightarrow  \nec_1 \phi, \Delta'\)}
    \DisplayProof
  \]

  Assume \(\pi\) and \(\tau\) are of the following shape
  \[
    \AxiomC{\(\pi_0\)}
    \noLine
    \UnaryInfC{\(\necd_0 \Sigma_0, \necd_1 \Sigma_1 \Rightarrow \chi_0\)}
    \RightLabel{\(\modal[1]{\KTCSM}\)}
    \UnaryInfC{\(\nec_0 \Sigma_0, \nec_1 \Sigma_1, \Gamma' \Rightarrow \nec_1 \chi_0, \nec_0 \phi, \Delta'\)}
    \DisplayProof
    \quad
    \AxiomC{\(\tau_0\)}
    \noLine
    \UnaryInfC{\(\nec_1 \chi_0, \necd_0 \Sigma_0, \nec_1 \Sigma_1 \Rightarrow \phi\)}
    \RightLabel{\(\modal[0]{\KTCSM}\)}
    \UnaryInfC{\(\nec_1 \chi_0, \nec_0 \Sigma_0, \nec_1 \Sigma_1, \Gamma' \Rightarrow  \nec_0 \phi, \Delta'\)}
    \DisplayProof
  \]
  where \(\Gamma = \nec_0 \Sigma_0, \nec_1 \Sigma_1, \Gamma'\), \(\Delta = \nec_0 \phi, \Delta'\) and we assumed (again) without loss of generality that the auxiliary formulas of the last rule of \(\pi\) and \(\tau\) are equal, as otherwise we can use weakening on \(\pi_0\) and \(\tau_0\) to make them equal by making the rule instances maximal.
  The desired locally cut-free proof is
  \[
    \AxiomC{\(\pi_0\)}
    \noLine
    \UnaryInfC{\(\necd_0 \Sigma_0, \necd_1 \Sigma_1 \Rightarrow \chi_0\)}
    \RightLabel{\(\modal[1]{\KTCSM}\)}
    \UnaryInfC{\(\necd_0 \Sigma_0, \nec_1 \Sigma_1 \Rightarrow \phi, \nec_1 \chi_0\)}
    \AxiomC{\(\tau_0\)}
    \noLine
    \UnaryInfC{\(\nec_1 \chi_0, \necd_0 \Sigma_0, \nec_1 \Sigma_1 \Rightarrow \phi\)}
    \RightLabel{\(\cut\)}
    \BinaryInfC{\(\necd_0 \Sigma_0, \nec_1 \Sigma_1 \Rightarrow \phi\)}
    \RightLabel{\(\modal[0]{\KTCSM}\)}
    \UnaryInfC{\(\nec_0 \Sigma_0, \nec_1 \Sigma_1, \Gamma' \Rightarrow  \nec_0 \phi, \Delta'\)}
    \DisplayProof
  \]

  Assume \(\pi\) and \(\tau\) are of the following shape
  \[
    \AxiomC{\(\pi_0\)}
    \noLine
    \UnaryInfC{\(\necd_0 \Sigma_0, \necd_1 \Sigma_1 \Rightarrow \chi_0\)}
    \RightLabel{\(\modal[1]{\KTCSM}\)}
    \UnaryInfC{\(\nec_0 \Sigma_0, \nec_1 \Sigma_1, \Gamma' \Rightarrow \nec_1 \chi_0, \nec_1 \phi, \Delta'\)}
    \DisplayProof
    \quad
    \AxiomC{\(\tau_0\)}
    \noLine
    \UnaryInfC{\(\necd_1 \chi_0, \necd_0 \Sigma_0, \necd_1 \Sigma_1 \Rightarrow  \phi\)}
    \RightLabel{\(\modal[1]{\KTCSM}\)}
    \UnaryInfC{\(\nec_1 \chi_0, \nec_0 \Sigma_0, \nec_1 \Sigma_1, \Gamma' \Rightarrow  \nec_1 \phi, \Delta'\)}
    \DisplayProof
  \]
  where \(\Gamma = \nec_0 \Sigma_0, \nec_1 \Sigma_1, \Gamma'\), \(\Delta = \nec_1 \phi, \Delta'\) and we assumed (again) without loss of generality that the auxiliary formulas of the last rule of \(\pi\) and \(\tau\) are equal, as otherwise we can use weakening on \(\pi_0\) and \(\tau_0\) to make them equal by making the rule instances maximal.
  The desired locally cut-free proof is
  \[
    \AxiomC{\(\wk(\pi_0)\)}
    \noLine
    \UnaryInfC{\(\necd_0 \Sigma_0, \necd_1 \Sigma_1 \Rightarrow  \phi, \chi_0\)}
    \AxiomC{\(\pi_0\)}
    \noLine
    \UnaryInfC{\(\necd_0 \Sigma_0, \necd_1 \Sigma_1 \Rightarrow \chi_0\)}
    \RightLabel{\(\modal[1]{\KTCSM}\)}
    \UnaryInfC{\(\chi_0, \necd_0 \Sigma_0, \necd_1 \Sigma_1 \Rightarrow  \phi, \nec_1 \chi_0\)}
    \AxiomC{\(\tau_0\)}
    \noLine
    \UnaryInfC{\(\necd_1 \chi_0, \necd_0 \Sigma_0, \necd_1 \Sigma_1 \Rightarrow  \phi\)}
    \RightLabel{\(\cut\)}
    \BinaryInfC{\(\chi_0, \necd_0 \Sigma_0, \necd_1 \Sigma_1 \Rightarrow  \phi\)}
    \RightLabel{\(\cut\)}
    \BinaryInfC{\(\necd_0 \Sigma_0, \necd_1 \Sigma_1 \Rightarrow  \phi\)}
    \RightLabel{\(\modal[1]{\KTCSM}\)}
    \UnaryInfC{\(\nec_0 \Sigma_0, \nec_1 \Sigma_1, \Gamma' \Rightarrow  \nec_1 \phi, \Delta'\)}
    \DisplayProof
  \]
\end{proof}

\begin{theorem}\label{th:cut-elim-ER}
  \(\cut\) is eliminable in \(\n{\ER}\).
  As a corollary \(\cut\) is eliminable in \(\g{\ER}\).
\end{theorem}
\begin{proof}
  We will show that \(\cut\) is locally admissible in \(\n{\CS}\).
  Let \(\pi \vdash \Gamma \gg \Delta, \chi\) and \(\tau \vdash \chi, \Gamma \gg \Delta\) in \(\n{\CS}\).
  We proceed by induction on the ordinal measure \(\omega|\chi| + (\lhg(\pi) + \lhg(\tau))\) and a case analysis.
  Again,  most cases can be found in Appendix~\ref{sec:cut-reductions}, here we only show the cases that are unique for this calculus.

  Assume \(\pi\) and \(\tau\) are of the following shape
  \[
    \AxiomC{\(\pi_0\)}
    \noLine
    \UnaryInfC{\(\necd_0 \Sigma_0, \nec_{1} \Sigma_{1} \Rightarrow_0 \chi_0 \)}
    \RightLabel{\(\modal[0]{\KTER}\)}
    \UnaryInfC{\(\nec_0 \Sigma_0, \nec_{1} \Sigma_{1}, \Gamma' \Rightarrow_1 \nec_0 \chi_0, \Delta \)}
    \DisplayProof
    \quad
    \AxiomC{\(\tau_0\)}
    \noLine
    \UnaryInfC{\(\necd_0 \chi_0,\nec_0 \Sigma_0, \nec_{1} \Sigma_{1}, \Gamma' \Rightarrow_1  \Delta \)}
    \RightLabel{\(\ERax[1,0]\)}
    \UnaryInfC{\(\nec_0 \chi_0, \nec_0 \Sigma_0, \nec_{1} \Sigma_{1}, \Gamma' \Rightarrow_1  \Delta \)}
    \DisplayProof
  \]
  where \(\chi = \nec_0 \chi_0\), \(\Gamma = \nec_0 \Sigma_0, \nec_{1} \Sigma_{1}, \Gamma'\).
  The desired locally cut-free proof is
  {
    \def\defaultHypSeparation{\hskip 2pt} %could be used to make rules less wide
    \[
      \AxiomC{\( {\Rightarrow_1}(\wk(\pi_0))\)}
      \noLine
      \UnaryInfC{\(\Sigma_0,\nec_{i \leq 1} \Sigma_i, \Gamma' \Rightarrow_1 \Delta, \chi_0 \)}
      \doubleLine
      \RightLabel{\(\ERax[1,0]\)}
      \UnaryInfC{\(\nec_{i \leq 1} \Sigma_i, \Gamma' \Rightarrow_1 \Delta, \chi_0 \)}
      \AxiomC{\(\pi_0\)}
      \noLine
      \UnaryInfC{\(\Sigma_0, \nec_{i \leq 1} \Sigma_i, \Rightarrow_0 \chi_0 \)}
      \RightLabel{\(\modal[0]{\KTER}\)}
      \UnaryInfC{\(\chi_0,\nec_{i \leq 1} \Sigma_i, \Gamma' \Rightarrow_1  \Delta, \nec_0 \chi_0\)}
      \AxiomC{\(\tau_0\)}
      \noLine
      \UnaryInfC{\(\necd_0 \chi_0,\nec_{i \leq 1} \Sigma_i, \Gamma' \Rightarrow_1  \Delta \)}
      \RightLabel{\(\cut\) (I.H.)}
      \BinaryInfC{\(\chi_0,\nec_{i \leq 1} \Sigma_i, \Gamma' \Rightarrow_1  \Delta \)}
      \RightLabel{\(\cut\) (I.H.)}
      \BinaryInfC{\(\nec_{i \leq 1} \Sigma_i, \Gamma' \Rightarrow_1  \Delta \)}
      \DisplayProof
    \]
  }
  where \(\nec_{i \leq 1} \Sigma_i = \nec_0 \Sigma_0, \nec_1 \Sigma_1\) and the first use of the induction hypothesis is thanks to reducing the sum of local heights while the second use of the induction hypothesis is thanks to lowering the complexity of the cut formula.

  Assume \(\pi\) and \(\tau\) are of the following shape
  \[
    \AxiomC{\(\pi_0\)}
    \noLine
    \UnaryInfC{\(\necd_i \Sigma_i, \nec_{\overline{i}} \Sigma_{\overline{i}} \Rightarrow_i \chi_0\)}
    \RightLabel{\(\modal[i]{\KTER}\)}
    \UnaryInfC{\(\nec_i \Sigma_i, \nec_{\overline{i}} \Sigma_{\overline{i}}, \Gamma' \gg \nec_i \chi_0, \nec_i \phi, \Delta'\)}
    \DisplayProof
    \quad
    \AxiomC{\(\tau_0\)}
    \noLine
    \UnaryInfC{\(\necd_i \chi_0, \necd_i \Sigma_i, \nec_{\overline{i}} \Sigma_{\overline{i}} \Rightarrow_i \phi\)}
    \RightLabel{\(\modal[i]{\KTER}\)}
    \UnaryInfC{\(\nec_i \chi_0,\nec_i \Sigma_i, \nec_{\overline{i}} \Sigma_{\overline{i}}, \Gamma' \gg  \nec_{i} \phi, \Delta'\)}
    \DisplayProof
  \]
  where \(\chi = \nec_i \chi_0\), \(\Gamma = \nec_i \Sigma_i, \nec_{\overline{i}} \Sigma_{\overline{i}}, \Gamma'\), \(\Delta = \nec_i \phi, \Delta'\) and we assumed  without loss of generality that the auxiliary formulas of the last rule of \(\pi\) and \(\tau\) are equal, as otherwise we can use weakening on \(\pi_0\) and \(\tau_0\) to make them equal by making the rule instances maximal.
  The desired locally cut-free proof is
  \[
    \AxiomC{\(\wk(\pi_0)\)}
    \noLine
    \UnaryInfC{\(\necd_i \Sigma_i, \nec_{\overline{i}} \Sigma_{\overline{i}} \Rightarrow_i \phi, \chi_0 \)}
    \AxiomC{\(\pi_0\)}
    \noLine
    \UnaryInfC{\(\necd_i \Sigma_i, \nec_{\overline{i}} \Sigma_{\overline{i}} \Rightarrow_i \chi_0\)}
    \RightLabel{\(\modal[i]{\KTCS}\)}
    \UnaryInfC{\(\chi_0, \necd_i \Sigma_i, \nec_{\overline{i}} \Sigma_{\overline{i}} \Rightarrow_i \phi, \nec_i \chi_0\)}
    \AxiomC{\(\tau_0\)}
    \noLine
    \UnaryInfC{\(\necd_i \chi_0, \necd_i \Sigma_i, \nec_{\overline{i}} \Sigma_{\overline{i}} \Rightarrow_i \phi\)}
    \RightLabel{\(\cut\)}
    \BinaryInfC{\(\chi_0, \necd_i \Sigma_i, \nec_{\overline{i}} \Sigma_{\overline{i}} \Rightarrow_i \phi\)}
    \RightLabel{\(\cut\)}
    \BinaryInfC{\(\necd_i \Sigma_i, \nec_{\overline{i}} \Sigma_{\overline{i}} \Rightarrow_i \phi\)}
    \RightLabel{\(\modal[i]{\KTCS}\)}
    \UnaryInfC{\(\nec_i \Sigma_i, \nec_{\overline{i}} \Sigma_{\overline{i}}, \Gamma' \gg  \nec_{i} \phi, \Delta'\)}
    \DisplayProof
  \]

  Assume \(\pi\) and \(\tau\) are of the following shape
  \[
    \AxiomC{\(\pi_0\)}
    \noLine
    \UnaryInfC{\(\necd_i \Sigma_i, \nec_{\overline{i}} \Sigma_{\overline{i}} \Rightarrow_i \chi_0\)}
    \RightLabel{\(\modal[i]{\KTCS}\)}
    \UnaryInfC{\(\nec_i \Sigma_i, \nec_{\overline{i}} \Sigma_{\overline{i}}, \Gamma' \gg \nec_i \chi_0, \nec_{\overline{i}} \phi, \Delta'\)}
    \DisplayProof
    \quad
    \AxiomC{\(\tau_0\)}
    \noLine
    \UnaryInfC{\(\nec_i \chi_0, \nec_i \Sigma_i, \necd_{\overline{i}} \Sigma_{\overline{i}} \Rightarrow_{\overline{i}} \phi\)}
    \RightLabel{\(\modal[\overline{i}]{\KTCS}\)}
    \UnaryInfC{\(\nec_i \chi_0,\nec_i \Sigma_i, \nec_{\overline{i}} \Sigma_{\overline{i}}, \Gamma' \gg  \nec_{\overline{i}} \phi, \Delta'\)}
    \DisplayProof
  \]
  where \(\chi = \nec_i \chi_0\), \(\Gamma = \nec_i \Sigma_i, \nec_{\overline{i}} \Sigma_{\overline{i}}, \Gamma'\), \(\Delta = \nec_{\overline{i}} \phi, \Delta'\) and we assume again without loss of generality that the auxiliary formulas of the last rule of \(\pi\) and \(\tau\) are equal.
  The desired locally cut-free proof is
  \[
    \AxiomC{\(\pi_0\)}
    \noLine
    \UnaryInfC{\(\necd_i \Sigma_i, \nec_{\overline{i}} \Sigma_{\overline{i}} \Rightarrow_i \chi_0\)}
    \RightLabel{\(\modal[i]{\KTCS}\)}
    \UnaryInfC{\(\nec_i \Sigma_i, \necd_{\overline{i}} \Sigma_{\overline{i}} \Rightarrow_{\overline{i}} \phi, \nec_i \chi_0 \)}
    \AxiomC{\(\tau_0\)}
    \noLine
    \UnaryInfC{\(\nec_i \chi_0, \nec_i \Sigma_i, \necd_{\overline{i}} \Sigma_{\overline{i}} \Rightarrow_{\overline{i}} \phi\)}
    \RightLabel{\(\cut\)}
    \BinaryInfC{\(\nec_i \Sigma_i, \necd_{\overline{i}} \Sigma_{\overline{i}} \Rightarrow_{\overline{i}} \phi\)}
    \RightLabel{\(\modal[\overline{i}]{\KTCS}\)}
    \UnaryInfC{\(\nec_i \Sigma_i, \nec_{\overline{i}} \Sigma_{\overline{i}}, \Gamma' \gg  \nec_{\overline{i}} \phi, \Delta'\)}
    \DisplayProof
  \]
\end{proof}

The cut elimination for \(\ER\) allows us to improve the result of Theorem~\ref{th:correspondence-to-hilbert-ER-weak} to all sequents.

\begin{theorem}\label{th:correspondence-to-hilbert-ER}
  For any \(\ER\)-sequent \(S\), we have that \(\ER \vdash S^\flat\) iff \(\g{\ER} \vdash S\).
\end{theorem}
\begin{proof}
  The if direction is proven by a simple induction in the height of the proof in \(\g{\ER}\) using Lemma~\ref{lm:soundness-of-rules-ER}, so let us show the only if direction.
  If \(S\) is a \( \Rightarrow_0\)-sequent, then the result follows from Theorem~\ref{th:correspondence-to-hilbert-ER-weak}.  Hence assume \(S = \Gamma \Rightarrow_1 \Delta\).
  From \(\ER \vdash \nec_1(\bigwedge \Gamma \to \bigvee \Delta)\), we obtain by Lemma~\ref{lm:correspondence-to-hilbert-ER} that \(\g{\ER} + \cut \vdash { \Rightarrow_0 \nec_1(\bigwedge \Gamma \to \bigvee \Delta)}\).
  Then by one of the translations of Subsection~\ref{subsec:translations} and cut elimination in \(\n{\ER}\), we obtain that \(\n{\ER} \vdash { \Rightarrow_0 \nec_1 (\bigwedge \Gamma \to \bigvee \Delta)}\).
  By the shape of the rules of \(\n{\ER}\) (in particular the proof of the previous sequent must end in an application of \(\modal[i]{\KTER}\)), this implies that \(\n{\ER} \vdash { \Rightarrow_1 \bigwedge \Gamma \to \bigvee \Delta}\) and hence, by invertibility, \(\n{\ER} \vdash { \Gamma \Rightarrow_1 \Delta}\).
  Again, by one of the translations of Subsection~\ref{subsec:translations} we can conclude  \(\g{\ER} \vdash {\Gamma \Rightarrow_1 \Delta}\).
\end{proof}

We finalize the section with a corollary which sums up all the results we have until now.

\begin{corollary}
  Let \(S\) be a sequent (when working with \(\ER\) an \(\ER\)-sequent).
  We have that
  \begin{enumerate}
    \item \(\CS \vdash S^\flat\) iff \(\g{\CS}\; (+ \cut) \vdash S\) iff \(\n{\CS}\; (+ \cut) \vdash S\);
    \item \(\CSM \vdash S^\flat\) iff \(\g{\CSM}\; (+ \cut) \vdash S\) iff \(\n{\CSM}\; (+ \cut) \vdash S\);
    \item \(\ER \vdash S^\flat\) iff \(\g{\ER}\; (+ \cut) \vdash S\) iff \(\n{\ER}\; (+ \cut) \vdash S\).
  \end{enumerate}
\end{corollary}

\section{Interpolation}
\label{sec:interpolation}

Having at our disposal the non-wellfounded calculi and the cut elimination theorem, we are prepared to show uniform Lyndon interpolation for \(\CS\), \(\CSM\) and \(\ER\).

For the uniformity,  we will use a modification of proof search which we call \emph{interpolation templates}.
The idea of an interpolation template as follows:
perform a proof search of the sequent at its root, but take into account that parts of the sequent may be missing.
For example, when evaluating which modal rules can be applied to a sequent \(\nec_0 \Sigma_0, \nec_1 \Sigma_1, \Gamma \Rightarrow \nec_0 \Theta_0, \nec_1 \Theta_1, \Delta\) in \(\CS\),  we have to take into account that the principal formula may be \emph{outside} \(\nec_0 \Theta_0, \nec_1 \Theta_1, \Delta\).
This will provoke that in addition to the usual possible premise,  we need to add two extra ones: \(\necd_0 \Sigma_0, \nec_1 \Sigma_1 \Rightarrow \) in case the principal formula is outside and it is a \(\nec_0\)-formula, and \(\nec_0 \Sigma_0, \necd_1 \Sigma_1 \Rightarrow\) in case the principal formula is outside and it is a \(\nec_1\)-formula.

To get not only the uniform property but also the Lyndon property for the interpolants, the use of the non-wellfounded calculi is essential.
This is because the diagonal formula in the wellfounded calculi ruins the polarity of the principal formula in the modal rules of \(\g{\CS}\), \(\g{\CSM}\) and \(\g{\ER}\).
The non-wellfounded calculi,  not having a diagonal formula,  behave well with respect to the polarities.

Finally, let us comment that the proof of uniform Lyndon interpolation for \(\CSM\) is really simple: we will use that \(\CSM\) can be interpreted in \(\CS\) and that \(\CS\) has uniform Lyndon interpolation.
It is an open question whether a similar interpretation of \(\ER\) in \(\CSM\) or \(\CS\) is possible.
%This will save a lot of space, and we leave as an open question is such an interpretation of \(\ER\) in \(\CSM\) or \(\CS\) is possible.

%With all this considerations in mind, we jump to the proofs.

\subsection{Uniform Lyndon Interpolation for \(\CS\) and \(\CSM\)}

We start by defining the notion of interpolation template for \(\CS\).
Notice that, even if it is based on an non-wellfounded system, we need the interpolation template to be a finite object, i.e., a cyclic proof and not any non-wellfounded proof.
Otherwise, the construction of the interpolant 
%(being the interpolant a finite, wellfounded, formula) 
would be hard if not impossible.

\begin{figure}
  \[
    \AxiomC{\(\)}
    \RightLabel{\(\emp\)}
    \UnaryInfC{\( \Rightarrow \)}
    \DisplayProof
    \qquad
    \AxiomC{\([\necd_i \Sigma^s_i, \nec_{\overline{i}} \Sigma^s_{\overline{i}} \Rightarrow]_{i \leq 1}\)}
    \AxiomC{\([\necd_i \Sigma^s_i, \nec_{\overline{i}} \Sigma^s_{\overline{i}} \Rightarrow\phi]_{i \leq 1, \phi \in \Theta_i}\)}
    \RightLabel{\(\modal[+]{\KTCS}\)}
    \BinaryInfC{\(\nec_0 \Sigma_0, \nec_1 \Sigma_1, \Gamma \Rightarrow \nec_0 \Theta_0, \nec_1 \Theta_1, \Delta\)}
    \DisplayProof
  \]
  where \(\Gamma, \Delta \subseteq \var\) and \(\Gamma \inter \Delta = \varnothing\).

  \caption{Interpolation template rules for \(\CS\)}
  \label{fig:interpolation-template-CS}
\end{figure}

\begin{definition}
  A \emph{\(\CS\)-interpolation template} for \(\Gamma \Rightarrow \Delta\) is a cyclic proof of \(\Gamma \Rightarrow \Delta\) in the local progress sequent calculus with the propositional rules of Figure~\ref{fig:prop-rules} and the rules of Figure~\ref{fig:interpolation-template-CS} where progress is only made at the premises of \((\modal[+]{\KTCS})\).
\end{definition}

We can show that any sequent has an interpolation template.
For that purpose let \(\Gamma \Rightarrow \Delta\) be a sequent, define \(|\Gamma \Rightarrow \Delta| = \sum_{\phi \in \Gamma,\Delta} |\phi|\).
We will use this measure in the following lemma.

\begin{lemma}\label{existence-of-interpolation-templates-CS}
  Every sequent has a \(\CS\)-interpolation template.
\end{lemma}
\begin{proof}
  Notice that if \(R\) is a rule of a \(\CS\)-interpolation templates and \((S_0,\ldots,S_{n-1},S)\) is an instance of \(R\) then \(\sub(S_i) \subseteq \sub(S)\) for \(i < n\), so any preproof generated with rules of \(\CS\)-interpolation templates contains only formulas which are subformulas of the sequent at the root.
  In addition, if \(R\) is distinct from \((\modal[+]{\KTCS})\) then \(|S_i| < |S|\) for \(i < n\).

  Also, notice that for any sequent \(\Gamma \Rightarrow \Delta\),  there is always one rule of interpolation templates that can be applied to it.
  Then we proceed as follows to construct an interpolation template for \(\Gamma \Rightarrow \Delta\):
  \begin{enumerate}
    \item Stage 0. Start by putting \(\Gamma \Rightarrow \Delta\) at the root of the tree without a rule.
    \item Stage \(n+1\). For each leaf \(w\) at stage \(n\) without a rule do one of the following (in order):
      \begin{enumerate}
        \item If there is a \(v < w\) with the same sequent, make \(w\) a repeat node pointing to \(v\).
        \item Otherwise, select a rule \(R\) and a rule instance \((S_0,\ldots,S_{n-1},S) \in R\) such that \(w\) is annotated with sequent \(S\). Annotate \(w\) with the rule \(R\) and create new nodes \(w0\), \ldots, \(w(n-1)\) immediate successors of \(w\) annotated with sequent \(S_0,\ldots,S_{n-1}\) respectively.
      \end{enumerate}
  \end{enumerate}
  First, we show that this process finishes.

  Suppose otherwise, then at the limit we would have an non-wellfounded tree with at least one infinite branch \((w_i)_{i \in \mathbb{N}}\).
  Let \(w_i\) be annotated with sequent \(S_i\) and rule \(R_i\), notice that all the \(S_i\)s are pairwise different, as otherwise the infinite branch would have been closed during the construction with a repeat.
  If there are only finitely many \(i\)s such that \(R_i = (\modal[+]{\KTCS})\) then there is a maximum \(i_0\) such that for any \(j > i_0\) we have that \(R_i \neq (\modal[+]{\KTCS})\).
  Then \(|S_{i_0}| > |S_{i_0 + 1}| > \cdots > |S_{i_0 + k}| > \cdots\), which is impossible.
  So there are infinitely many \(i\)s such that \(R_i = (\modal[+]{\KTCS})\), let \(\set{i_j}_{j \in \mathbb{N}}\) be those \(i\)s.
  Let \(m\) be the cardinality of \(\sub(\Gamma \Rightarrow \Delta)\), thanks to the subformula property and since premises of \((\modal[+]{\KTCS})\) are determined by two sets of formulas (\(\Sigma^s_0\) and \(\Sigma^s_1\)) and a possible choice of formula (\(\phi\)), we have that the possible number of premises of \((\modal[+]{\KTCS})\) in \(T\) is \(4^m(m+1) = 2^m2^m(m+1)\).
  However, in \((w_i)_{i \in \mathbb{N}}\) there are infinitely many premises of \((\modal[+]{\KTCS})\), so at least there is a repeated sequent, a contradiction.

  The created tree \(T\) is clearly a cyclic preproof.
  Given a node \(w\) of \(T\) let \(S_w\) be the sequent at \(w\) and \(R_w\) the rule at \(w\).
  Asumme \(w \in \rep(T)\) such that for each \(v \in (w^\circ,w]\) is non-progressing.
  Then for each \(v \in [w^\circ,w)\), \(R_v \neq (\modal[+]{\KTCS})\) so \(|S_{w^\circ}| > |S_w| = |S_{w^\circ}|\), a contradiction.
\end{proof}

We fix vocabularies \(V_+\) and \(V_-\) for which we want to calculate the uniform Lyndon interpolant.
Once we have the interpolation template we will divide the construction of the interpolant in two parts.
First, we will calculate a formula at each node \(w\) of the interpolation template which we call \emph{preinterpolant at \(w\)}.
Then, using the preinterpolants, we will create a equational system which will be solvable in \(\CS\).
The interpolant will then be obtained by applying the substitution which solves the equational system to the preinterpolant at the root.

We have no way of interpolating repeat nodes, so in principle we would get stuck in them.
However, due to the uniqueness of interpolants we know that the interpolant at a repeat node must be logically equivalent to the interpolant at its cyclic companion, in other words, we obtain an equation.
Solving all the equations that appear in the interpolation template will just be solving an equational system.

We start with the definitions of preinterpolants.

\begin{definition}
  Let \(T\) be an interpolation template and \(w\) a node of \(T\), let us denote the sequent at \(w\) as \(\Gamma_w \Rightarrow \Delta_w\).
  We will annotate each of the sequents of \(T\) with a formula \(\kappa\) called the \emph{preinterpolant at \(\kappa\)}, denoted as \(\kappa : \Gamma_w \Rightarrow \Delta_w\).
  We proceed by induction on the tree structure of \(T\).
  If \(w\) is a repeat node of shape
  \[
    \AxiomC{\(\)}
    \RightLabel{\(\rep\)}
    \UnaryInfC{\(\Gamma_w \Rightarrow \Delta_w\)}
    \DisplayProof
  \]
  then \(\kappa_w\) will be a fresh propositional variable \(x_w\).
  This variable shall not appear in any formula of \(T\) and for repeat nodes \(w,v\) such that \(w \neq v\) we must have that \(x_w \neq x_v\).
  In case \(w\) is not a repeat node,  we proceed by cases on the rule at \(w\). We give the construction as follows (the preinterpolant at the conclusion is defined recursively from the preinterpolants at the premises):
  \[
    \AxiomC{\(\)}
    \RightLabel{\(\ax\)}
    \UnaryInfC{\(\bot : p, \Gamma \Rightarrow p ,\Delta\)}
    \DisplayProof
    \qquad
    \AxiomC{\(\)}
    \RightLabel{\(\emp\)}
    \UnaryInfC{\(\top : {\Rightarrow}\)}
    \DisplayProof
    \qquad
    \AxiomC{\(\)}
    \RightLabel{\(\botL\)}
    \UnaryInfC{\(\bot : p, \Gamma \Rightarrow p ,\Delta\)}
    \DisplayProof
    \qquad
    \AxiomC{\(\kappa : \Gamma \Rightarrow \Delta\)}
    \RightLabel{\(\botR\)}
    \UnaryInfC{\(\kappa : \Gamma \Rightarrow \bot, \Delta\)}
    \DisplayProof
  \]

  \[
    \AxiomC{\(\kappa_0 : \Gamma \Rightarrow \phi, \Delta\)}
    \AxiomC{\(\kappa_1 : \psi, \Gamma \Rightarrow \Delta\)}
    \RightLabel{\(\toL\)}
    \BinaryInfC{\(\kappa_0 \vee \kappa_1 : \phi \to \psi, \Gamma \Rightarrow \Delta\)}
    \DisplayProof
    \qquad
    \AxiomC{\(\kappa : \phi, \Gamma \Rightarrow \psi, \Delta\)}
    \RightLabel{\(\toR\)}
    \UnaryInfC{\(\kappa : \Gamma \Rightarrow \phi \to \psi, \Delta\)}
    \DisplayProof
  \]

  \[
    \AxiomC{\(\kappa^{\nec_i} : \necd_i \Sigma^s_i, \nec_{\overline{i}}\Sigma^s_{\overline{i}} \Rightarrow\)}
    \AxiomC{\([\kappa^{\nec_i}_\phi :\necd_i \Sigma^s, \nec_{\overline{i}} \Sigma^s_i \Rightarrow\phi]_{i \leq 1, \phi \in \Theta_i}\)}
    \RightLabel{\(\modal[+]{\KTCS}\)}
    \BinaryInfC{\( \kappa  :\nec_0 \Sigma_0, \nec_1 \Sigma_1, \Gamma \Rightarrow \nec_0 \Theta_0, \nec_1 \Theta_1, \Delta\)}
    \DisplayProof
  \]
  where
  \[
    \kappa := \bigwedge_{i \leq 1} \left(\nec_i \kappa^{\nec_i} \wedge \bigwedge\limits_{\phi \in \Theta_i} \pos_i \kappa^{\nec_i}_\phi\right)\wedge \bigwedge (\Gamma \inter V_+) \wedge \bigwedge \neg(\Delta \inter V_-)
  \]
\end{definition}

Let \(T\) be an interpolation template, we proceed to define the equational system attached to \(T\). 
For each node \(w\) let \(\kappa_w\) be the preinterpolant at \(w\).
Remember that the \emph{height of a node \(w\)}, denoted \(\hg(w)\), is the height of the subtree generated at \(w\).
Then leafs have height \(0\) and the height of the root is the same as the height of the tree.
For each \(i \in \mathbb{N}\) let \(\bar{x}_i\) be an arbitrary enumeration of \(\set{x_{w} \mid w \in \rep(T), \hg(w^\circ) = i}\).
If \(\hg(T) = H\), then we define \(\bar{x}_T = \bar{x}_0 \bar{x}_1 \cdots \bar{x}_H\), \(\bar{x}_T\) is an anumeration of \(\set{x_w \mid w \in \rep(T)}\).
We define the \emph{equational system of \(T\)}, denoted \(\mathcal{E}_T\), as
\[
  \set{(x_w, +, \kappa_{w}) \mid w \in \rep(T)}.
\]
In the following lemma we show that it \(\mathcal{E}_T\) is solvable in \(\CS\).

\begin{lemma}\label{equational-systems-are-solvable-CS}
  \(\mathcal{E}_T\) is a positive modalized Lyndon equational system over \((\bar{x}_T, V_+, V_-)\).
  As a corollary, \(\mathcal{E}_T\) has a solution in \(\CS\).
\end{lemma}
\begin{proof}
  For each node \(w\) of \(T\) let \(\kappa_w\) be the preinterpolant at \(w\).
  By induction on the height of \(w\), we can show that \(\voc_+(\kappa_w) \subseteq V_+ \union \set{x_v \mid v \in \rep(T)}\) and \(\voc_-(\kappa_w) \subseteq V_-\), so it is clearly positive.

  It remains to show that it is modalized. For this it suffices to prove that for every \(w,v \in \rep(T)\) if \(\hg(v^\circ) \leq \hg(w^\circ) \), then \(\kappa_w\) is modalized in \(x_v\).
  We have the following facts which are easy to establish:
  \begin{enumerate}
    \item for each \(v \in \rep(T)\) and node \(w\) if \(x_v\) occurs at \(\kappa_w\) then \(w \leq v\),
    \item for each node \(w\) with rule \((\modal[+]{\KTCS})\) we have that \(\kappa_v\) is modalized in \(\set{x_w \mid w \in \rep(T)}\),
    \item for each \(v \in \rep(T)\) and node \(w \not \in \rep(T)\) with immediate successors \(w_0,\ldots,w_{n-1}\), if \(\kappa_{w_0}\), \ldots, \(\kappa_{w_{n-1}}\) are modalized in \(x_v\) then \(\kappa_w\) is modalized in \(v\).
  \end{enumerate}
  If \(\hg(v^\circ) \leq \hg(w^\circ)\) we have that either \(v^\circ \leq w^\circ\) or \(v^\circ\) and \(w^\circ\) are incomparable.
  Assume \(v^\circ\) and \(w^\circ\) are incomparable.
  Then it must also be the case that \(v\) and \(w^\circ\) are incomparable, so \(\kappa_{w^\circ}\) is modalized in \(x_v\) as \(x_v\) does not occur in \(\kappa_{w^\circ}\).
  Finally, assume that \(v^\circ \leq w^\circ\).
  We know there must be a node \(u \in [v^\circ, v)\) which is annotated with \(\modal[+]{\KTCS}\) so \(\kappa_u\) is modalized in \(x_v\).
  Then, using the third fact, we can show that (by induction) the preinterpolant of any node below \(u\) is modalized in \(x_v\) (since it is not a repeat node and all the preinterpolants of the immediate successors are also modalized in \(x_v\), either by virtue of the node not being comparable with \(v\) or by the induction hypothesis).
  We can conclude, as \(w^\circ \leq v^\circ \leq u\), that \(\kappa_{w^\circ}\) is modalized in \(x_v\).
\end{proof}

Finally, we have all the necessary tools to define the interpolant.

\begin{definition}[Interpolant]
  Given an interpolation template \(T\) we define the interpolant of \(T\), denoted \(\iota_T\), as the formula obtained by applying the substitution solving \(\mathcal{E}_T\) to the preinterpolant at the root of \(T\).
\end{definition}

Strictly speaking the definition of the interpolant is not unique.
For an interpolation template \(T\), \(\iota_T\) will depend on the choice of the solution of \(T\).
However, any two interpolants will be \(\CS\)-logically equivalent, so this lack of unicity should not worry us.
\label{uniqueness-of-interpolant-CS}

With the interpolant defined there is only one task left, we need to verify that it truly fulfills the necessary interpolation properties.
Notice that \(\voc_+(\iota_T) \subset V_+\) and \(\voc_-(\iota_T)\) holds from having applied a solution of an equational system over \((\bar{x}_T,V_+,V_-)\) to a formula \(\kappa\) with \(\voc_+(\kappa) \subseteq V_+ \union \set{x_w \mid w \in \rep(T)}\) and \(\voc_-(\kappa) \subseteq V_-\) together with Lemma~\ref{substitution-and-polarity}.
The other two properties of interpolants will be proven thanks to the following two theorems.
\label{vocabularies-interpolant-CS}

\begin{theorem}\label{th:first-verification-CS}
  Let \(T\) be an interpolation template for \(\Gamma \Rightarrow \Delta\).
  Then \(\n{\CS} \vdash \Gamma \Rightarrow \Delta, \iota_T\).
\end{theorem}
\begin{proof}
  Given a node \(w\) of \(T\) let us denote the sequent at \(w\) as \(\Gamma_w \Rightarrow \Delta_w\), the preinterpolant of \(T\) at \(w\) as \(\kappa_w\) and assume \((\cdot)^*\) is a solution of \(\mathcal{E}_T\).
  We define a function \(\alpha\) such that if \(w\) is a node of \(T\) then \(\alpha(w)\) is a preproof in \(\n{\CS} + \cut + \wk\) of \(\Gamma_w \Rightarrow \Delta_w, \kappa^*_w\).
  Then we will show that \(\alpha(w)\) is in fact a proof in \(\n{\CS} + \cut + \wk\) and then using that \(\wk\) is eliminable in \(\n{\CS} + \cut\) and that \(\cut\) is eliminable in \(\n{\CS}\) we obtain the desired proof in \(\n{\CS}\).

  We proceed to define \(\alpha\) by corecursion and cases.
  First, assume \(w \in \rep(T)\) then \(w^\circ\) has the same sequent as \(w\) and \(\kappa_w = x_w\).
  Also, by virtue of \((\cdot)^*\) being a solution of \(\mathcal{E}_T\), we have that \(\CS \vdash x^*_w \leftrightarrow \kappa^*_{w^\circ}\).
  Then we define \(\alpha\) as
  \[
    \AxiomC{}
    \RightLabel{\(\rep\)}
    \UnaryInfC{\(x_w : \Gamma \Rightarrow \Delta\)}
    \DisplayProof
    \quad
    \overset{\alpha}{\longmapsto}
    \quad
    \AxiomC{\(\alpha(w^\circ)\)}
    \noLine
    \UnaryInfC{\(\Gamma \Rightarrow \Delta, \kappa^*_w\)}
    \RightLabel{\(\wk\)}
    \UnaryInfC{\( \Gamma \Rightarrow \Delta, x^*_w, \kappa^*_w\)}
    \AxiomC{\(\tau\)}
    \noLine
    \UnaryInfC{\(\kappa^*_w,\Gamma \Rightarrow \Delta, x^*_w, \)}
    \RightLabel{\(\cut\)}
    \BinaryInfC{\(\Gamma \Rightarrow \Delta, x^*_w\)}
    \DisplayProof
  \]
  where the proof \(\tau\) exists since \(\CS \vdash x^*_w \leftrightarrow \kappa^*_{w^\circ}\).

  Now assume \(w\) is not a repetition node and let \(R\) be the rule at \(w\).
  Below there are the definitions when \(R\) is \((\ax)\), \((\emp)\), \((\botL)\), \((\botR)\), \((\toL)\) or \((\toR)\).
  \[
    \AxiomC{}
    \RightLabel{\(\ax\)}
    \UnaryInfC{\(\bot : p, \Gamma \Rightarrow p, \Delta\)}
    \DisplayProof
    \quad
    \overset{\alpha}{\longmapsto}
    \quad
    \AxiomC{}
    \RightLabel{\(\ax\)}
    \UnaryInfC{\(p, \Gamma \Rightarrow p, \Delta, \bot\)}
    \DisplayProof
  \]

  \[
    \AxiomC{}
    \RightLabel{\(\emp\)}
    \UnaryInfC{\(\top : {\Rightarrow}\)}
    \DisplayProof
    \quad
    \overset{\alpha}{\longmapsto}
    \quad
    \AxiomC{}
    \RightLabel{\(\botL\)}
    \UnaryInfC{\(\bot \Rightarrow \bot\)}
    \RightLabel{\(\toR\)}
    \UnaryInfC{\( \Rightarrow \top\)}
    \DisplayProof
  \]

  \[
    \AxiomC{}
    \RightLabel{\(\botL\)}
    \UnaryInfC{\(\bot : \bot, \Gamma \Rightarrow \Delta\)}
    \DisplayProof
    \quad
    \overset{\alpha}{\longmapsto}
    \quad
    \AxiomC{}
    \RightLabel{\(\botL\)}
    \UnaryInfC{\(\bot \Gamma \Rightarrow \Delta, \bot\)}
    \DisplayProof
  \]

  \[
    \AxiomC{\(w0\)}
    \noLine
    \UnaryInfC{\(\kappa : \Gamma \Rightarrow \Delta\)}
    \RightLabel{\(\botR\)}
    \UnaryInfC{\(\kappa : \Gamma \Rightarrow \bot, \Delta\)}
    \DisplayProof
    \quad
    \overset{\alpha}{\longmapsto}
    \quad
    \AxiomC{\(\alpha(w0)\)}
    \noLine
    \UnaryInfC{\(\Gamma \Rightarrow \Delta, \kappa^*\)}
    \RightLabel{\(\botR\)}
    \UnaryInfC{\(\Gamma \Rightarrow \bot, \Delta, \kappa^*\)}
    \DisplayProof
  \]

  \[
    \AxiomC{\(w0\)}
    \noLine
    \UnaryInfC{\(\kappa_0 : \Gamma \Rightarrow \phi, \Delta\)}
    \AxiomC{\(w1\)}
    \noLine
    \UnaryInfC{\(\kappa_1 : \psi, \Gamma \Rightarrow \Delta\)}
    \BinaryInfC{\(\kappa_0 \vee \kappa_1 : \phi \to \psi, \Gamma \Rightarrow \Delta\)}
    \DisplayProof
    \quad
    \overset{\alpha}{\longmapsto}
    \quad
    \AxiomC{\(w0\)}
    \noLine
    \UnaryInfC{\(\Gamma \Rightarrow \phi, \Delta, \kappa^*_0\)}
    \RightLabel{\(\wk\)}
    \UnaryInfC{\(\Gamma \Rightarrow \phi, \Delta, \kappa^*_0, \kappa^*_1\)}
    \RightLabel{\(\veeR\)}
    \UnaryInfC{\(\Gamma \Rightarrow \phi, \Delta, \kappa^*_0 \vee \kappa^*_1\)}
    \AxiomC{\(w1\)}
    \noLine
    \UnaryInfC{\(\psi, \Gamma \Rightarrow \Delta, \kappa^*_1\)}
    \RightLabel{\(\wk\)}
    \UnaryInfC{\(\Gamma \Rightarrow \phi, \Delta, \kappa^*_0, \kappa^*_1\)}
    \RightLabel{\(\veeR\)}
    \UnaryInfC{\(\Gamma \Rightarrow \phi, \Delta, \kappa^*_0 \vee \kappa^*_1\)}
    \BinaryInfC{\(\phi \to \psi, \Gamma \Rightarrow \Delta, \kappa^*_0 \vee \kappa^*_1\)}
    \DisplayProof
  \]

  \[
    \AxiomC{\(w0\)}
    \noLine
    \UnaryInfC{\(\kappa : \phi, \Gamma \Rightarrow \psi, \Delta\)}
    \RightLabel{\(\toR\)}
    \UnaryInfC{\(\kappa : \Gamma \Rightarrow \phi \to \psi, \Delta\)}
    \DisplayProof
    \quad
    \overset{\alpha}{\longmapsto}
    \quad
    \AxiomC{\(\alpha(w0)\)}
    \noLine
    \UnaryInfC{\(\phi, \Gamma \Rightarrow \psi, \Delta, \kappa^*\)}
    \RightLabel{\(\toR\)}
    \UnaryInfC{\(\Gamma \Rightarrow \phi \to \psi, \Delta, \kappa^*\)}
    \DisplayProof
  \]

  Finally, let us define \(\alpha(w)\) when \(R = (\modal[+]{\KTCS})\).
  Then \(w\) is of shape
  \[
    \AxiomC{\(\left[\begin{matrix} w^{\nec_i} \\ \kappa^{\nec_i} : \necd_i \Sigma^s_i, \nec_{\overline{i}} \Sigma^s_{\overline{i}} \Rightarrow\end{matrix}\right]_{i \leq 1}\)}
    \AxiomC{\(\left[\begin{matrix} w^{\nec_i}_\phi \\ \kappa^{\nec_i}_\phi :\necd_i \Sigma^s, \nec_{\overline{i}} \Sigma^s_i \Rightarrow\phi\end{matrix}\right]_{i \leq 1, \phi \in \Theta_i}\)}
    \RightLabel{\(\modal[+]{\KTCS}\)}
    \BinaryInfC{\( \kappa_w  :\nec_0 \Sigma_0, \nec_1 \Sigma_1, \Gamma \Rightarrow \nec_0 \Theta_0, \nec_1 \Theta_1, \Delta\)}
    \DisplayProof
  \]
  where
  \[
\kappa^*_{w} := \bigwedge_{i \leq 1} \left(\nec_i (\kappa^{\nec_i})^* \wedge \bigwedge_{\phi \in \Theta_i} \pos_i (\kappa^{\nec_i}_\phi)^*\right)\wedge \bigwedge (\Gamma \inter V_+) \wedge \bigwedge \neg(\Delta \inter V_-).
  \]
  Since \(\kappa^*_w\) is a conjunction of formulas it suffices to construct a preproof for each conjunct and join them with~\((\wedgeR)\).
  \begin{itemize}
    \item Preproof for \(\nec_i \kappa^{\nec_i}\) for \(i \leq 1\).
      The desired preproof is
      \[
        \AxiomC{\(\alpha(w^{\nec_i})\)}
        \noLine
        \UnaryInfC{\(\necd_i \Sigma^s_i, \nec_{\overline{i}} \Sigma^s_{\overline{i}} \Rightarrow (\kappa^{\nec_i})^*\)}
        \RightLabel{\(\modal[i]{\KTCS}\)}
        \UnaryInfC{\(\nec_0 \Sigma_0, \nec_1 \Sigma_1, \Gamma \Rightarrow \nec_0 \Theta_0, \nec_1 \Theta_1, \Delta, \nec_i (\kappa^{\nec_i})^*\)}
        \DisplayProof
      \]
    \item Preproof for \(\pos_i \kappa^{\nec_i}_\phi\) for \(i \leq 1\) and \(\phi \in \Theta_i\).
      The desired preproof is
      \[
        \AxiomC{\(\alpha(w^{\nec_i}_\phi)\)}
        \noLine
        \UnaryInfC{\(\necd_i \Sigma^s_i, \nec_{\overline{i}} \Sigma^s_{\overline{i}} \Rightarrow \phi, (\kappa^{\nec_i}_\phi)^*\)}
        \RightLabel{\(\negL\)}
        \UnaryInfC{\(\neg (\kappa^{\nec_i}_\phi)^*, \necd_i \Sigma^s_i, \nec_{\overline{i}} \Sigma^s_{\overline{i}} \Rightarrow \phi \)}
        \RightLabel{\(\wk\)}
        \UnaryInfC{\(\necd_i \neg (\kappa^{\nec_i}_\phi)^*, \necd_i \Sigma^s_i, \nec_{\overline{i}} \Sigma^s_{\overline{i}} \Rightarrow \phi \)}
        \RightLabel{\(\modal[i]{\KTCS}\)}
        \UnaryInfC{\(\nec_i \neg (\kappa^{\nec_i}_\phi)^*, \nec_0 \Sigma_0, \nec_1 \Sigma_1, \Gamma \Rightarrow \nec_0 \Theta_0, \nec_1 \Theta_1, \Delta, \nec_i \phi\)}
        \RightLabel{\(\negR\)}
        \UnaryInfC{\( \nec_0 \Sigma_0, \nec_1 \Sigma_1, \Gamma \Rightarrow \nec_0 \Theta_0, \nec_1 \Theta_1, \Delta, \nec_i \phi, \pos_i (\kappa^{\nec_i}_\phi)^*\)}
        \DisplayProof
      \]
    \item Preproof for \(p \in \Gamma \inter V_+\) and for \(\neg q\) where \(q \in \Delta \inter V_-\).
      The desired preproof, respectively, are
      \[
        \AxiomC{\(\)}
        \RightLabel{\(\ax\)}
        \UnaryInfC{\(\nec_0 \Sigma_0, \nec_1 \Sigma, \Gamma \Rightarrow \nec_0 \Theta_0, \nec_1 \Theta_1, \Delta, p\)}
        \DisplayProof
        \quad
        \AxiomC{\(\)}
        \RightLabel{\(\ax\)}
        \UnaryInfC{\(q, \nec_0 \Sigma_0, \nec_1 \Sigma, \Gamma \Rightarrow \nec_0 \Theta_0, \nec_1 \Theta_1, \Delta\)}
        \RightLabel{\(\negR\)}
        \UnaryInfC{\(\nec_0 \Sigma_0, \nec_1 \Sigma, \Gamma \Rightarrow \nec_0 \Theta_0, \nec_1 \Theta_1, \Delta, \neg q\)}
        \DisplayProof
      \]
  \end{itemize}

  Finally, let us argue that \(\alpha(w)\) is always a proof and not only a preproof.
  Given a node \(w\) assign to it the measure \(\omega |\Gamma_w \Rightarrow \Delta_w| + \lgth(w)\) where \(\lgth(w)\) is the length of \(w\) as a sequence of natural numbers.
  We notice that this measure always decreases from \(\alpha(w)\) to its corecursive call, except when \(w\) is annotated with the rule \((\modal[+]{\KTCS})\).
  However, in this case we will find that progress is made (i.e., there is an application of \((\modal[i]{\KTCS})\)) from the root of the preproof fragment given by \(\alpha(w)\) to the corecursive calls.
  This implies that any infinite branch will have infinitely many applications of \((\modal[i]{\KTCS})\), as otherwise we will infinitely decrease the measure on corecursive calls.
\end{proof}

\begin{theorem}\label{th:second-verification-CS}
  Let \(T\) be an interpolation template for \(\Gamma \Rightarrow \Delta\) and \(\Xi \Rightarrow \Lambda\) be a sequent such that \(\voc_b(\Xi \Rightarrow \Lambda) \subseteq V_b\) for \(b \in \set{+,-}\).
  Then \(\n{\CS} \vdash \Gamma, \Xi \Rightarrow \Delta, \Lambda\) implies \(\n{\CS} \vdash \iota_T, \Xi \Rightarrow \Lambda\).
\end{theorem}
\begin{proof}
  Given a node \(w\) of \(T\) let us denote the sequent at \(w\) as \(\Gamma_w \Rightarrow \Delta_w,\), the preinterpolant at \(w\) as \(\kappa_w\) and assume \((\cdot)^*\) is a solution of \(\mathcal{E}_T\).
  We define a function \(\beta\) such that if \(w\) is a node of \(T\) and \(\pi\) is a proof in \(\n{\CS}\) of \(\Gamma_w, \Xi \Rightarrow \Delta_w, \Lambda\) where \(\Xi \Rightarrow \Lambda\) is a sequent with \(\voc_b(\Xi \Rightarrow \Lambda) \subseteq V_b\) for \(b \in \set{+,-}\), then \(\beta(w,\pi)\) is a preproof in \(\n{\CS} + \cut + \wk\) of \(\kappa^*_w, \Xi \Rightarrow \Lambda\).
  Then we will show that \(\beta(w,\pi)\) is in fact a proof in \(\n{\CS} + \cut + \wk\) and then using that \(\wk\) is eliminable in \(\n{\CS}\) and that \(\cut\) is eliminable in \(\n{\CS}\) we obtian the desired proof in \(\n{\CS}\).

  We proceed to define \(\beta\) by corecursion and cases.
  First, assume \(w \in \rep(T)\) then \(w^\circ\) has the same sequent as \(w\) and \(\kappa_w = x_w\).
  Also, by virtue of \((\cdot)^*\) being a solution of \(\mathcal{E}_T\), we have that \(\CS \vdash x^*_w \leftrightarrow \kappa^*_{w^\circ}\).
  Then we define \(\beta\) as
  \[
    \left(
      \AxiomC{}
      \RightLabel{\(\rep\)}
      \UnaryInfC{\(x_w : \Gamma \Rightarrow \Delta\)}
      \DisplayProof
      ,\quad
      \AxiomC{\(\pi\)}
      \noLine
      \UnaryInfC{\(\Gamma, \Xi \Rightarrow \Delta, \Lambda\)}
      \DisplayProof
    \right)
    \quad
    \overset{\beta}{\longmapsto}
    \quad
    \AxiomC{\(\tau\)}
    \noLine
    \UnaryInfC{\(x^*_w, \Xi \Rightarrow \Lambda, \kappa^*_{w^\circ}\)}
    \AxiomC{\(\beta(w^\circ, \pi)\)}
    \noLine
    \UnaryInfC{\(\kappa^*_{w^\circ}, \Xi \Rightarrow \Lambda\)}
    \RightLabel{\(\wk\)}
    \UnaryInfC{\(\kappa^*_{w^\circ}, x^*_w, \Xi \Rightarrow \Lambda\)}
    \RightLabel{\(\cut\)}
    \BinaryInfC{\(x^*_w, \Xi \Rightarrow \Lambda\)}
    \DisplayProof
  \]
  where the proof \(\tau\) exists since \(\CS \vdash x^*_w \leftrightarrow \kappa^*_{w^\circ}\).

  Now assume that \(w \not\in\rep(T)\) and let \(R\) be the rule at \(w\).
  Below there are the definitions when \(R\) is \((\ax)\), \((\emp)\), \((\botL)\), \((\toL)\) or \((\toR)\).
  \[
    \left(
      \AxiomC{}
      \RightLabel{\(\ax\)}
      \UnaryInfC{\(\bot : p, \Gamma \Rightarrow p, \Delta\)}
      \DisplayProof,
      \AxiomC{\(\pi\)}
      \noLine
      \UnaryInfC{\(p, \Gamma, \Xi \Rightarrow p, \Delta, \Lambda\)}
      \DisplayProof
    \right)
    \quad
    \overset{\beta}{\longmapsto}
    \quad
    \AxiomC{\(\)}
    \RightLabel{\(\botL\)}
    \UnaryInfC{\(\bot, \Xi \Rightarrow \Lambda\)}
    \DisplayProof
  \]

  \[
    \left(
      \AxiomC{}
      \RightLabel{\(\emp\)}
      \UnaryInfC{\(\top : { \Rightarrow }\)}
      \DisplayProof,
      \AxiomC{\(\pi\)}
      \noLine
      \UnaryInfC{\(\Xi \Rightarrow \Lambda\)}
      \DisplayProof
    \right)
    \quad
    \overset{\beta}{\longmapsto}
    \quad
    \AxiomC{\(\pi\)}
    \noLine
    \UnaryInfC{\(\Xi \Rightarrow \Lambda\)}
    \RightLabel{\(\wk\)}
    \UnaryInfC{\(\top, \Xi \Rightarrow \Lambda\)}
    \DisplayProof
  \]

  \[
    \left(
      \AxiomC{}
      \RightLabel{\(\botL\)}
      \UnaryInfC{\(\bot : \bot, \Gamma \Rightarrow \Delta\)}
      \DisplayProof,
      \AxiomC{\(\pi\)}
      \noLine
      \UnaryInfC{\(\bot, \Gamma, \Xi \Rightarrow \Delta, \Lambda\)}
      \DisplayProof
    \right)
    \quad
    \overset{\beta}{\longmapsto}
    \quad
    \AxiomC{\(\)}
    \RightLabel{\(\botL\)}
    \UnaryInfC{\(\bot, \Xi \Rightarrow \Lambda\)}
    \DisplayProof
  \]

  \[
    \left(
      \AxiomC{\(w0\)}
      \noLine
      \UnaryInfC{\(\kappa_w :  \Gamma \Rightarrow \Delta\)}
      \RightLabel{\(\botR\)}
      \UnaryInfC{\(\kappa_w :  \Gamma \Rightarrow \bot, \Delta\)}
      \DisplayProof,
      \AxiomC{\(\pi\)}
      \noLine
      \UnaryInfC{\(\Gamma, \Xi \Rightarrow \bot, \Delta, \Lambda\)}
      \DisplayProof
    \right)
    \quad
    \overset{\beta}{\longmapsto}
    \quad
    \AxiomC{\(\beta(w0,\inv{\botR}(\pi))\)}
    \noLine
    \UnaryInfC{\(\kappa^*_w, \Xi \Rightarrow \Lambda\)}
    \RightLabel{\(\wk\)}
    \UnaryInfC{\(\kappa^*_w, \Xi \Rightarrow \Lambda\)}
    \DisplayProof
  \]

  \begin{multline*}
    \left(
      \AxiomC{\(w0\)}
      \noLine
      \UnaryInfC{\(\kappa_0 :  \psi, \Gamma \Rightarrow \Delta\)}
      \AxiomC{\(w1\)}
      \noLine
      \UnaryInfC{\(\kappa_1 :  \Gamma \Rightarrow \phi, \Delta\)}
      \RightLabel{\(\toL\)}
      \BinaryInfC{\(\kappa_0 \vee \kappa_1 :  \phi \to \psi, \Gamma \Rightarrow \Delta\)}
      \DisplayProof,
      \AxiomC{\(\pi\)}
      \noLine
      \UnaryInfC{\(\phi \to \psi, \Gamma, \Xi \Rightarrow \Delta, \Lambda\)}
      \DisplayProof
    \right)
    \quad
    \overset{\beta}{\longmapsto} \\
    \AxiomC{\(\beta(w0,\inv{\toL_0}(\pi))\)}
    \noLine
    \UnaryInfC{\(\kappa^*_0, \Xi \Rightarrow \Lambda\)}
    \AxiomC{\(\beta(w1,\inv{\toL_1}(\pi))\)}
    \noLine
    \UnaryInfC{\(\kappa^*_1, \Xi \Rightarrow \Lambda\)}
    \RightLabel{\(\veeL\)}
    \BinaryInfC{\(\kappa^*_0 \vee \kappa^*_1, \Xi \Rightarrow \Lambda\)}
    \DisplayProof
  \end{multline*}

  \[
    \left(
      \AxiomC{\(w0\)}
      \noLine
      \UnaryInfC{\(\kappa_w :  \phi, \Gamma \Rightarrow \psi, \Delta\)}
      \RightLabel{\(\toR\)}
      \UnaryInfC{\(\kappa_w :  \Gamma \Rightarrow \phi \to \psi, \Delta\)}
      \DisplayProof,
      \AxiomC{\(\pi\)}
      \noLine
      \UnaryInfC{\(\Gamma, \Xi \Rightarrow \phi \to \psi, \Delta, \Lambda\)}
      \DisplayProof
    \right)
    \quad
    \overset{\beta}{\longmapsto}
    \quad
    \AxiomC{\(\beta(w0,\inv{\toR}(\pi))\)}
    \noLine
    \UnaryInfC{\(\kappa^*_w, \Xi \Rightarrow \Lambda\)}
    \RightLabel{\(\wk\)}
    \UnaryInfC{\(\kappa^*_w, \Xi \Rightarrow \Lambda\)}
    \DisplayProof
  \]

  Finally, assume \(R\) is \((\modal[+]{\KTCS})\).
  Then \(w\) is of shape
  \[
    \AxiomC{\(\left[\begin{matrix} w^{\nec_i} \\ \kappa^{\nec_i} : \necd_i \Sigma^s_i, \nec_{\overline{i}}\Sigma^s_{\overline{i}}  \Rightarrow \end{matrix}\right]_{i \leq 1}\)}
    \AxiomC{\(\left[\begin{matrix} w^{\nec_i}_\phi \\ \kappa^{\nec_i}_\phi :\necd_i \Sigma^s, \nec_{\overline{i}} \Sigma^s_i \Rightarrow\phi\end{matrix}\right]_{i \leq 1, \phi \in \Theta_i}\)}
    \RightLabel{\(\modal[+]{\KTCS}\)}
    \BinaryInfC{\( \kappa_w  :\nec_0 \Sigma_0, \nec_1 \Sigma_1, \Gamma \Rightarrow \nec_0 \Theta_0, \nec_1 \Theta_1, \Delta\)}
    \DisplayProof
  \]
  where
  \[
    \kappa^*_{w} := \bigwedge_{i \leq 1}\left(\nec_i (\kappa^{\nec_i})^* \wedge \bigwedge_{\phi \in \Theta_i} \pos_i (\kappa^{\nec_i}_\phi)^*\right)\wedge \bigwedge (\Gamma \inter V_+) \wedge \bigwedge \neg(\Delta \inter V_-)
  \]
  We define \(\beta(w,\pi)\) by case analysis on the shape of \(\pi\).
  \begin{itemize}
    \item Last rule of \(\pi\) is \((\ax)\).
      There are 4 cases depending on where the repeated propositional variable is located.
      \begin{itemize}
        \item Assume \(p \in \Gamma\) and \(p \in \Delta\).
          This case is impossble, as \(w\) is annotated with \((\modal[+]{\KTCS})\) implies that \(\Gamma \cap \Delta = \varnothing\).
        \item Assume \(p \in \Gamma\) and \(p \in \Lambda\).
          Since \(p \in \Lambda\) and \(\voc_+(\Xi \Rightarrow \Lambda) \subseteq V_+\) we have that \(p \in V_+\).
          Then, as \(p \in \Gamma \inter V_+\) and \(p \not \in \set{x_w \mid w \in \rep(T)}\), we obtain that \(p\) is a conjunct of \(\kappa^*_w\).
          The desired preproof is
          \[
            \AxiomC{\(\)}
            \RightLabel{\(\ax\)}
            \UnaryInfC{\(p, \Xi \Rightarrow p, \Lambda'\)}
            \RightLabel{\(\wk + \wedgeL\)}
            \doubleLine
            \UnaryInfC{\(\kappa^*_w, \Xi \Rightarrow p, \Lambda'\)}
            \DisplayProof
          \]
          where \(\Lambda = p, \Lambda'\)
        \item Assume \(p \in \Xi\) and \(p \in \Delta\).
          Since \(p \in \Xi\) and \(\voc_-(\Xi \Rightarrow \Lambda) \subseteq V_-\) we have that \(p \in V_-\).
          Then, as \(p \in \Delta \inter V_-\) and \(p \not \in \set{x_w \mid w \in \rep(T)}\), we obtain that \(\neg p\) is a conjunct of \(\kappa^*_w\).
          The desired preproof is
          \[
            \AxiomC{\(\)}
            \RightLabel{\(\ax\)}
            \UnaryInfC{\(p, \Xi' \Rightarrow p, \Lambda\)}
            \RightLabel{\(\negL\)}
            \UnaryInfC{\(\neg p, p, \Xi' \Rightarrow \Lambda\)}
            \RightLabel{\(\wk + \wedgeL\)}
            \doubleLine
            \UnaryInfC{\(\kappa^*_w, p, \Xi' \Rightarrow \Lambda\)}
            \DisplayProof
          \]
          where \(\Xi = p, \Xi'\)
        \item Assume \(p \in \Xi\) and \(p \in \Lambda\).
          Then the desired preproof is
          \[
            \AxiomC{}
            \RightLabel{\(\ax\)}
            \UnaryInfC{\(\kappa^*_w, p, \Xi' \Rightarrow p, \Lambda'\)}
            \DisplayProof
          \]
          where \(\Xi = p, \Xi'\) and \(\Lambda = p, \Lambda'\).
      \end{itemize}
    \item Last rule of \(\pi\) is \((\botL)\).
      Since \(\bot\) is not in \(\Gamma\) (which consists in propositional variables only), then we have that \(\pi\) has the following shape
      \[
        \AxiomC{\(\)}
        \RightLabel{\(\botL\)}
        \UnaryInfC{\(\nec_0 \Sigma_0, \nec_1 \Sigma_1, \Gamma, \bot, \Xi' \Rightarrow \nec_0 \Theta_0, \nec_1 \Theta_1, \Delta, \Lambda\)}
        \DisplayProof
      \]
      where \(\Xi = \bot, \Xi'\).
      Then \(\beta(w,\pi)\) is defined as
      \[
        \AxiomC{\(\)}
        \RightLabel{\(\botL\)}
        \UnaryInfC{\(\kappa^*_w, \bot, \Xi' \Rightarrow \Lambda\)}
        \DisplayProof
      \]
    \item Last rule of \(\pi\) is \((\botR)\).
      Since \(\bot\) is not in \(\Delta\) (which consists in propositional variables only), then we have that \(\pi\) has the following shape
      \[
        \AxiomC{\(\pi_0\)}
        \noLine
        \UnaryInfC{\(\nec_0 \Sigma_0, \nec_1 \Sigma_1, \Gamma, \Xi \Rightarrow \nec_0 \Theta_0, \nec_1 \Theta_1, \Delta, \Lambda'\)}
        \RightLabel{\(\botR\)}
        \UnaryInfC{\(\nec_0 \Sigma_0, \nec_1 \Sigma_1, \Gamma, \Xi \Rightarrow \nec_0 \Theta_0, \nec_1 \Theta_1, \Delta, \bot, \Lambda'\)}
        \DisplayProof
      \]
      where \(\Lambda = \bot, \Lambda'\).
      Then \(\beta(w,\pi)\) is defined as
      \[
        \AxiomC{\(\beta(w,\pi_0)\)}
        \noLine
        \UnaryInfC{\(\kappa^*_w, \bot, \Xi \Rightarrow \Lambda'\)}
        \RightLabel{\(\botR\)}
        \UnaryInfC{\(\kappa^*_w, \bot, \Xi \Rightarrow \bot, \Lambda'\)}
        \DisplayProof
      \]
    \item Last rule of \(\pi\) is \((\toL)\).
      Since \(\Gamma\) consists of propositional variables only, we have that the principal formula of \((\toL)\) must be at \(\Xi\).
      Then we have that \(\pi\) has the following shape
      \[
        \AxiomC{\(\pi_0\)}
        \noLine
        \UnaryInfC{\(\nec_0 \Sigma_0, \nec_1 \Sigma_1, \Gamma, \Xi' \Rightarrow \nec_0 \Theta_0, \nec_1 \Theta_1, \Delta, \phi, \Lambda\)}
        \AxiomC{\(\pi_1\)}
        \noLine
        \UnaryInfC{\(\nec_0 \Sigma_0, \nec_1 \Sigma_1, \Gamma,  \psi, \Xi' \Rightarrow \nec_0 \Theta_0, \nec_1 \Theta_1, \Delta, \Lambda\)}
        \RightLabel{\(\toL\)}
        \BinaryInfC{\(\nec_0 \Sigma_0, \nec_1 \Sigma_1, \Gamma,  \phi \to \psi, \Xi' \Rightarrow \nec_0 \Theta_0, \nec_1 \Theta_1, \Delta, \Lambda\)}
        \DisplayProof
      \]
      where \(\Xi = \phi \to \psi, \Xi'\).
      Then \(\beta(w,\pi)\) is defined as
      \[
        \AxiomC{\(\beta(w,\pi_0)\)}
        \noLine
        \UnaryInfC{\(\kappa^*_w, \Xi' \Rightarrow \phi, \Lambda\)}
        \AxiomC{\(\beta(w,\pi_1)\)}
        \noLine
        \UnaryInfC{\(\kappa^*_w, \psi, \Xi' \Rightarrow \Lambda\)}
        \RightLabel{\(\toL\)}
        \BinaryInfC{\(\kappa^*_w, \phi \to \psi, \Xi' \Rightarrow \Lambda\)}
        \DisplayProof
      \]
    \item Last rule of \(\pi\) is \((\toR)\).
      Since \(\Delta\) consists of propositional variables only, we have that the principal formula of \((\toR)\) must be at \(\Lambda\).
      Then we have that \(\pi\) has the following shape
      \[
        \AxiomC{\(\pi_0\)}
        \noLine
        \UnaryInfC{\(\nec_0 \Sigma_0, \nec_1 \Sigma_1, \Gamma, \phi, \Xi \Rightarrow \nec_0 \Theta_0, \nec_1 \Theta_1, \Delta, \psi, \Lambda'\)}
        \RightLabel{\(\toR\)}
        \UnaryInfC{\(\nec_0 \Sigma_0, \nec_1 \Sigma_1, \Gamma, \Xi \Rightarrow \nec_0 \Theta_0, \nec_1 \Theta_1, \Delta, \phi \to \psi, \Lambda'\)}
        \DisplayProof
      \]
      where \(\Lambda = \phi \to \psi, \Lambda'\).
      Then \(\beta(w,\pi)\) is defined as
      \[
        \AxiomC{\(\beta(w,\pi_0)\)}
        \noLine
        \UnaryInfC{\(\kappa^*_w, \phi, \Xi \Rightarrow \psi, \Lambda'\)}
        \RightLabel{\(\toL\)}
        \UnaryInfC{\(\kappa^*_w, \Xi \Rightarrow \phi \to \psi, \Lambda\)}
        \DisplayProof
      \]
    \item Last rule of \(\pi\) is \((\modal[i]{\KT})\) and the principal formula \(\nec_i \phi\) is in \(\nec_i \Theta_i\).
      Then \(\pi\) has the following shape
      \[
        \AxiomC{\(\pi_0\)}
        \noLine
        \UnaryInfC{\(\necd_i \Sigma'_i, \nec_{\overline{i}} \Sigma'_{\overline{i}}, \necd_i \Xi_{\nec_{\overline{i}}}, \nec_i \Xi_{\nec_{\overline{i}}} \Rightarrow \phi\)}
        \RightLabel{\(\modal[i]{\KTCS}\)}
        \UnaryInfC{\(\nec_0 \Sigma_0, \nec_1 \Sigma_1, \Gamma, \Xi \Rightarrow \nec_0 \Theta_0, \nec_1 \Theta_1, \Delta, \Lambda\)}
        \DisplayProof
      \]
      where \(\Sigma'_i \subseteq \Sigma_i\), \(\Sigma'_{\overline{i}} \subseteq \Sigma_{\overline{i}}\) and \(\nec_i \Xi_{\nec_i}, \nec_{\overline{i}} \Xi_{\nec_{\overline{i}}} \subseteq \Xi\).
      Then the desired preproof is
      \[
        \AxiomC{\(\beta(w^{\nec_i}_\phi, \pi_0)\)}
        \noLine
        \UnaryInfC{\((\kappa^{\nec_i}_\phi)^*, \necd_i \Xi_{\nec_{\overline{i}}}, \nec_i \Xi_{\nec_{\overline{i}}} \Rightarrow\)}
        \RightLabel{\(\negR\)}
        \UnaryInfC{\(\necd_i \Xi_{\nec_{\overline{i}}}, \nec_i \Xi_{\nec_{\overline{i}}} \Rightarrow \neg (\kappa^{\nec_i}_\phi)^* \)}
        \RightLabel{\(\modal[i]{\KTCS}\)}
        \UnaryInfC{\(\Xi \Rightarrow \nec_i \neg (\kappa^{\nec_i}_\phi)^*, \Lambda\)}
        \RightLabel{\(\negL\)}
        \UnaryInfC{\(\pos_i (\kappa^{\nec_i}_\phi)^*, \Xi \Rightarrow  \Lambda\)}
        \doubleLine
        \RightLabel{\(\wk + \wedgeL\)}
        \UnaryInfC{\(\kappa^*_w, \Xi \Rightarrow  \Lambda\)}
        \DisplayProof
      \]
    \item Last rule of \(\pi\) is \((\modal[i]{\KT})\) and the principal formula is in \(\Lambda\).
      Then \(\pi\) has the following shape
      \[
        \AxiomC{\(\pi_0\)}
        \noLine
        \UnaryInfC{\(\necd_i \Sigma'_i, \nec_{\overline{i}} \Sigma'_{\overline{i}}, \necd_i \Xi_{\nec_{\overline{i}}}, \nec_i \Xi_{\nec_{\overline{i}}} \Rightarrow \phi\)}
        \RightLabel{\(\modal[i]{\KTCS}\)}
        \UnaryInfC{\(\nec_0 \Sigma_0, \nec_1 \Sigma_1, \Gamma, \Xi \Rightarrow \nec_0 \Theta_0, \nec_1 \Theta_1, \Delta, \nec_i\phi, \Lambda'\)}
        \DisplayProof
      \]
      where \(\Lambda = \nec_i \phi, \Lambda'\), \(\Sigma'_i \subseteq \Sigma_i\), \(\Sigma'_{\overline{i}} \subseteq \Sigma_{\overline{i}}\) and \(\nec_i \Xi_{\nec_i}, \nec_{\overline{i}} \Xi_{\nec_{\overline{i}}} \subseteq \Xi\).
      Then the desired preproof is
      \[
        \AxiomC{\(\beta(w^{\nec_i}, \pi_0)\)}
        \noLine
        \UnaryInfC{\((\kappa^{\nec_i})^*, \necd_i \Xi_{\nec_{\overline{i}}}, \nec_i \Xi_{\nec_{\overline{i}}} \Rightarrow \phi\)}
        \RightLabel{\(\wk\)}
        \UnaryInfC{\(\necd_i (\kappa^{\nec_i})^*, \necd_i \Xi_{\nec_{\overline{i}}}, \nec_i \Xi_{\nec_{\overline{i}}} \Rightarrow \phi\)}
        \RightLabel{\(\modal[i]{\KTCS}\)}
        \UnaryInfC{\(\nec_i (\kappa^{\nec_i})^*, \Xi \Rightarrow \nec_i \phi, \Lambda'\)}
        \doubleLine
        \RightLabel{\(\wk + \wedgeL\)}
        \UnaryInfC{\(\kappa^*_w, \Xi \Rightarrow \nec_i \phi, \Lambda'\)}
        \DisplayProof
      \]
  \end{itemize}

  Finally, let us argue that \(\beta(w, \pi)\) is always a proof and not only a preproof.
  Given a node \(w\) assign to it the measure \(\omega^2 |\Gamma_w \Rightarrow \Delta_w| + \omega\lgth(w) + \lhg(\pi)\) where \(\lgth(w)\) is the length of \(w\) as a sequence of natural numbers.
  We notice that this measure always decreases from \(\beta(w,\pi)\) to its corecursive calls, except when \(w\) is annotated with the rule \((\modal[+]{\KTCS})\) and the last rule of \(\pi\) is \(\modal[i]{\CS}\).
  However, in this case we will find that progress is made (i.e., there is an application of \((\modal[i]{\KTCS})\)) from the root of the preproof fragment given by \(\beta(w,\pi)\) to the corecursive calls.
  This implies that any infinite branch will have infinitely many applications of \((\modal[i]{\KTCS})\), as otherwise we will infinitely decrease the measure on corecursive calls.
\end{proof}

After all the necessary technical work, we can finally conclude the desired result.

\begin{theorem}\label{th:ulip-CS}
  \(\CS\) has uniform Lyndon interpolation.
\end{theorem}
\begin{proof}
  Let \(T\) be an interpolation template for \( \phi \Rightarrow \).
  We already know that \(\voc_b(\iota_T) \subseteq V_b\) for \(b \in \set{+,-}\) and by Theorem~\ref{th:first-verification-CS} we have that \(\CS \vdash \phi \to \iota_T\).
  Finally, assume that \(\CS \vdash \phi \to \psi\) where \(\psi\) is such that \(\voc_b(\psi) \subseteq V_b\) for \(b \in \set{+,-}\).
  Then \(\n{\CS} \vdash \phi \Rightarrow \psi\) and by Theorem~\ref{th:second-verification-CS} we obtain that \(\n{\CS} \vdash \iota_T \Rightarrow \psi\), i.e., \(\CS \vdash \iota_T \to \psi\) as desired.
\end{proof}

Using the uniform Lyndon interpolation of \(\CS\) we show it also for \(\CSM\).
First let us define a function \((\cdot)^\sharp : \wff \function \wff\) as follows:
  \begin{align*}
    &p^\sharp = p, 
    &&\bot^\sharp = \bot, \\
    &(\phi \to \psi)^\sharp = \phi^\sharp \to \psi^\sharp, \\
    &(\nec_0 \phi)^\sharp = \nec_0 \phi^\sharp \wedge \nec_1 \phi^\sharp, 
    &&(\nec_1 \phi)^\sharp = \nec_1 \phi^\sharp.
  \end{align*}
The following lemma shows that \((\cdot)^\sharp\) is an interpretation of \(\CSM\) in \(\CS\).

\begin{lemma}\label{lm:interpretation-CSM-in-CS}
  Let \(\phi\) be any formula, we have the following.
  \begin{enumerate}
    \item \(\voc_b(\phi^\sharp) \subseteq \voc_b(\phi)\) for \(b \in \set{+,-}\).
    \item \(\CSM \vdash \phi \leftrightarrow \phi^\sharp\).
    \item If \(\CSM \vdash \phi\) then \(\CS \vdash \phi^\sharp\).
  \end{enumerate}
\end{lemma}
\begin{proof}
  The first two points follow easily by induction on the complexity of the formula \(\phi\).
  To show the last point it suffices to prove that \(\set{\phi \mid \CS \vdash \phi^\sharp}\) contains the axioms of \(\CSM\) and is closed under \((\MP)\) and \((\NEC[i])\) for \(i \leq 1\).
  The closure properties are straightforward to show, so we focus in showing that all the axioms of \(\CS\) belong to that set.
  For \((\Kax[1])\) and \((\Lax[1])\) the proof is trivial, as \((\cdot)^\sharp\) commutes will all the connectives occuring in those axiom schemes.

  Axiom \((\Kax[0])\).
  We have that
  \[
    (\nec_0(\phi \to \psi) \to \nec_0 \phi \to \nec_0 \psi)^\sharp = 
    (\nec_0(\phi^\sharp \to \psi^\sharp) \wedge \nec_1(\phi^\sharp \to \psi^\sharp)) \to ( \nec_0 \phi^\sharp \wedge \nec_1 \phi ^\sharp) \to (\nec_0 \psi^\sharp \wedge \nec_1 \psi^\sharp).
  \]
  That \(\CS\) proves that formula is equivalent to
  \[
    \CS \vdash
    \nec_0(\phi^\sharp \to \psi^\sharp) \wedge \nec_1(\phi^\sharp \to \psi^\sharp) \wedge \nec_0 \phi^\sharp \wedge \nec_1 \phi ^\sharp \to \nec_0 \psi^\sharp \wedge \nec_1 \psi^\sharp,
  \]
  which clearly holds by \((\Kax[0])\) and \((\Kax[1])\).

  Axiom \((\Lax[0])\).
  We have that
  \[
    (\nec_0(\nec_0 \phi \to \phi) \to \nec_0 \phi)^\sharp =
     \nec_0(\nec_0 \phi^\sharp \wedge \nec_1 \phi^\sharp \to \phi^\sharp) \wedge \nec_1(\nec_0 \phi^\sharp \wedge \nec_1 \phi^\sharp \to \phi^\sharp) \to \nec_0 \phi^\sharp \wedge \nec_1 \phi^\sharp.
  \]
  Proving this in \(\CS\) is equivalent to showing
  \[
    \CS 
    \vdash
     \nec_0(\nec_0 \phi^\sharp \wedge \nec_1 \phi^\sharp \to \phi^\sharp) \wedge \nec_1(\nec_1 \phi^\sharp \wedge \nec_0 \phi^\sharp \to \phi^\sharp) \to \nec_0 \phi^\sharp \wedge \nec_1 \phi^\sharp.
  \]
  We show the proof of this formula in \(\g{\CS}\).
  Let \(\Gamma = \set{\nec_0(\nec_0 \phi^\sharp \wedge \nec_1 \phi^\sharp \to \phi^\sharp) \to \phi^\sharp, \nec_1(\nec_1 \phi^\sharp \wedge \nec_0 \phi^\sharp \to \phi^\sharp)}\).
  Define \(\pi_i\) for \(i \leq 1\) to be the proof
  \[
    \AxiomC{}
    \RightLabel{\(\Ax\)}
    \UnaryInfC{\(  \Gamma, \nec_i \phi^\sharp, \nec_{\overline{i}} \phi^\sharp \Rightarrow \nec_i \phi^\sharp, \phi^\sharp\)}
    \AxiomC{}
    \RightLabel{\(\Ax\)}
    \UnaryInfC{\(  \Gamma, \nec_i \phi^\sharp, \nec_{\overline{i}} \phi^\sharp \Rightarrow  \nec_i \phi^\sharp, \phi^\sharp\)}
    \RightLabel{\(\wedgeR\)}
    \BinaryInfC{\(  \Gamma, \nec_i \phi^\sharp, \nec_{\overline{i}} \phi^\sharp \Rightarrow \nec_{\overline{i}} \phi^\sharp \wedge \nec_i \phi^\sharp, \phi^\sharp\)}
    \AxiomC{}
    \RightLabel{\(\Ax\)}
    \UnaryInfC{\( \phi^\sharp, \Gamma, \nec_i \phi^\sharp, \nec_{\overline{i}} \phi^\sharp \Rightarrow \phi^\sharp\)}
    \RightLabel{\(\toL\)}
    \BinaryInfC{\(  \nec_{\overline{i}} \phi^\sharp \wedge \nec_i \phi^\sharp \to \phi^\sharp, \Gamma, \nec_i \phi^\sharp, \nec_{\overline{i}} \phi^\sharp \Rightarrow \phi^\sharp\)}
    \RightLabel{\(\modal[{\overline{i}}]{\CS}\)}
    \UnaryInfC{\( \Gamma,  \nec_i \phi^\sharp \Rightarrow \phi^\sharp, \nec_{\overline{i}} \phi^\sharp\)}
    \DisplayProof
  \]
  Then define \(\tau_i\) for \(i \leq 1\) to be the proof
  \[
    \AxiomC{}
    \RightLabel{\(\Ax\)}
    \UnaryInfC{\(  \Gamma, \nec_i \phi^\sharp \Rightarrow \phi^\sharp, \nec_i \phi^\sharp \)}
    \AxiomC{\(\pi_i\)}
    \noLine
    \UnaryInfC{\( \Gamma,  \nec_i \phi^\sharp \Rightarrow \phi^\sharp, \nec_{\overline{i}} \phi^\sharp\)}
    \RightLabel{\(\wedgeR\)}
    \BinaryInfC{\( \Gamma, \nec_i \phi^\sharp \Rightarrow \phi^\sharp, \nec_i \phi^\sharp \wedge \nec_{\overline{i}} \phi^\sharp\)}
    \AxiomC{}
    \RightLabel{\(\Ax\)}
    \UnaryInfC{\(\phi^\sharp, \Gamma, \nec_i \phi^\sharp\Rightarrow \phi^\sharp\)}
    \RightLabel{\(\toL\)}
    \BinaryInfC{\(\nec_i \phi^\sharp \wedge \nec_{\overline{i}} \phi^\sharp \to \phi^\sharp, \Gamma,  \nec_i \phi^\sharp \Rightarrow \phi^\sharp\)}
    \RightLabel{\(\modal[i]{\CS}\)}
    \UnaryInfC{\(\Gamma \Rightarrow \nec_i \phi^\sharp\)}
    \DisplayProof
  \]
  Then the desired proof is
  \[
    \AxiomC{\(\tau_0\)}
    \noLine
    \UnaryInfC{\( \Gamma \Rightarrow \nec_0 \phi^\sharp\)}
    \AxiomC{\(\tau_1\)}
    \noLine
    \UnaryInfC{\( \Gamma \Rightarrow \nec_1 \phi^\sharp \)}
    \RightLabel{\(\wedgeR\)}
    \BinaryInfC{\( \Gamma \Rightarrow \nec_0 \phi^\sharp \wedge \nec_1 \phi^\sharp \)}
    \doubleLine \RightLabel{\(\toR + \wedgeL\)}
    \UnaryInfC{\( \Rightarrow \nec_0(\nec_0 \phi^\sharp \wedge \nec_1 \phi^\sharp \to \phi^\sharp) \wedge \nec_1(\nec_1 \phi^\sharp \wedge \nec_0 \phi^\sharp \to \phi^\sharp) \to \nec_0 \phi^\sharp \wedge \nec_1 \phi^\sharp \)}
    \DisplayProof
  \]

  Axiom \((\Cax[1,0])\).
  We have that \((\nec_0 \phi \to \nec_1 \nec_0 \phi)^\sharp = \nec_0 \phi^\sharp \wedge \nec_1 \phi^\sharp \to \nec_1(\nec_0 \phi^\sharp \wedge \nec_1 \phi^\sharp)\).
  By normality, showing that \(\CS\) proves this is equivalent to showing that
  \[
    \CS \vdash \nec_0 \phi^\sharp \wedge \nec_1 \phi^\sharp \to \nec_1\nec_0 \phi^\sharp \wedge \nec_1\nec_1 \phi^\sharp,
  \]
  which clearly holds using \((\Cax[1,0])\) and transitivity of \(\nec_1\).

  Axiom \((\Cax[0,1])\).
  We have that \((\nec_1 \phi \to \nec_0 \nec_1 \phi)^\sharp = \nec_1 \phi^\sharp \to \nec_0 \nec_1 \phi^\sharp \wedge \nec_1 \nec_1 \phi^\sharp\).
  Again, \(\CS\) clearly proves this using \(\Cax[1,0]\) and transitivity of \(\nec_1\).

  Axiom \((\Max[0,1])\).
  We have that \((\nec_0 \phi \to \nec_1 \phi)^\sharp = \nec_0 \phi^\sharp \wedge \nec_1 \phi^\sharp \to \nec_1 \phi^\sharp\).
  Clearly, \(\CS\) proves this as it is a propositional tautology.
\end{proof}

We can finally conclude that \(\CSM\) has uniform Lyndon interpolation.

\begin{theorem}
  \(\CSM\) has uniform Lyndon interpolation.
\end{theorem}
\begin{proof}
  Let \(\phi\) be a formula and \(V_+\) and \(V_-\) be vocabularies.
  Note that \(\phi^\sharp\) has a \(\CS\)-uniform Lyndon interpolant~\(\iota\). Let us show that \(\iota\) is the \(\CSM\)-uniform Lyndon interpolant of \(\phi\).
  We automatically have that \(\voc_{b}(\iota) \subseteq V_b\) for \(b \in \set{+,-}\).
  From \(\CS \vdash \phi^\sharp \to \iota\) we obtain that \(\CSM \vdash \phi \to \iota\) using that \(\CS \subseteq \CSM\) and Lemma~\ref{lm:interpretation-CSM-in-CS}.

  Finally, assume that \(\CS \vdash \phi \to \psi\) where \(\voc_b(\psi) \subseteq V_b\) for \(b \in \set{+,-}\).
  By Lemma~\ref{lm:interpretation-CSM-in-CS} we have that \(\CS \vdash \phi^\sharp \to \psi^\sharp\) and \(\voc_b(\psi^\sharp) \subseteq V_b\) for \(b \in \set{+,-}\).
  Using that \(\iota\) is the \(\CS\)-uniform Lyndon interpolant of \(\phi^\sharp\) we obtain that \(\CS \vdash \iota \to \psi^\sharp\).
  Finally, using Lemma~\ref{lm:interpretation-CSM-in-CS} and that \(\CS \subseteq \CSM\) we obtain that \(\CSM \vdash \iota \to \psi\), as desired.
\end{proof}

\subsection{Uniform Lyndon Interpolation for ER}

Contrary to the \(\CS\) case, the rules of \(\ER\) do not always lower the complexity of sequents, due to the rule \((\ERax[1,0])\).
For this reason, we will need to change a bit the proofs we did for \(\CS\) in the previous subsection.
The main difference is that we will need a notion of saturation (defined below) to exhaust the proof search space.

\begin{definition}
  Let \(\Gamma \Rightarrow_i \Delta\) be an \(\ER\)-sequent.
  We say that a formula \(\phi\) is
  \begin{multicols}{2}
    \begin{enumerate}
      \item \emph{Left-saturated} if
        \begin{enumerate}
          \item \(\phi\) is atomic or a \(\nec_1\)-formula,
          \item \(\phi = \phi_0 \to \phi_1\) and \(\phi_0 \in \Delta\) or \(\phi_1 \in \Gamma\),
          \item \(\phi = \nec_0 \phi_0\) and \(i = 0\) or \(i = 1\) and \(\phi_0 \in \Gamma\).
        \end{enumerate}
      \item \emph{Right-saturated} if
        \begin{enumerate}
          \item \(\phi\) is atomic or a \(\nec_i\)-formula for \(i \leq 1\).
          \item \(\phi = \phi_0 \to \phi_1\) and \(\phi_0 \in \Gamma\) and \(\phi_1 \in \Delta\),
        \end{enumerate}
    \end{enumerate}
  \end{multicols}
  A sequent \(\Gamma \Rightarrow_i \Delta\) is said to be \emph{saturated} if every \(\phi \in \Gamma\) is left-saturated in \(\Gamma \Rightarrow_i \Delta\) and every \(\psi \in \Delta\) is right-saturated in \(\Gamma \Rightarrow_i \Delta\).
  We define the \emph{saturation complexity} of a sequent \(\Gamma \Rightarrow_i \Delta\), denoted \(\satc{\Gamma \Rightarrow_i \Delta}\), as the number
  \[
    \sum \left(\set{|\phi| \mid \phi \in \Gamma, \text{\(\phi\) not left-saturated in \(\Gamma \Rightarrow_i \Delta\)}} \union
    \set{|\phi| \mid \phi \in \Delta, \text{\(\phi\) not right-saturated in \(\Gamma \Rightarrow_i \Delta\)}}\right).
  \]
\end{definition}
Notice that if \((S_0,\ldots,S_{n-1},S)\) is an instance of a rule for interpolation templates distinct from \(\modal[+]{\KTER}\) then \(\satc{S_i} < \satc{S}\) for \(i < n\).

\begin{figure}
  \[
    \AxiomC{\(\)}
    \RightLabel{\(\emp\)}
    \UnaryInfC{\( \Rightarrow \)}
    \DisplayProof
    \quad
    \AxiomC{\(\necd_0 \phi, \Gamma \Rightarrow_1 \Delta\)}
    \RightLabel{\(\ERax[1,0]^{\mathrm{Sat}}\)}
    \UnaryInfC{\(\nec_0 \phi, \Gamma \Rightarrow_1 \Delta\)}
    \DisplayProof
  \]
  where in \((\ERax[1,0]^{\mathrm{Sat}})\) the formula \(\nec_0 \phi\) is not left-saturated at the conclusion.

  \[
    \AxiomC{\(\phi \to \psi, \Gamma \Rightarrow_i \phi, \Delta\)}
    \AxiomC{\(\psi, \phi \to \psi, \Gamma \Rightarrow_i \Delta\)}
    \RightLabel{\(\toL^{\mathrm{Sat}}\)}
    \BinaryInfC{\(\phi \to \psi, \Gamma \Rightarrow_i \Delta\)}
    \DisplayProof
    \quad
    \AxiomC{\(\phi, \Gamma \Rightarrow_i \psi, \phi \to \psi, \Delta\)}
    \RightLabel{\(\toR^{\mathrm{Sat}}\)}
    \UnaryInfC{\(\Gamma \Rightarrow_i \phi \to \psi, \Delta\)}
    \DisplayProof
  \]
  where
  \begin{enumerate}
    \item in \((\toL^{\mathrm{Sat}})\) the formula \(\phi \to \psi\) is not left-saturated at the conclusion,
    \item in \((\toR^{\mathrm{Sat}})\) the formula \(\phi \to \psi\) is not right-saturated at the conclusion.
  \end{enumerate}

  \[
    \AxiomC{\([\necd_i \Sigma^s_i, \nec_{\overline{i}} \Sigma^s_{\overline{i}} \Rightarrow_i]_{i \leq 1}\)}
    \AxiomC{\([\necd_i \Sigma^s_i, \nec_{\overline{i}} \Sigma^s_{\overline{i}} \Rightarrow_i \phi]_{i \leq 1, \phi \in \Theta_i}\)}
    \RightLabel{\(\modal[+]{\KTER}\)}
    \BinaryInfC{\(\nec_0 \Sigma_0, \nec_1 \Sigma_1, \Gamma \Rightarrow_i \nec_0 \Theta_0, \nec_1 \Theta_1, \Delta\)}
    \DisplayProof
  \]
  where
  \begin{enumerate}
    \item \(\nec_0 \Sigma_0, \nec_1 \Sigma_1, \Gamma \Rightarrow \nec_0 \Theta_0, \nec_1 \Theta_1, \Delta\) is saturated,
    \item \(\Gamma, \Delta\) contain no \(\nec\)-formulas,
    \item \((\Gamma \inter \Delta) \inter \text{Var} = \varnothing\) and \(\bot \not\ \in \Gamma\) (in other words, the conclusion is not an instance of \((\ax)\) or \((\botL)\)).
  \end{enumerate}

  \caption{Interpolation template rules for \(\ER\)}
  \label{fig:interpolation-template-ER}
\end{figure}

\begin{definition}
  An \emph{\(\ER\)-interpolation template} is a cyclic proof in the local progress calculus with the rules \((\ax), (\botL)\) from Figure~\ref{fig:prop-rules-ER} and the rules of Figure~\ref{fig:interpolation-template-ER}, where progress is only made at the premises of \(\modal[+]{\KTER}\).
\end{definition}

Mimicking the proof of Lemma~\ref{existence-of-interpolation-templates-CS}, changing the measure on sequents from \(|S|\) to \(\satc{S}\), we obtain the following lemma.

\begin{lemma}
  Every \(\ER\)-sequent has an \(\ER\)-interpolation template.
\end{lemma}

We fix vocabularies \(V_+\) and \(V_-\) for which we want to calculate the uniform Lyndon interpolant.

\begin{definition}
  Let \(T\) be an interpolation template and \(w\) a node of \(T\), let us denote the sequent at \(w\) as \(\Gamma_w \Rightarrow \Delta_w\).
  We will annotate each of the sequents of \(T\) with a formula \(\kappa\) called the \emph{preinterpolant at \(\kappa\)}, denoted as \(\kappa : \Gamma_w \Rightarrow \Delta_w\).
  We proceed by induction on the tree structure of \(T\).
  If \(w\) is a repeat node of shape
  \[
    \AxiomC{\(\)}
    \RightLabel{\(\rep\)}
    \UnaryInfC{\(\Gamma_w \Rightarrow_{i_w} \Delta_w\)}
    \DisplayProof
  \]
  then \(\kappa_w\) will be a fresh propositional variable \(x_w\).
  This variable shall not appear in any formula of \(T\) and for repeat nodes \(w,v\) such that \(w \neq v\) we must have that \(x_w \neq x_v\).
  In case \(w\) is not a repeat node we proceed by cases on the rule at \(w\), we describe the construction pictorically as follows (the preinterpolant at the conclusion is defined recursively from the preinterpolants at the premises):
  \[
    \AxiomC{\(\)}
    \RightLabel{\(\ax\)}
    \UnaryInfC{\(\bot : p, \Gamma \Rightarrow_i p, \Delta\)}
    \DisplayProof
    \quad
    \AxiomC{\(\)}
    \RightLabel{\(\botL\)}
    \UnaryInfC{\(\bot : \bot, \Gamma \Rightarrow_i \Delta\)}
    \DisplayProof
    \quad
    \AxiomC{\(\)}
    \RightLabel{\(\emp\)}
    \UnaryInfC{\(\top : { \Rightarrow_i }\)}
    \DisplayProof
  \]

  \[
    \AxiomC{\(\kappa_0 : \phi \to \psi, \Gamma \Rightarrow_i \phi, \Delta\)}
    \AxiomC{\(\kappa_1 : \psi, \phi \to \psi, \Gamma \Rightarrow_i \Delta\)}
    \RightLabel{\(\toL^{\mathrm{Sat}}\)}
    \BinaryInfC{\(\kappa_0 \vee \kappa_1 : \phi \to \psi, \Gamma \Rightarrow_i \Delta\)}
    \DisplayProof
    \quad
    \AxiomC{\(\kappa : \phi, \Gamma \Rightarrow_i \psi, \phi \to \psi, \Delta\)}
    \RightLabel{\(\toR^{\mathrm{Sat}}\)}
    \UnaryInfC{\(\kappa : \Gamma \Rightarrow_i \phi \to \psi, \Delta\)}
    \DisplayProof
  \]

  \[
    \AxiomC{\(\kappa : \necd_0 \phi, \Gamma \Rightarrow_1 \Delta\)}
    \RightLabel{\(\ERax[1,0]^{\mathrm{Sat}}\)}
    \UnaryInfC{\(\kappa : \nec_0 \phi, \Gamma \Rightarrow_1 \Delta\)}
    \DisplayProof
    \quad
    \AxiomC{\([\kappa^{\nec_i} : \necd_i \Sigma^s_i, \nec_{\overline{i}} \Sigma^s_{\overline{i}} \Rightarrow_i]_{i \leq 1}\)}
    \AxiomC{\([\kappa^{\nec_i}_\phi : \necd_i \Sigma^s_i, \nec_{\overline{i}} \Sigma^s_{\overline{i}} \Rightarrow_i \phi]_{i \leq 1, \phi \in \Theta_i}\)}
    \RightLabel{\(\modal[+]{\KTER}\)}
    \BinaryInfC{\(\kappa : \nec_0 \Sigma_0, \nec_1 \Sigma_1, \Gamma \gg \nec_0 \Theta_0, \nec_1 \Theta_1, \Delta\)}
    \DisplayProof
  \]
  where
  \[
    \kappa := \bigwedge_{i \leq 1}\left(\nec_i \kappa^{\nec_i} \wedge \bigwedge_{\phi \in \Theta_i} \pos_i \kappa^{\nec_i}_\phi\right)\wedge \bigwedge (\Gamma \inter V_+) \wedge \bigwedge \neg(\Delta \inter V_-).
  \]
\end{definition}

Again, each interpolation template gives an equational system that we proceed to define. For each \(i \in \mathbb{N}\) let \(\bar{x}_i\) be an arbitrary enumeration of \(\set{x_{w} \mid w \in \rep(T), \hg(w^\circ) = i}\).
If \(\hg(T) = H\), then we define \(\bar{x}_T = \bar{x}_0 \bar{x}_1 \cdots \bar{x}_H\), \(\bar{x}_T\) is an anumeration of \(\set{x_w \mid w \in \rep(T)}\).
We define the \emph{equational system of \(T\)}, denoted \(\mathcal{E}_T\), as
\[
  \set{(x_w, +, \kappa_{w}) \mid w \in \rep(T)}.
\]
In the following lemma we show that it \(\mathcal{E}_T\) is solvable in \(\ER\).

\begin{lemma}
  \(\mathcal{E}_T\) is a positive modalized Lyndon equational system over \((\bar{x}_T, V_+, V_-)\).
  As a corollary, \(\mathcal{E}_T\) has a solution in \(\CS\).
\end{lemma}
\begin{proof}
  The proof follows the same structure as the proof of Lemma~\ref{equational-systems-are-solvable-CS}.
\end{proof}

The definition of interpolant is analogous to the \(\CS\) logic using interpolation templates and equational systems.
Notice that the same remarks about the uniqueness of interpolants that we made for \(\CS\) can be applied to \(ER\) (see page \pageref{uniqueness-of-interpolant-CS}).

\begin{definition}[Interpolant]
  Given an interpolation template \(T\) we define the interpolant of \(T\), denoted \(\iota_T\), as the formula obtained by applying the substitution solving \(\mathcal{E}_T\) to the preinterpolant at the root of \(T\).
\end{definition}

As in the case of \(\CS\), we finish by proving that the intepolant has the necessary properties.
Again, showing that \(\voc_b(\iota_T) \subseteq V_b\) for \(b \in \set{+,-}\) is straightforward by definition (see page \pageref{vocabularies-interpolant-CS}).
The proofs of the two following theorems differ slightly at some points from those for \(\CS\).
In particular, we have to exploit the properties of saturation, as now we measure the complexity of sequent via saturation complexity.

\begin{theorem}\label{th:first-verification-ER}
  Let \(T\) be an interpolation template for \(\Gamma \Rightarrow_i \Delta\).
  Then \(\n{\ER} \vdash \Gamma \Rightarrow_i \Delta, \iota_T\).
\end{theorem}
\begin{proof}
  Given a node \(w\) of \(T\) let us denote the sequent at \(w\) as \(\Gamma_w \Rightarrow_{i_w} \Delta_w\), the preinterpolant of \(T\) at \(w\) as \(\kappa_w\) and assume \((\cdot)^*\) is a solution of \(\mathcal{E}_T\).
  We define a function \(\alpha\) such that if \(w\) is a node of \(T\) then \(\alpha(w)\) is a preproof in \(\n{\ER} + \cut + \wk\) of \(\Gamma_w \Rightarrow_{i_w} \Delta_w, \kappa^*_w\).
  Then we will show that \(\alpha(w)\) is in fact a proof in \(\n{\ER} + \cut + \wk\) and then using that \(\wk\) is eliminable in \(\n{\ER} + \cut\) and that \(\cut\) is eliminable in \(\n{\ER}\) we obtain the desired proof in \(\n{\ER}\).

  We proceed to define \(\alpha\) by corecursion and cases.
  First, assume \(w \in \rep(T)\) then \(w^\circ\) has the same sequent as \(w\) and \(\kappa_w = x_w\).
  Also, by virtue of \((\cdot)^*\) being a solution of \(\mathcal{E}_T\), we have that \(\ER \vdash x^*_w \leftrightarrow \kappa^*_{w^\circ}\).
  Then we define \(\alpha\) as
  \[
    \AxiomC{}
    \RightLabel{\(\rep\)}
    \UnaryInfC{\(x_w : \Gamma \Rightarrow_i \Delta\)}
    \DisplayProof
    \quad
    \overset{\alpha}{\longmapsto}
    \quad
    \AxiomC{\(\alpha(w^\circ)\)}
    \noLine
    \UnaryInfC{\(\Gamma \Rightarrow_i \Delta, \kappa^*_w\)}
    \RightLabel{\(\wk\)}
    \UnaryInfC{\( \Gamma \Rightarrow_i \Delta, x^*_w, \kappa^*_w\)}
    \AxiomC{\(\tau\)}
    \noLine
    \UnaryInfC{\(\kappa^*_w,\Gamma \Rightarrow_i \Delta, x^*_w, \)}
    \RightLabel{\(\cut\)}
    \BinaryInfC{\(\Gamma \Rightarrow_i \Delta, x^*_w\)}
    \DisplayProof
  \]
  where the proof \(\tau\) exists since \(\ER \vdash x^*_w \leftrightarrow \kappa^*_{w^\circ}\) (so also \(\ER \vdash \nec_1(x^*_w \leftrightarrow \kappa^*_{w^\circ})\)).

  Now assume \(w\) is not a repetition node and let \(R\) be the rule at \(w\).
  Below there are the definitions when \(R\) is \((\ax)\), \((\emp)\), \((\botL)\), \((\toL^{\mathrm{Sat}})\), \((\toR^{\mathrm{Sat}})\) or \((\ERax[1,0]^{\mathrm{Sat}})\).
  \[
    \AxiomC{}
    \RightLabel{\(\ax\)}
    \UnaryInfC{\(\bot : p, \Gamma \Rightarrow p, \Delta\)}
    \DisplayProof
    \quad
    \overset{\alpha}{\longmapsto}
    \quad
    \AxiomC{}
    \RightLabel{\(\ax\)}
    \UnaryInfC{\(p, \Gamma \Rightarrow p, \Delta, \bot\)}
    \DisplayProof
  \]

  \[
    \AxiomC{}
    \RightLabel{\(\emp\)}
    \UnaryInfC{\(\top : {\Rightarrow}\)}
    \DisplayProof
    \quad
    \overset{\alpha}{\longmapsto}
    \quad
    \AxiomC{}
    \RightLabel{\(\botL\)}
    \UnaryInfC{\(\bot \Rightarrow \bot\)}
    \RightLabel{\(\toR\)}
    \UnaryInfC{\( \Rightarrow \top\)}
    \DisplayProof
  \]

  \[
    \AxiomC{}
    \RightLabel{\(\botL\)}
    \UnaryInfC{\(\bot : \bot, \Gamma \Rightarrow \Delta\)}
    \DisplayProof
    \quad
    \overset{\alpha}{\longmapsto}
    \quad
    \AxiomC{}
    \RightLabel{\(\botL\)}
    \UnaryInfC{\(\bot \Gamma \Rightarrow \Delta, \bot\)}
    \DisplayProof
  \]

  \begin{multline*}
    \AxiomC{\(w0\)}
    \noLine
    \UnaryInfC{\(\kappa_0 : \phi \to \psi, \Gamma \Rightarrow \phi, \Delta\)}
    \AxiomC{\(w1\)}
    \noLine
    \UnaryInfC{\(\kappa_1 : \psi, \phi \to \psi, \Gamma \Rightarrow \Delta\)}
    \RightLabel{\(\toL^{\mathrm{Sat}}\)}
    \BinaryInfC{\(\kappa_0 \vee \kappa_1 : \phi \to \psi, \Gamma \Rightarrow \Delta\)}
    \DisplayProof
    \quad
    \overset{\alpha}{\longmapsto} \\
    \AxiomC{\(\alpha(w0)\)}
    \noLine
    \UnaryInfC{\(\phi \to \psi, \Gamma \Rightarrow \phi, \Delta, \kappa^*_0\)}
    \RightLabel{\(\wk\)}
    \UnaryInfC{\(\phi \to \psi,\Gamma \Rightarrow \phi, \Delta, \kappa^*_0, \kappa^*_1\)}
    \RightLabel{\(\veeR\)}
    \UnaryInfC{\(\phi \to \psi,\Gamma \Rightarrow \phi, \Delta, \kappa^*_0 \vee \kappa^*_1\)}
    \AxiomC{\(\alpha(w1)\)}
    \noLine
    \UnaryInfC{\(\psi, \phi \to \psi, \Gamma \Rightarrow \Delta, \kappa^*_1\)}
    \RightLabel{\(\wk\)}
    \UnaryInfC{\(\psi,\phi \to \psi,\Gamma \Rightarrow \Delta, \kappa^*_0, \kappa^*_1\)}
    \RightLabel{\(\veeR\)}
    \UnaryInfC{\(\psi,\phi \to \psi,\Gamma \Rightarrow  \Delta, \kappa^*_0 \vee \kappa^*_1\)}
    \RightLabel{\(\toL\)}
    \BinaryInfC{\(\phi \to \psi,\phi \to \psi, \Gamma \Rightarrow \Delta, \kappa^*_0 \vee \kappa^*_1\)}
    \RightLabel{\(\ctr\)}
    \UnaryInfC{\(\phi \to \psi, \Gamma \Rightarrow \Delta, \kappa^*_0 \vee \kappa^*_1\)}
    \DisplayProof
  \end{multline*}

  \[
    \AxiomC{\(w0\)}
    \noLine
    \UnaryInfC{\(\kappa : \phi, \Gamma \Rightarrow \psi, \phi \to \psi, \Delta\)}
    \RightLabel{\(\toR^{\mathrm{Sat}}\)}
    \UnaryInfC{\(\kappa : \Gamma \Rightarrow \phi \to \psi, \Delta\)}
    \DisplayProof
    \quad
    \overset{\alpha}{\longmapsto}
    \quad
    \AxiomC{\(\alpha(w0)\)}
    \noLine
    \UnaryInfC{\(\phi, \Gamma \Rightarrow \psi, \phi \to \psi, \Delta, \kappa^*\)}
    \RightLabel{\(\toR\)}
    \UnaryInfC{\(\Gamma \Rightarrow \phi \to \psi, \phi \to \psi, \Delta, \kappa^*\)}
    \RightLabel{\(\ctr\)}
    \UnaryInfC{\(\Gamma \Rightarrow \phi \to \psi, \Delta, \kappa^*\)}
    \DisplayProof
  \]

  \[
    \AxiomC{\(w0\)}
    \noLine
    \UnaryInfC{\(\kappa : \necd_0 \phi, \Gamma \Rightarrow_1 \Delta\)}
    \RightLabel{\(\ERax[1,0]^{\mathrm{Sat}}\)}
    \UnaryInfC{\(\kappa : \nec_0 \phi, \Gamma \Rightarrow_1 \Delta\)}
    \DisplayProof
    \quad
    \overset{\alpha}{\longmapsto}
    \quad
    \AxiomC{\(\alpha(w0)\)}
    \noLine
    \UnaryInfC{\(\necd_0 \phi, \Gamma \Rightarrow_1 \Delta, \kappa^*\)}
    \RightLabel{\(\ERax[1,0]\)}
    \UnaryInfC{\(\nec_0 \phi, \Gamma \Rightarrow_1 \Delta, \kappa^*\)}
    \DisplayProof
  \]

  Finally, let us define \(\alpha(w)\) when \(R = \modal[+]{\KTCS}\).
  Then \(w\) is of shape
  \[
    \AxiomC{\(\left[\begin{matrix}w^{\nec_i} \\ \kappa^{\nec_i} : \necd_i \Sigma^s_i, \nec_{\overline{i}} \Sigma^s_{\overline{i}} \Rightarrow_i\end{matrix}\right]_{i \leq 1}\)}
    \AxiomC{\(\left[\begin{matrix} w^{\nec_i}_\phi \\ \kappa^{\nec_i}_\phi :\necd_i \Sigma^s_i, \nec_{\overline{i}} \Sigma^s_{\overline{i}} \Rightarrow_i\phi\end{matrix}\right]_{i \leq 1, \phi \in \Theta_i}\)}
    \RightLabel{\(\modal[+]{\KTER}\)}
    \BinaryInfC{\( \kappa_w  :\nec_0 \Sigma_0, \nec_1 \Sigma_1, \Gamma \gg \nec_0 \Theta_0, \nec_1 \Theta_1, \Delta\)}
    \DisplayProof
  \]
  where
  \[
    \kappa^*_{w} := \bigwedge_{i \leq 1} \left(\nec_i (\kappa^{\nec_i})^* \wedge \bigwedge\limits_{\phi \in \Theta_i} \pos_i (\kappa^{\nec_i}_\phi)^*\right)\wedge \bigwedge (\Gamma \inter V_+) \wedge \bigwedge \neg(\Delta \inter V_-)
  \]
  Since \(\kappa^*_w\) is a conjunction of formulas it suffices to construct a preproof for each conjunct and join them with \((\wedgeR)\).
  \begin{itemize}
    \item Preproof for \(\nec_i \kappa^{\nec_i}\) for \(i \leq 1\).
      The desired preproof is
      \[
        \AxiomC{\(\alpha(w^{\nec_i})\)}
        \noLine
        \UnaryInfC{\(\necd_i \Sigma^s_i, \nec_{\overline{i}} \Sigma^s_{\overline{i}} \Rightarrow_i (\kappa^{\nec_i})^*\)}
        \RightLabel{\(\modal[i]{\KTCS}\)}
        \UnaryInfC{\(\nec_0 \Sigma_0, \nec_1 \Sigma_1, \Gamma \gg \nec_0 \Theta_0, \nec_1 \Theta_1, \Delta, \nec_i (\kappa^{\nec_i})^*\)}
        \DisplayProof
      \]
    \item Preproof for \(\pos_i \kappa^{\nec_i}_\phi\) for \(i \leq 1\) and \(\phi \in \Theta_i\).
      The desired preproof is
      \[
        \AxiomC{\(\alpha(w^{\nec_i}_\phi)\)}
        \noLine
        \UnaryInfC{\(\necd_i \Sigma^s_i, \nec_{\overline{i}} \Sigma^s_{\overline{i}} \Rightarrow_i \phi, (\kappa^{\nec_i}_\phi)^*\)}
        \RightLabel{\(\negL\)}
        \UnaryInfC{\(\neg (\kappa^{\nec_i}_\phi)^*, \necd_i \Sigma^s_i, \nec_{\overline{i}} \Sigma^s_{\overline{i}} \Rightarrow_i \phi \)}
        \RightLabel{\(\wk\)}
        \UnaryInfC{\(\necd_i \neg (\kappa^{\nec_i}_\phi)^*, \necd_i \Sigma^s_i, \nec_{\overline{i}} \Sigma^s_{\overline{i}} \Rightarrow_i \phi \)}
        \RightLabel{\(\modal[i]{\KTCS}\)}
        \UnaryInfC{\(\nec_i \neg (\kappa^{\nec_i}_\phi)^*, \nec_0 \Sigma_0, \nec_1 \Sigma_1, \Gamma \gg \nec_0 \Theta_0, \nec_1 \Theta_1, \Delta, \nec_i \phi\)}
        \RightLabel{\(\negR\)}
        \UnaryInfC{\( \nec_0 \Sigma_0, \nec_1 \Sigma_1, \Gamma \gg \nec_0 \Theta_0, \nec_1 \Theta_1, \Delta, \nec_i \phi, \pos_i (\kappa^{\nec_i}_\phi)^*\)}
        \DisplayProof
      \]
    \item Preproof for \(p \in \Gamma \inter V_+\) and for \(\neg q\) where \(q \in \Delta \inter V_-\).
      The desired preproof, respectively, are
      \[
        \AxiomC{\(\)}
        \RightLabel{\(\ax\)}
        \UnaryInfC{\(\nec_0 \Sigma_0, \nec_1 \Sigma, \Gamma \gg \nec_0 \Theta_0, \nec_1 \Theta_1, \Delta, p\)}
        \DisplayProof
        \quad
        \AxiomC{\(\)}
        \RightLabel{\(\ax\)}
        \UnaryInfC{\(q, \nec_0 \Sigma_0, \nec_1 \Sigma, \Gamma \gg \nec_0 \Theta_0, \nec_1 \Theta_1, \Delta\)}
        \RightLabel{\(\negR\)}
        \UnaryInfC{\(\nec_0 \Sigma_0, \nec_1 \Sigma, \Gamma \gg \nec_0 \Theta_0, \nec_1 \Theta_1, \Delta, \neg q\)}
        \DisplayProof
      \]
  \end{itemize}

  Finally, let us argue that \(\alpha(w)\) is always a proof and not only a preproof.
  Given a node \(w\) assign to it the measure \(\omega \satc{\Gamma_w \Rightarrow_{i_w} \Delta_w} + \lgth(w)\) where \(\lgth(w)\) is the length of \(w\) as a sequence of natural numbers.
  We notice that this measure always decreases from \(\alpha(w)\) to its corecursive call, except when \(w\) is annotated with the rule \((\modal[+]{\KTER})\).
  However, in this case we will find that progress is made (i.e., there is an application of \((\modal[i]{\KTER})\)) from the root of the preproof fragment given by \(\alpha(w)\) to the corecursive calls.
  This implies that any infinite branch will have infinitely many applications of \((\modal[i]{\KTER})\), as otherwise we will infinitely decrease the measure on corecursive calls.
\end{proof}

\begin{theorem}\label{th:second-verification-ER}
  Let \(T\) be an interpolation template for \(\Gamma \Rightarrow_i \Delta\) and \(\Xi \Rightarrow_i \Lambda\) be a sequent such that \(\voc_b(\Xi \Rightarrow_i \Lambda) \subseteq V_b\) for \(b \in \set{+,-}\).
  Then \(\n{\ER} \vdash \Gamma, \Xi \Rightarrow_i \Delta, \Lambda\) implies \(\n{\ER} \vdash \iota_T, \Xi \Rightarrow_i \Lambda\).
\end{theorem}
\begin{proof}
  Given a node \(w\) of \(T\) let us denote the sequent at \(w\) as \(\Gamma_w \Rightarrow_{i_w} \Delta_w,\), the preinterpolant at \(w\) as \(\kappa_w\) and assume \((\cdot)^*\) is a solution of \(\mathcal{E}_T\).
  We define a function \(\beta\) such that if \(w\) is a node of \(T\) and \(\pi\) is a proof in \(\n{\ER}\) of \(\Gamma_w, \Xi \Rightarrow_{i_w} \Delta_w, \Lambda\) where \(\Xi \Rightarrow_{i_w} \Lambda\) is a sequent with \(\voc_b(\Xi \Rightarrow_{i_w} \Lambda) \subseteq V_b\) for \(b \in \set{+,-}\), then \(\beta(w,\pi)\) is a preproof in \(\n{\ER} + \cut + \wk\) of \(\kappa^*_w, \Xi \Rightarrow_{i_w} \Lambda\).
  Then we will show that \(\beta(w,\pi)\) is in fact a proof in \(\n{\ER} + \cut + \wk\) and then using that \(\wk\) is eliminable in \(\n{\ER}\) and that \(\cut\) is eliminable in \(\n{\ER}\) we obtian the desired proof in \(\n{\ER}\).

  We proceed to define \(\beta\) by corecursion and cases.
  First, assume \(w \in \rep(T)\) then \(w^\circ\) has the same sequent as \(w\) and \(\kappa_w = x_w\).
  Also, by virtue of \((\cdot)^*\) being a solution of \(\mathcal{E}_T\), we have that \(\ER \vdash x^*_w \leftrightarrow \kappa^*_{w^\circ}\).
  Then we define \(\beta\) as
  \[
    \left(
      \AxiomC{}
      \RightLabel{\(\rep\)}
      \UnaryInfC{\(x_w : \Gamma \Rightarrow_i \Delta\)}
      \DisplayProof
      ,\quad
      \AxiomC{\(\pi\)}
      \noLine
      \UnaryInfC{\(\Gamma, \Xi \Rightarrow_i \Delta, \Lambda\)}
      \DisplayProof
    \right)
    \quad
    \overset{\beta}{\longmapsto}
    \quad
    \AxiomC{\(\tau\)}
    \noLine
    \UnaryInfC{\(x^*_w, \Xi \Rightarrow_i \Lambda, \kappa^*_{w^\circ}\)}
    \AxiomC{\(\beta(w^\circ, \pi)\)}
    \noLine
    \UnaryInfC{\(\kappa^*_{w^\circ}, \Xi \Rightarrow_i \Lambda\)}
    \RightLabel{\(\wk\)}
    \UnaryInfC{\(\kappa^*_{w^\circ}, x^*_w, \Xi \Rightarrow_i \Lambda\)}
    \RightLabel{\(\cut\)}
    \BinaryInfC{\(x^*_w, \Xi \Rightarrow_i \Lambda\)}
    \DisplayProof
  \]
  where the proof \(\tau\) exists since \(\ER \vdash x^*_w \leftrightarrow \kappa^*_{w^\circ}\).

  Now assume that \(w\) is not are repetition node and let \(R\) be the rule at \(w\).
  Below there are the definitions when \(R\) is \((\ax)\), \((\emp)\), \((\botL)\), \((\toL^{\mathrm{Sat}})\), \((\toR^{\mathrm{Sat}})\) or \((\ERax[1,0]^{\mathrm{Sat}})\).
  \[
    \left(
      \AxiomC{}
      \RightLabel{\(\ax\)}
      \UnaryInfC{\(\bot : p, \Gamma \Rightarrow_i p, \Delta\)}
      \DisplayProof,
      \AxiomC{\(\pi\)}
      \noLine
      \UnaryInfC{\(p, \Gamma, \Xi \Rightarrow_i p, \Delta, \Lambda\)}
      \DisplayProof
    \right)
    \quad
    \overset{\beta}{\longmapsto}
    \quad
    \AxiomC{\(\)}
    \RightLabel{\(\botL\)}
    \UnaryInfC{\(\bot, \Xi \Rightarrow_i \Lambda\)}
    \DisplayProof
  \]

  \[
    \left(
      \AxiomC{}
      \RightLabel{\(\emp\)}
      \UnaryInfC{\(\top : { \Rightarrow_i }\)}
      \DisplayProof,
      \AxiomC{\(\pi\)}
      \noLine
      \UnaryInfC{\(\Xi \Rightarrow_i \Lambda\)}
      \DisplayProof
    \right)
    \quad
    \overset{\beta}{\longmapsto}
    \quad
    \AxiomC{\(\pi\)}
    \noLine
    \UnaryInfC{\(\Xi \Rightarrow_i \Lambda\)}
    \RightLabel{\(\wk\)}
    \UnaryInfC{\(\top, \Xi \Rightarrow_i \Lambda\)}
    \DisplayProof
  \]

  \[
    \left(
      \AxiomC{}
      \RightLabel{\(\botL\)}
      \UnaryInfC{\(\bot : \bot, \Gamma \Rightarrow_i \Delta\)}
      \DisplayProof,
      \AxiomC{\(\pi\)}
      \noLine
      \UnaryInfC{\(\bot, \Gamma, \Xi \Rightarrow_i \Delta, \Lambda\)}
      \DisplayProof
    \right)
    \quad
    \overset{\beta}{\longmapsto}
    \quad
    \AxiomC{\(\)}
    \RightLabel{\(\botL\)}
    \UnaryInfC{\(\bot, \Xi \Rightarrow_i \Lambda\)}
    \DisplayProof
  \]

  \begin{multline*}
    \left(
      \AxiomC{\(w0\)}
      \noLine
      \UnaryInfC{\(\kappa_0 : \phi \to \psi,\Gamma \Rightarrow_i \phi, \Delta\)}
      \AxiomC{\(w1\)}
      \noLine
      \UnaryInfC{\(\kappa_1 : \psi, \phi \to \psi,\Gamma \Rightarrow_i \Delta\)}
      \RightLabel{\(\toL^{\mathrm{Sat}}\)}
      \BinaryInfC{\(\kappa_0 \vee \kappa_1 :  \phi \to \psi, \Gamma \Rightarrow_i \Delta\)}
      \DisplayProof,
      \AxiomC{\(\pi\)}
      \noLine
      \UnaryInfC{\(\phi \to \psi, \Gamma, \Xi \Rightarrow_i \Delta, \Lambda\)}
      \DisplayProof
    \right)
    \quad
    \overset{\beta}{\longmapsto} \\
    \AxiomC{\(\beta(w0,\wk(\pi))\)}
    \noLine
    \UnaryInfC{\(\kappa^*_0, \Xi \Rightarrow_i \Lambda\)}
    \AxiomC{\(\beta(w1,\wk(\pi))\)}
    \noLine
    \UnaryInfC{\(\kappa^*_1, \Xi \Rightarrow_i \Lambda\)}
    \RightLabel{\(\veeL\)}
    \BinaryInfC{\(\kappa^*_0 \vee \kappa^*_1, \Xi \Rightarrow_i \Lambda\)}
    \DisplayProof
  \end{multline*}

  \[
    \left(
      \AxiomC{\(w0\)}
      \noLine
      \UnaryInfC{\(\kappa_w :  \phi, \Gamma \Rightarrow_i \psi, \phi \to \psi, \Delta\)}
      \RightLabel{\(\toR^{\mathrm{Sat}}\)}
      \UnaryInfC{\(\kappa_w :  \Gamma \Rightarrow_i \phi \to \psi, \Delta\)}
      \DisplayProof,
      \AxiomC{\(\pi\)}
      \noLine
      \UnaryInfC{\(\Gamma, \Xi \Rightarrow_i \phi \to \psi, \Delta, \Lambda\)}
      \DisplayProof
    \right)
    \quad
    \overset{\beta}{\longmapsto}
    \quad
    \AxiomC{\(\beta(w0,\wk(\pi))\)}
    \noLine
    \UnaryInfC{\(\kappa^*_w, \Xi \Rightarrow_i \Lambda\)}
    \RightLabel{\(\wk\)}
    \UnaryInfC{\(\kappa^*_w, \Xi \Rightarrow_i \Lambda\)}
    \DisplayProof
  \]

  \[
    \left(
      \AxiomC{\(w0\)}
      \noLine
      \UnaryInfC{\(\kappa_w : \necd_0 \phi, \Gamma \Rightarrow_i \Delta\)}
      \RightLabel{\(\ERax[1,0]^{\mathrm{Sat}}\)}
      \UnaryInfC{\(\kappa_w : \nec_0 \phi, \Gamma \Rightarrow_i \Delta\)}
      \DisplayProof,
      \AxiomC{\(\pi\)}
      \noLine
      \UnaryInfC{\(\nec_0 \phi, \Gamma, \Xi \Rightarrow_i \Delta, \Lambda\)}
      \DisplayProof
    \right)
    \quad
    \overset{\beta}{\longmapsto}
    \quad
    \AxiomC{\(\beta(w0, \wk(\pi))\)}
    \noLine
    \UnaryInfC{\(\kappa^*_w, \Xi \Rightarrow_i \Lambda\)}
    \RightLabel{\(\wk\)}
    \UnaryInfC{\(\kappa^*_w, \Xi \Rightarrow_i \Lambda\)}
    \DisplayProof
  \]

  Finally, assume \(R\) is \((\modal[+]{\KTCS})\).
  Then \(w\) is of shape
  \[
    \AxiomC{\(\left[\begin{matrix} w^{\nec_i} \\ \kappa^{\nec_i} : \necd_i \Sigma^s_i, \nec_{\overline{i}}\Sigma^s_{\overline{i}}  \Rightarrow_i \end{matrix}\right]_{i \leq 1}\)}
    \AxiomC{\(\left[\begin{matrix} w^{\nec_i}_\phi \\ \kappa^{\nec_i}_\phi :\necd_i \Sigma^s, \nec_{\overline{i}} \Sigma^s_i \Rightarrow_i\phi\end{matrix}\right]_{i \leq 1, \phi \in \Theta_i}\)}
    \RightLabel{\(\modal[+]{\KTER}\)}
    \BinaryInfC{\( \kappa_w  :\nec_0 \Sigma_0, \nec_1 \Sigma_1, \Gamma \gg \nec_0 \Theta_0, \nec_1 \Theta_1, \Delta\)}
    \DisplayProof
  \]
  where
  \[
    \kappa^*_{w} := \bigwedge_{i \leq 1}\left(\nec_i (\kappa^{\nec_i})^* \wedge \bigwedge_{\phi \in \Theta_i} \pos_i (\kappa^{\nec_i}_\phi)^*\right)\wedge \bigwedge (\Gamma \inter V_+) \wedge \bigwedge \neg(\Delta \inter V_-).
  \]
  We define \(\beta(w,\pi)\) by case analysis on the shape of \(\pi\).
  \begin{itemize}
    \item Last rule of \(\pi\) is \((\ax)\).
      There are \(4\) cases depending on where the repeated propositional variable is located.
      \begin{itemize}
        \item Assume \(p \in \Gamma \cap \in \Delta\).
          This case is impossible, as \(w\) is annotated with \((\modal[+]{\KTER})\) so \((\Gamma \cap \Delta) \cap \text{Var} = \varnothing\).
        \item Assume \(p \in \Gamma \cap \Lambda\).
          Since \(p \in \Lambda\) and \(\voc_+(\Xi \gg \Lambda) \subseteq V_+\) we have that \(p \in V_+\).
          Then as \(p \in \Gamma \cap V_+\) and \(p \not\in \set{x_w \mid w \in \rep(T)}\), we obtain that \(p\) is a conjunct of \(\kappa^*_w\).
          The desired preproof is
          \[
            \AxiomC{\(\)}
            \RightLabel{\(\ax\)}
            \UnaryInfC{\(p, \Xi \gg p, \Lambda'\)}
            \doubleLine\RightLabel{\(\wk + \wedgeL\)}
            \UnaryInfC{\(\kappa^*_w, \Xi \gg p, \Lambda'\)}
            \DisplayProof
          \]
          where \(\Lambda = p, \Lambda'\).
        \item Assume \(p \in \Xi \cap \Delta\).
          Since \(p \in \Xi\) and \(\voc_-(\Xi \gg \Lambda) \subseteq V_-\) we have that \(p \in V_-\).
          Then as \(p \in \Delta \cap V_-\) and \(p \not\in \set{x_w \mid w \in \rep(T)}\), we obtain that \(\neg p\) is a conjunct of \(\kappa^*_w\).
          The desired preproof is
          \[
            \AxiomC{\(\)}
            \RightLabel{\(\ax\)}
            \UnaryInfC{\(p, \Xi' \gg p, \Lambda\)}
            \RightLabel{\(\negL\)}
            \UnaryInfC{\(\neg p, p, \Xi' \gg \Lambda\)}
            \doubleLine\RightLabel{\(\wk + \wedgeL\)}
            \UnaryInfC{\(\kappa^*_w, p, \Xi' \gg \Lambda\)}
            \DisplayProof
          \]
          where \(\Xi = p, \Xi'\).
        \item Assume \(p \in \Xi \cap \Lambda\).
          Then the desired preproof is
          \[
            \AxiomC{}
            \RightLabel{\(\ax\)}
            \UnaryInfC{\(\kappa^*_w, \Xi \gg \Lambda\)}
            \DisplayProof
          \]
      \end{itemize}
    \item Last rule of \(\pi\) is \((\botL)\).
      Since \(\bot\) is not in \(\Gamma\), then we have that \(\pi\) has the following shape
      \[
        \AxiomC{\(\)}
        \RightLabel{\(\botL\)}
        \UnaryInfC{\(\nec_0 \Sigma_0, \nec_1 \Sigma_1, \Gamma, \bot, \Xi' \gg \nec_0 \Theta_0, \nec_1 \Theta_1, \Delta, \Lambda\)}
        \DisplayProof
      \]
      where \(\Xi = \bot, \Xi'\).
      Then \(\beta(w,\pi)\) is defined as
      \[
        \AxiomC{\(\)}
        \RightLabel{\(\botL\)}
        \UnaryInfC{\(\kappa^*_w, \bot, \Xi' \gg \Lambda\)}
        \DisplayProof
      \]
    \item Last rule of \(\pi\) is \((\botR)\) and principal formula is in \(\Delta\).
      Then \(\pi\) is of shape
      \[
        \AxiomC{\(\pi_0\)}
        \noLine
        \UnaryInfC{\(\nec_0 \Sigma_0, \nec_1 \Sigma_1, \Gamma, \Xi \Rightarrow \nec_0 \Theta_0, \nec_1 \Theta_1, \Delta', \Lambda\)}
        \RightLabel{\(\botR\)}
        \UnaryInfC{\(\nec_0 \Sigma_0, \nec_1 \Sigma_1, \Gamma, \Xi \Rightarrow \nec_0 \Theta_0, \nec_1 \Theta_1, \bot, \Delta', \Lambda\)}
        \DisplayProof
      \]
      where \(\Delta = \bot, \Delta'\).
      Then \(\beta(w,\pi)\) is defined as
      \[
        \AxiomC{\(\beta(w, \wk(\pi_0))\)}
        \noLine
        \UnaryInfC{\(\kappa^*_w, \Xi \Rightarrow \Lambda\)}
        \RightLabel{\(\wk\)}
        \UnaryInfC{\(\kappa^*_w, \Xi \Rightarrow \Lambda\)}
        \DisplayProof
      \]
    \item Last rule of \(\pi\) is \((\botR)\) and principal formula is in \(\Lambda\).
      Then \(\pi\) is of shape
      \[
        \AxiomC{\(\pi_0\)}
        \noLine
        \UnaryInfC{\(\nec_0 \Sigma_0, \nec_1 \Sigma_1, \Gamma, \Xi \Rightarrow \nec_0 \Theta_0, \nec_1 \Theta_1, \Delta, \Lambda'\)}
        \RightLabel{\(\botR\)}
        \UnaryInfC{\(\nec_0 \Sigma_0, \nec_1 \Sigma_1, \Gamma, \Xi \Rightarrow \nec_0 \Theta_0, \nec_1 \Theta_1, \Delta, \bot, \Lambda'\)}
        \DisplayProof
      \]
      where \(\Lambda = \bot, \Lambda'\).
      Then \(\beta(w,\pi)\) is defined as
      \[
        \AxiomC{\(\beta(w, \pi_0)\)}
        \noLine
        \UnaryInfC{\(\kappa^*_w, \Xi \Rightarrow \Lambda'\)}
        \RightLabel{\(\botR\)}
        \UnaryInfC{\(\kappa^*_w, \Xi \Rightarrow  \bot, \Lambda'\)}
        \DisplayProof
      \]
    \item Last rule of \(\pi\) is \((\toL)\) and principal formula is in \(\Gamma\).
      Then we have that \(\pi\) has the following shape
      \[
        \AxiomC{\(\pi_0\)}
        \noLine
        \UnaryInfC{\(\nec_0 \Sigma_0, \nec_1 \Sigma_1, \Gamma', \Xi \gg \nec_0 \Theta_0, \nec_1 \Theta_1,\phi, \Delta,  \Lambda\)}
        \AxiomC{\(\pi_1\)}
        \noLine
        \UnaryInfC{\(\nec_0 \Sigma_0, \nec_1 \Sigma_1, \psi, \Gamma', \Xi \gg \nec_0 \Theta_0, \nec_1 \Theta_1, \Delta, \Lambda\)}
        \RightLabel{\(\toL\)}
        \BinaryInfC{\(\nec_0 \Sigma_0, \nec_1 \Sigma_1,  \phi \to \psi,\Gamma',  \Xi \gg \nec_0 \Theta_0, \nec_1 \Theta_1, \Delta, \Lambda\)}
        \DisplayProof
      \]
      where \(\Gamma = \phi \to \psi, \Gamma'\).
      By saturation (the conclusion of \(\modal[+]{\KTER}\) is saturated) we have that either \(\phi \in \nec_0 \Theta_0, \nec_1 \Theta_1, \Delta\) or \(\psi \in \nec_0 \Sigma_0, \nec_1 \Sigma_1, \Gamma'\).
      In the first case, \(\beta(w,\pi)\) is defined as
      \[
        \AxiomC{\(\beta(w,\wk(\ctr(\pi_0)))\)}
        \noLine
        \UnaryInfC{\(\kappa^*_w, \Xi \gg \Lambda\)}
        \RightLabel{\(\wk\)}
        \UnaryInfC{\(\kappa^*_w, \Xi \gg \Lambda\)}
        \DisplayProof
      \]
      where in \(\pi_0\) first we contract \(\phi\) on the right and then \(\phi \to \psi\) is reintroduced via weakening.
      In the second case \(\beta(w,\pi)\) is defined analogously using \(\pi_1\) instead of \(\pi_0\).
    \item Last rule of \(\pi\) is \((\toL)\) and principal formula is in \(\Xi\).
      Then we have that \(\pi\) has the following shape
      \[
        \AxiomC{\(\pi_0\)}
        \noLine
        \UnaryInfC{\(\nec_0 \Sigma_0, \nec_1 \Sigma_1, \Gamma, \Xi' \gg \nec_0 \Theta_0, \nec_1 \Theta_1, \Delta, \phi, \Lambda\)}
        \AxiomC{\(\pi_1\)}
        \noLine
        \UnaryInfC{\(\nec_0 \Sigma_0, \nec_1 \Sigma_1, \Gamma,  \psi, \Xi' \gg \nec_0 \Theta_0, \nec_1 \Theta_1, \Delta, \Lambda\)}
        \RightLabel{\(\toL\)}
        \BinaryInfC{\(\nec_0 \Sigma_0, \nec_1 \Sigma_1, \Gamma,  \phi \to \psi, \Xi' \gg \nec_0 \Theta_0, \nec_1 \Theta_1, \Delta, \Lambda\)}
        \DisplayProof
      \]
      where \(\Xi = \phi \to \psi, \Xi'\).
      Then \(\beta(w,\pi)\) is defined as
      \[
        \AxiomC{\(\beta(w,\pi_0)\)}
        \noLine
        \UnaryInfC{\(\kappa^*_w, \Xi' \gg \phi, \Lambda\)}
        \AxiomC{\(\beta(w,\pi_1)\)}
        \noLine
        \UnaryInfC{\(\kappa^*_w, \psi, \Xi' \gg \Lambda\)}
        \RightLabel{\(\toL\)}
        \BinaryInfC{\(\kappa^*_w, \phi \to \psi, \Xi' \gg \Lambda\)}
        \DisplayProof
      \]
    \item Last rule of \(\pi\) is \((\toR)\) and principal formula is in \(\Delta\).
      Then we have that \(\pi\) has the following shape
      \[
        \AxiomC{\(\pi_0\)}
        \noLine
        \UnaryInfC{\(\nec_0 \Sigma_0, \nec_1 \Sigma_1,  \phi, \Gamma, \Xi \gg \nec_0 \Theta_0, \nec_1 \Theta_1, \psi, \Delta', \Lambda\)}
        \RightLabel{\(\toR\)}
        \UnaryInfC{\(\nec_0 \Sigma_0, \nec_1 \Sigma_1, \Gamma, \Xi \gg \nec_0 \Theta_0, \nec_1 \Theta_1, \phi \to \psi, \Delta', \Lambda\)}
        \DisplayProof
      \]
      where \(\Delta = \phi \to \psi, \Delta'\).
      By saturation (the conclusion of \(\modal[+]{\KTER}\) is saturated) we have that \(\phi \in \nec_0 \Sigma_0, \nec_1 \Sigma_1, \Gamma\) and \(\psi \in \nec_0 \Theta_0, \nec_1 \Theta_1, \Delta'\).
      Then \(\beta(w,\pi)\) is defined as
      \[
        \AxiomC{\(\beta(w,\wk(\ctr(\pi_0)))\)}
        \noLine
        \UnaryInfC{\(\kappa^*_w, \Xi \gg \Lambda\)}
        \RightLabel{\(\wk\)}
        \UnaryInfC{\(\kappa^*_w, \Xi \gg \Lambda\)}
        \DisplayProof
      \]
      where in \(\pi_0\) first we contract \(\phi\) on the left and \(\psi\) on the right and then \(\phi \to \psi\) is reintroduced via weakening.
    \item Last rule of \(\pi\) is \((\toR)\) and principal formula is in \(\Lambda\).
      Then we have that \(\pi\) has the following shape
      \[
        \AxiomC{\(\pi_0\)}
        \noLine
        \UnaryInfC{\(\nec_0 \Sigma_0, \nec_1 \Sigma_1, \Gamma, \phi, \Xi \gg \nec_0 \Theta_0, \nec_1 \Theta_1, \Delta, \psi, \Lambda'\)}
        \RightLabel{\(\toR\)}
        \UnaryInfC{\(\nec_0 \Sigma_0, \nec_1 \Sigma_1, \Gamma, \Xi \gg \nec_0 \Theta_0, \nec_1 \Theta_1, \Delta, \phi \to \psi, \Lambda'\)}
        \DisplayProof
      \]
      where \(\Lambda = \phi \to \psi, \Lambda'\).
      Then \(\beta(w,\pi)\) is defined as
      \[
        \AxiomC{\(\beta(w,\pi_0)\)}
        \noLine
        \UnaryInfC{\(\kappa^*_w, \phi, \Xi \gg \psi, \Lambda'\)}
        \RightLabel{\(\toL\)}
        \UnaryInfC{\(\kappa^*_w, \Xi \gg \phi \to \psi, \Lambda\)}
        \DisplayProof
      \]
    \item Last rule of \(\pi\) is \((\ERax[1,0])\) and the principal formula \(\nec_0 \phi\) is in \(\nec_0 \Sigma_0\).
      Note that in this case \({\gg} = { \Rightarrow_ 1}\).
      Then \(\pi\) has the following shape
      \[
        \AxiomC{\(\pi_0\)}
        \noLine
        \UnaryInfC{\(\necd_0 \phi, \nec_0 \Sigma'_0, \nec_1 \Sigma_1, \Gamma, \Xi \Rightarrow_1 \nec_0 \Theta_0, \nec_1 \Theta_1, \Delta, \Lambda\)}
        \RightLabel{\(\ERax[1,0]\)}
        \UnaryInfC{\(\nec_0 \phi, \nec_0 \Sigma'_0, \nec_1 \Sigma_1, \Gamma, \Xi \Rightarrow_1 \nec_0 \Theta_0, \nec_1 \Theta_1, \Delta, \Lambda\)}
        \DisplayProof
      \]
      where \(\Sigma_0 = \phi, \Sigma'_0\).
      By saturation (the conclusion of \((\modal[+]{\KTER})\) is saturated) we have that \(\phi \in \nec_0 \Sigma'_0, \nec_1 \Sigma_1, \Gamma\).
      Then \(\beta(w,\pi)\) is defined as
      \[
        \AxiomC{\(\beta(w,\ctr(\pi_0))\)}
        \UnaryInfC{\(\kappa^*_w, \Xi \Rightarrow \Lambda\)}
        \RightLabel{\(\wk\)}
        \UnaryInfC{\(\kappa^*_w, \Xi \Rightarrow \Lambda\)}
        \DisplayProof
      \]
      where in \(\pi_0\) first we contract \(\phi\) on the left.
    \item Last rule of \(\pi\) is \((\ERax[1,0])\) and the principal formula \(\nec_0 \phi\) is in \(\Xi\).
      Note that in this case \({\gg} = { \Rightarrow_ 1}\).
      Then \(\pi\) has the following shape
      \[
        \AxiomC{\(\pi_0\)}
        \noLine
        \UnaryInfC{\( \nec_0 \Sigma_0, \nec_1 \Sigma_1, \Gamma, \necd_0 \phi,\Xi' \Rightarrow_1 \nec_0 \Theta_0, \nec_1 \Theta_1, \Delta, \Lambda\)}
        \RightLabel{\(\ERax[1,0]\)}
        \UnaryInfC{\( \nec_0 \Sigma_0, \nec_1 \Sigma_1, \Gamma, \nec_0 \phi,\Xi' \Rightarrow_1 \nec_0 \Theta_0, \nec_1 \Theta_1, \Delta, \Lambda\)}
        \DisplayProof
      \]
      where \(\Xi = \nec_0 \phi, \Xi'\).
      Then \(\beta(w,\pi)\) is defined as
      \[
        \AxiomC{\(\beta(w,\pi_0)\)}
        \noLine
        \UnaryInfC{\(\kappa^*_w, \necd_0 \phi, \Xi' \Rightarrow_1 \Lambda\)}
        \RightLabel{\(\ERax[1,0]\)}
        \UnaryInfC{\(\kappa^*_w, \nec_0 \phi, \Xi' \Rightarrow_1 \Lambda\)}
        \DisplayProof
      \]
    \item Last rule of \(\pi\) is \((\modal[i]{\KTER})\) and the principal formula \(\nec_i \phi\) is in \(\nec_i \Theta_i\).
      Then \(\pi\) has the following shape
      \[
        \AxiomC{\(\pi_0\)}
        \noLine
        \UnaryInfC{\(\necd_i \Sigma'_i, \nec_{\overline{i}} \Sigma'_{\overline{i}}, \necd_i \Xi_{\nec_{i}}, \nec_i \Xi_{\nec_{\overline{i}}} \Rightarrow_i \phi\)}
        \RightLabel{\(\modal[i]{\KTER}\)}
        \UnaryInfC{\(\nec_0 \Sigma_0, \nec_1 \Sigma_1, \Gamma, \Xi \gg \nec_0 \Theta_0, \nec_1 \Theta_1, \Delta, \Lambda\)}
        \DisplayProof
      \]
      where \(\Sigma'_i \subseteq \Sigma_i\), \(\Sigma'_{\overline{i}} \subseteq \Sigma_{\overline{i}}\) and \(\nec_i \Xi_{\nec_i}, \nec_{\overline{i}} \Xi_{\nec_{\overline{i}}} \subseteq \Xi\).
      Then the desired preproof is
      \[
        \AxiomC{\(\beta(w^{\nec_i}_\phi, \pi_0)\)}
        \noLine
        \UnaryInfC{\((\kappa^{\nec_i}_\phi)^*, \necd_i \Xi_{\nec_{i}}, \nec_i \Xi_{\nec_{\overline{i}}} \Rightarrow_i\)}
        \RightLabel{\(\negR\)}
        \UnaryInfC{\(\necd_i \Xi_{\nec_{i}}, \nec_i \Xi_{\nec_{\overline{i}}} \Rightarrow_i \neg (\kappa^{\nec_i}_\phi)^* \)}
        \RightLabel{\(\modal[i]{\KTER}\)}
        \UnaryInfC{\(\Xi \gg \nec_i \neg (\kappa^{\nec_i}_\phi)^*, \Lambda\)}
        \RightLabel{\(\negL\)}
        \UnaryInfC{\(\pos_i (\kappa^{\nec_i}_\phi)^*, \Xi \gg  \Lambda\)}
        \doubleLine
        \RightLabel{\(\wk + \wedgeL\)}
        \UnaryInfC{\(\kappa^*_w, \Xi \gg  \Lambda\)}
        \DisplayProof
      \]
    \item Last rule of \(\pi\) is \((\modal[i]{\KTER})\) and the principal formula is in \(\Lambda\).
      Then \(\pi\) has the following shape
      \[
        \AxiomC{\(\pi_0\)}
        \noLine
        \UnaryInfC{\(\necd_i \Sigma'_i, \nec_{\overline{i}} \Sigma'_{\overline{i}}, \necd_i \Xi_{\nec_{i}}, \nec_i \Xi_{\nec_{\overline{i}}} \Rightarrow_i \phi\)}
        \RightLabel{\(\modal[i]{\KTER}\)}
        \UnaryInfC{\(\nec_0 \Sigma_0, \nec_1 \Sigma_1, \Gamma, \Xi \gg \nec_0 \Theta_0, \nec_1 \Theta_1, \Delta, \nec_i\phi, \Lambda'\)}
        \DisplayProof
      \]
      where \(\Lambda = \nec_i \phi, \Lambda'\), \(\Sigma'_i \subseteq \Sigma_i\), \(\Sigma'_{\overline{i}} \subseteq \Sigma_{\overline{i}}\) and \(\nec_i \Xi_{\nec_i}, \nec_{\overline{i}} \Xi_{\nec_{\overline{i}}} \subseteq \Xi\).
      Then the desired preproof is
      \[
        \AxiomC{\(\beta(w^{\nec_i}, \pi_0)\)}
        \noLine
        \UnaryInfC{\((\kappa^{\nec_i})^*, \necd_i \Xi_{\nec_{i}}, \nec_i \Xi_{\nec_{\overline{i}}} \Rightarrow_i \phi\)}
        \RightLabel{\(\wk\)}
        \UnaryInfC{\(\necd_i (\kappa^{\nec_i})^*, \necd_i \Xi_{\nec_{i}}, \nec_i \Xi_{\nec_{\overline{i}}} \Rightarrow_i \phi\)}
        \RightLabel{\(\modal[i]{\KTER}\)}
        \UnaryInfC{\(\nec_i (\kappa^{\nec_i})^*, \Xi \gg \nec_i \phi, \Lambda'\)}
        \doubleLine
        \RightLabel{\(\wk + \wedgeL\)}
        \UnaryInfC{\(\kappa^*_w, \Xi \gg \nec_i \phi, \Lambda'\)}
        \DisplayProof
      \]
  \end{itemize}

  Finally, let us argue that \(\beta(w, \pi)\) is always a proof and not only a preproof.
  Given a node \(w\) assign to it the measure \(\omega^2 \satc{\Gamma_w \Rightarrow \Delta_w} + \omega\lgth(w) + \lhg(\pi)\) where \(\lgth(w)\) is the length of \(w\) as a sequence of natural numbers.
  We notice that this measure always decreases from \(\beta(w,\pi)\) to its corecursive calls, except when \(w\) is annotated with the rule \((\modal[+]{\KTER})\) and the last rule of \(\pi\) is \(\modal[i]{\KTER}\).
  However, in this case we will find that progress is made (i.e., there is an application of \((\modal[i]{\KTER})\)) from the root of the preproof fragment given by \(\beta(w,\pi)\) to the corecursive calls.
  This implies that any infinite branch will have infinitely many applications of \((\modal[i]{\KTER})\), as otherwise we will infinitely decrease the measure on corecursive calls.
\end{proof}

Then following theorem follows the same proof as Theorem~\ref{th:ulip-CS} using Theorems~\ref{th:first-verification-ER} and \ref{th:second-verification-ER} instead of Theorems~\ref{th:first-verification-CS} and \ref{th:second-verification-CS}.

\begin{theorem}\label{th:ulip-ER}
  \(\ER\) has uniform Lyndon interpolation.
\end{theorem}

\section{Conclusion and Future Work}

We provided simple cut-free sequent calculi for the bimodal provability logics \(\CS\), \(\CSM\) and \(\ER\).
An interesting technical contribution is that we only had to consider multiple kinds of sequents to deal with  \(\ER\), but no other mechanism was needed. 
Moreover,  we introduced non-wellfounded calcluli for these logics to establish cut-elimination and uniform Lyndon interpolation for all three logics. 
These are the first results of this kind for bimodal provability logics with 'usual' provability predicates.

For future work, we would like to continue exploring the realm of bimodal provabiltiy logics by proof-theoretic means.
Designing nice sequent calculi for these logics could also provide alternative proofs for facts that are usually established via Kripke models.
This can be of particular interest for logics like \(\ER\) that do not have a (non-general) Kripke semantics.
Additionally, we leave open the question of finding an interpretation of \(\ER\) in \(\CSM\) that preserves polarity of variables.
We note that the interpretation
\[
  \phi \mapsto \left(\bigwedge \necd_0\set{\nec_1(\nec_0 \psi \to \psi) \mid \nec_0 \psi \in \sub(\phi)} \to \phi\right),
\]
which can be proven to be correct using the subformula property of our sequent calculi (or through the semantics), does not preserve the polarity of variables.

\section*{Acknowledgements}

We would like to thank Justus Becker for our discussions on the proof theory of provability logic.
In addition, the questions he raised in \cite{justus} served as a great inspiration and motivation for this paper.
Also, we would like to thank Anupam Das for calling our attention to Beki\'c Theorem (see footnote at page~\pageref{bekic-footnote}).

Borja Sierra Miranda and Thomas Studer are supported by the Swiss National Science Foundation project 200021\_214820 \emph{Non-wellfounded and cyclic proof theory}.

\appendix
\section{Solving Equational Systems}
\label{sec:solving-equational-systems}

In Subsection~\ref{subsec:interpolation} we claimed that it was enough to show that a logic \(L\) has basic Lyndon fixpoints to obtain that \(L\) has positive modalized Lyndon fixpoints and every positive modalized Lyndon equational system has a solution in \(L\).
Here we give the details of these results.

We start with a lemma.

\begin{lemma}\label{simple-equation-systems-solvable}
  Let \(L\) be a bimodal normal logic with basic Lyndon fixpoints and closed under substitutions.
  Every basic Lyndon equational system has a solution in \(L\).
\end{lemma}
\begin{proof}
  Let \(\mathcal{E} = \set{(p_i,b_i,\nec_{j_i} \phi_i) \mid i < n}\) be a basic Lyndon equational system over \((\bar{p}, V_+, V_-)\), remember that \(B_+ = \set{p_i \mid i < n, b_i = {+}}\), \(B_- = \set{p_i \mid i < n, b_i = {-}}\).
  We proceed by induction on \(n\), the number of unknowns.

  Take \(\nec_{j_0}\phi_0(p_0,\ldots,p_{n-1})\), we know that it has a Lyndon fixpoint \(\psi_0\) with respect to \(p_0\) in \(L\).
  Note that if \(b_0 = b_i\) then \(p_0 \not \in \voc_-(\phi_i)\) and if \(b_0 \neq b_i\) then \(p_0 \not \in \voc_+(\phi_i)\).
  Then, using Lemma~\ref{substitution-and-polarity}, we have that 
  \begin{align*}
    &\voc_{b_i}(\nec_{j_i} \phi_i[\psi_0/p_0]) \subseteq \voc_{b_i}(\phi_i) \setminus \set{p_0} \union \voc_{b_i}(\psi_0) = \voc_{b_i}(\phi_i) \setminus \set{p_0} \union \voc_{b_0}(\psi_0) \subseteq V_+ \union B_+ \setminus \set{p_0},\\
    &\hspace{14cm}\text{if }b_0 = b_i, \\
    &\voc_{b_i}(\nec_{j_i} \phi_i[\psi_0/p_0]) \subseteq \voc_{b_i}(\phi_i) \setminus \set{p_0} \union \voc_{\overline{b_i}}(\psi_0) = \voc_{b_i}(\phi_i) \setminus \set{p_0} \union \voc_{b_0}(\psi_0) \subseteq V_+ \union B_+ \setminus \set{p_0},\\
    &\hspace{14cm}\text{if }b_0 \neq b_i.
  \end{align*}
  We have a similar argument to show that \(\voc_{\overline{b_i}}(\nec_{j_i} \phi_i[\psi_0/p_0]) \subseteq V_- \union B_- \setminus \set{p_0}\).
  So we have that \(\mathcal{E}' = \set{(p_i, b_i, \nec_{j_i} \phi_i[\psi_0/p_0]) \mid 1 \leq i < n}\) is a basic Lyndon equational system over \((p_1 \cdots p_{n-1}, V_+,V_-)\) and it is solvable in \(L\) by the induction hypothesis, let \((\chi_1,\ldots,\chi_n)\) be a solution in \(L\) of it.
  Let us define \(\chi_0 = \psi_0[\chi_1/p_1,\ldots,\chi_n/p_n]\), then we claim that \((\chi_0,\ldots,\chi_n)\) is a solution of \(\mathcal{E}\).

  First, note that we already have that \(\voc_{b_i}(\chi_i) \subseteq V_+\) and \(\voc_{\overline{b_i}}(\chi_i) \subseteq V_-\) for \(1 \leq i < n\).
  We notice that for \(1 \leq i < n\) we have that if \(b_i = b_0\) then \(p_i \not\in \voc_-(\psi_0)\) and if \(b_i \neq b_0\) then \(p_i \not \in \voc_+(\psi_0)\) (it suffices to do cases on \(b_i\) and \(b_0\)).
  Using Lemma~\ref{substitution-and-polarity} we have that
  \begin{align*}
    \voc_{b_0}(\chi_0) &\subseteq \voc_{b_0}(\psi_0) \setminus \set{p_1,\ldots,p_n} \union \Union_{\footnotesize \begin{matrix} 1 \leq i < n \\ b_i = b_0 \end{matrix}} \voc_{b_0}(\chi_i) \union \Union_{\footnotesize \begin{matrix} 1 \leq i < n \\ b_i \neq b_0 \end{matrix}} \voc_{\overline{b_0}}(\chi_i) \\
                       &\subseteq \voc_{b_0}(\nec_{j_0} \phi_0) \setminus \set{p_0,\ldots,p_n} \union \Union_{\footnotesize \begin{matrix} 1 \leq i < n \\ b_i = b_0 \end{matrix}} \voc_{b_i}(\chi_i) \union \Union_{\footnotesize \begin{matrix} 1 \leq i < n \\ b_i \neq b_0 \end{matrix}} \voc_{b_i}(\chi_i) \subseteq V_+.
  \end{align*}
  where we used that \(\voc_{b_0}(\nec_{j_0} \phi_0) \subseteq V_+ \union B_+\).
  We have an anologous reasoning for showing that \(\voc_{\overline{b_0}}(\chi_0) \subseteq V_-\), so the desired polarity conditions hold.
  
  Finally, we show the desired equivalences.
  We have \(L \vdash \chi_i \leftrightarrow (\nec_{j_i}\phi_i[\psi_0/p_0])[\chi_1/p_1, \ldots, \chi_{n}/p_n]\) for \(1 \leq i < n\).
  Using that \(p_0 \neq p_i\) for \(1 \leq i < n\) we obtain that 
  \begin{align*}
    (\phi_i[\psi_0/p_0])[\chi_1/p_1, \ldots, \chi_{n}/p_n] 
    &= \phi_i[\psi_0[\chi_1/p_1, \ldots, \chi_{n}/p_n]/p_0, \chi_1/p_1, \ldots, \chi_{n}/p_n]\\
    &= \phi_i[\chi_0/p_0,\chi_1/p_1, \ldots, \chi_{n}/p_n],
  \end{align*}
  as desired.
  All left to show that \(L \vdash \chi_0 \leftrightarrow \nec_{j_0}\phi_0[\chi_0/p_0, \ldots, \chi_{n}/p_n]\).
  We have that, as \(\psi_0\) is a fixpoint, \(L \vdash \psi_0 \leftrightarrow \nec_{j_0}\phi_0[\psi_0/p_0]\).
  As \(L\) is closed under substitutions we obtain that 
  \[L \vdash \chi_0 \leftrightarrow (\nec_{j_0}\phi_0[\psi_0/p_0])[\chi_1/p_1,\ldots,\chi_n/p_n],\]
  and we can use the same equalities as before since \(p_0 \neq p_i\) for \(1 \leq i < n\).
\end{proof}

\begin{theorem}
  Let \(L\) be a bimodal normal logic with basic Lyndon fixpoints and closed under substitutions.
  We have that
  \begin{enumerate}
    \item \(L\) has positive modalized Lyndon fixpoins.
    \item Positive modalized Lyndon equational systems have solution in \(L\).
  \end{enumerate}
\end{theorem}
\begin{proof}
  Proof of 1.
  \footnote{This proof follows the proof in \cite{lindstrom} for \(\GL\), with the addition of the condition on the polarity of variables.}
  Let \(\phi(p)\) be a formula that is modalized in \(p\) and positive in \(p\).
  Then \(\phi(p)\) can be written as \( \phi'(\nec_{j_0}\psi_0(p),\ldots,\nec_{j_{n-1}}\psi_{n-1}(p),\nec_{k_0}\chi_0(p),\ldots,\nec_{k_{m-1}}\chi_{m-1}(p))\),
  where the variables in \(\bar{q}\bar{r}\) are pairwise distinct, \(\phi'\) does not contain any modality, and \(\voc_+(\phi') \subseteq \voc_+(\phi) \setminus \set{p} \union \bar{q}\), \(\voc_-(\phi') \subseteq \voc_-(\phi) \setminus \set{p} \union \bar{r}\) (so \(p\) does not occur in \(\phi'\)).
  In particular this means that \(\bar{q} \cap \voc_-(\phi') = \varnothing\) and \(\bar{r} \cap \voc_+(\phi') = \varnothing\).
  By Lemma~\ref{substitution-and-polarity} we have that for \(b \in \set{+,-}\)
    \[
      \voc_b(\phi) = \voc_{b}(\phi') \setminus \bar{q} \bar{r} \union \Union_{i < n} \voc_b(\psi_i) \union \Union_{i < m} \voc_{\overline{b}}(\chi_i)
      \quad \text{so}
      \quad \voc_b(\psi_i) \subseteq \voc_b(\phi) \text{ and } \voc_b(\chi_i) \subseteq \voc_{\overline{b}}(\phi).
    \]
    As \(p \not \in \voc_-(\phi)\) we obtain straightforwardly that \(p \not \in \voc_-(\psi_i)\quad \text{and}\quad p \not \in \voc_+(\chi_i)\).

  Consider the following equational system
  \( \set{(q_i, +,\nec_{j_i} \psi_i(\phi')) \mid i < n} \union \set{(r_i,-, \nec\chi_{k_i}(\phi')) \mid i < m}.
  \)
  Using Lemma~\ref{substitution-and-polarity} with the previous observations, it is straightforward to check that is a simple Lyndon \((\bar{q}\bar{r}, \voc_+(\phi)\setminus \set{p}, \voc_-(\phi)\setminus \set{p})\)-equational system.
  Thus it has a solution \((\cdot)^*\) in \(L\).
  Let us define \(\eta = \phi'(q^*_0, \ldots, q^*_{n-1}, r^*_0, \ldots,r^*_{m-1})\), by Lemma~\ref{substitution-and-polarity} we know that \(\voc_+(\eta) \subseteq \voc_+(\phi) \setminus \set{p}\) and \(\voc_-(\eta) \subseteq \voc_-(\phi) \setminus \set{p}\).
  Also, since \((\cdot)^*\) is a solution of the equational system, we have that \(L \vdash q^*_i \leftrightarrow \nec_{j_i} \psi_i(\eta)\) for \(i < n\) and \(L \vdash r^*_i \leftrightarrow \nec_{k_i} \chi_i(\eta)\) for \(i < m\).  So we obtain (using propositional reasoning) that
  \[L \vdash \eta \leftrightarrow \phi'(\nec_{j_0}\psi_0(\eta), \ldots,\nec_{j_{n-1}}\psi_{n-1}(\eta), \nec_{k_0}\chi_0(\eta), \ldots,\nec_{k_{m-1}}\chi_{m-1}(\eta)). \]
  In other words, \(L \vdash \eta \leftrightarrow \phi(\eta)\), so \(\eta\) is a Lyndon fixpoint of \(\phi\) with respect to \(p\).

  Proof of 2. The proof is similar to the second point of Lemma~\ref{simple-equation-systems-solvable} using that if \(\phi_i\) is modalised in \(p_0,\ldots,p_i\) then \(\phi_i[\psi/p_0]\) is also modalised in \(p_0,\ldots,p_i\).
\end{proof}

\section{Cut reductions}
\label{sec:cut-reductions}

We display some cut reduction needed for the Theorems~\ref{th:cut-elim-CS}, \ref{th:cut-elim-CSM} and \ref{th:cut-elim-ER}.
Remember that \(\gg\) can be any of \( \Rightarrow \), \( \Rightarrow_0 \) and \( \Rightarrow_1 \).

\textbf{Weakening formulas.}
If \(\chi\) belongs to the weakening formulas of the last rule instance of \(\pi\) or \(\tau\) we can delete directly, as weakening formulas can be modified arbitrarily.
From now on we assume that \(\chi\) is not a weakening formula of the rule instances.

\textbf{Axiomatic}.
Assume \(\pi\) ends in \((\ax)\), the case for \(\tau\) is analogous.
Then \(\chi = p\) for some variable \(p\) and the desired reduction is
\[
  \AxiomC{\(\)}
  \RightLabel{\(\ax\)}
  \UnaryInfC{\(p, \Gamma' \gg \Delta, p\)}
  \DisplayProof
  \quad
  \AxiomC{\(\tau\)}
  \noLine
  \UnaryInfC{\(p, p, \Gamma' \gg \Delta\)}
  \DisplayProof
  \longmapsto
  \AxiomC{\(\ctr(\tau)\)}
  \noLine
  \UnaryInfC{\(p, \Gamma' \gg \Delta\)}
  \DisplayProof
\]
where \(\Gamma = p, \Gamma'\).

If \(\pi\) ends in \(\botL\) then \(\chi\) would belong to the weakening formulas, which is already covered.
Assume \(\tau\) ends in \(\botL\), so \(\chi = \bot\).
\[
  \AxiomC{\(\pi\)}
  \noLine
  \UnaryInfC{\(\Gamma \gg \Delta, \bot\)}
  \DisplayProof
  \quad
  \AxiomC{\(\tau\)}
  \noLine
  \UnaryInfC{\(\bot, \Gamma \gg \Delta\)}
  \DisplayProof
  \longmapsto
  \AxiomC{\(\inv{\botR}(\pi)\)}
  \noLine
  \UnaryInfC{\(\Gamma \gg \Delta\)}
  \DisplayProof
\]

From now own we assume that neither \(\pi\) nor \(\tau\) end in \((\ax)\) or \((\botL)\).

\textbf{\(\botR\) case}.
Asssume \(\pi\) ends in an application of \((\botR)\), the case for \(\tau\) is analogous.
If \(\bot\) is the cut formula, the desired cut reduction is obtained taking the immediate subproof of \(\pi\).
%\[
%  \AxiomC{\(\pi_0\)}
%  \noLine
%  \UnaryInfC{\(\Gamma \gg \Delta\)}
%  \RightLabel{\(\botR\)}
%  \UnaryInfC{\(\Gamma \gg \Delta, \bot\)}
%  \DisplayProof \quad
%  \AxiomC{\(\tau\)}
%  \noLine
%  \UnaryInfC{\(\bot, \Gamma \gg \Delta\)}
%  \DisplayProof
%  \longmapsto
%  \AxiomC{\(\pi_0\)}
%  \noLine
%  \UnaryInfC{\(\Gamma \gg \Delta\)}
%  \DisplayProof
%\]
If \(\bot\) is not the cut formula, the desired cut reduction is
\[
  \AxiomC{\(\pi_0\)}
  \noLine
  \UnaryInfC{\(\Gamma \gg \Delta', \chi\)}
  \RightLabel{\(\botR\)}
  \UnaryInfC{\(\Gamma \gg \bot, \Delta', \chi\)}
  \DisplayProof \quad
  \AxiomC{\(\tau\)}
  \noLine
  \UnaryInfC{\(\chi, \Gamma \gg \bot, \Delta'\)}
  \DisplayProof
  \longmapsto
  \AxiomC{\(\pi_0\)}
  \noLine
  \UnaryInfC{\(\Gamma \gg \Delta', \chi\)}
  \AxiomC{\(\inv{\botR}(\tau)\)}
  \noLine
  \UnaryInfC{\(\chi, \Gamma \gg \Delta'\)}
  \RightLabel{\(\cut \text{(I.H.)}\)}
  \BinaryInfC{\(\Gamma \gg \Delta'\)}
  \DisplayProof
\]
where \(\Delta = \bot, \Delta'\).

From now own we assume that neither \(\pi\) nor \(\tau\) end in \((\botR)\).

\textbf{Principal cut reduction.}
Assume \(\chi\) is principal in \(\pi\) and \(\tau\).
Then either \(\chi\) is an implication or a \(\nec\)-formula.
The second case cannot occur in \(\n{\CS}\) or \(\n{\CSM}\) and is covered for \(\n{\ER}\) at the proof of Theorem~\ref{th:cut-elim-ER}, so we just show the cut reduction for the first.
\begin{multline*}
  \AxiomC{\(\pi_0\)}
  \noLine
  \UnaryInfC{\(\chi_0, \Gamma \gg \Delta, \chi_1\)}
  \RightLabel{\(\toR\)}
  \UnaryInfC{\(\Gamma \gg \Delta, \chi_0 \to \chi_1\)}
  \DisplayProof
  \quad
  \AxiomC{\(\tau_0\)}
  \noLine
  \UnaryInfC{\(\Gamma \gg \Delta, \chi_0\)}
  \AxiomC{\(\tau_1\)}
  \noLine
  \UnaryInfC{\(\chi_1, \Gamma \gg \Delta\)}
  \RightLabel{\(\toL\)}
  \BinaryInfC{\(\chi_0 \to \chi_1, \Gamma \gg \Delta\)}
  \DisplayProof
  \longmapsto  \\ 
  \\
  \AxiomC{\(\wk(\tau_0)\)}
  \noLine
  \UnaryInfC{\(\Gamma \gg \Delta, \chi_1, \chi_0\)}
  \AxiomC{\(\pi_0\)}
  \noLine
  \UnaryInfC{\(\chi_0, \Gamma \gg \Delta, \chi_1\)}
  \RightLabel{\(\cut \text{(I.H.)}\)}
  \BinaryInfC{\(\Gamma \gg \Delta, \chi_1\)}
  \AxiomC{\(\tau_1\)}
  \noLine
  \UnaryInfC{\(\chi_1, \Gamma \gg \Delta\)}
  \RightLabel{\(\cut \text{(I.H.)}\)}
  \BinaryInfC{\(\Gamma \gg \Delta\)}
  \DisplayProof
\end{multline*}

\textbf{Commutative cut reduction.}
Finally, assume that the cut formula is not principal in either \(\pi\) or \(\tau\).
Assume that the cut formula is not principal in \(\pi\) and the last rule of \(\pi\) is \((\toR)\), the cases where the rule is instead \((\toL)\) or \((\ERax[1,0])\) or where any of these occurs at \(\tau\) are analogous.
The desired cut reduction is

\[
  \AxiomC{\(\pi_0\)}
  \noLine
  \UnaryInfC{\(\phi, \Gamma \Rightarrow \psi, \Delta', \chi\)}
  \RightLabel{\(\toR\)}
  \UnaryInfC{\(\Gamma \Rightarrow \phi \to \psi, \Delta', \chi\)}
  \DisplayProof
  \quad
  \AxiomC{\(\tau\)}
  \noLine
  \UnaryInfC{\(\chi, \Gamma \Rightarrow \phi \to \psi, \Delta'\)}
  \DisplayProof
  \longmapsto
  \AxiomC{\(\pi_0\)}
  \noLine
  \UnaryInfC{\(\phi, \Gamma \Rightarrow \psi, \Delta', \chi\)}
  \AxiomC{\(\inv{\toR}(\tau)\)}
  \noLine
  \UnaryInfC{\(\chi, \phi, \Gamma \Rightarrow \psi, \Delta'\)}
  \RightLabel{\(\cut\text{ (I.H.)}\)}
  \BinaryInfC{\(\phi, \Gamma \Rightarrow \psi, \Delta'\)}
  \RightLabel{\(\toR\)}
  \UnaryInfC{\( \Gamma \Rightarrow \phi \to \psi, \Delta'\)}
  \DisplayProof
\]
where \(\Delta = \phi \to \psi, \Delta'\)

Finally, we have that the cut formula is not principal in either \(\pi\) or \(\tau\) with last rule \((\modal[i]{\KTCS})\), \((\modal[i]{\KTCSM})\) or \((\modal[i]{\KTER})\).
We notice that this cannot occur in \(\pi\), as then the cut formula would belong to the weakening formulas, so it must occur in \(\tau\).
To make the cut formula not belong to the weakening formulas it must be the case that \(\chi = \nec_i \chi_0\).
If \(\chi\) where not principal in \(\pi\), then we would be in one of the cases which we already covered, so we can assume it is prinicipal in \(\pi\), so \(\pi\) must end in \((\modal[i]{\KTCS})\), \((\modal[i]{\KTCSM})\) or \((\modal[i]{\KTER})\).
Each of these cases are covered at the proof of Theorems~\ref{th:cut-elim-CS}, \ref{th:cut-elim-CSM} and \ref{th:cut-elim-ER}.

\bibliography{bib}
\bibliographystyle{plain}
\end{document}